\documentclass{amsart}
\usepackage[nohug]{ptdiagrams}
\usepackage{amscd}
\usepackage[mathscr]{eucal}
\usepackage{amsfonts}
\usepackage{amsmath}
\usepackage{amsthm}
\usepackage{amssymb}
\usepackage{latexsym}
\usepackage{tikz-cd}
\usepackage{textcomp}
\usepackage{tabularx,array}
\usepackage[all]{xy}
\newfont{\msam}{msam10}

\newtheorem{theorem}[]{Theorem}
\newtheorem{proposition}[]{Proposition}
\newtheorem{corollary}[]{Corollary}
\newtheorem{lemma}[]{Lemma}
\theoremstyle{definition}
\newtheorem{definition}[]{Definition}
\newtheorem{remark}[]{Remark}
\newtheorem{example}[]{Example}

\newtheorem{prop}[theorem]{Proposition}
\newtheorem{cor}[theorem]{Corollary}

\let\nc\newcommand

\def\bthm{\begin{theorem}}
\def\ethm{\end{theorem}}
\def\blemma{\begin{lemma}}
\def\elemma{\end{lemma}}
\def\bproof{\begin{proof}}
\def\eproof{\end{proof}}
\def\bprop{\begin{proposition}}
\def\eprop{\end{proposition}}
\def\bcor{\begin{corollary}}
\def\ecor{\end{corollary}}
\nc{\la}{\label}
\def\Z{\mathbb{Z}}
\def\N{\mathbb{N}}
\def\Q{\mathbb{Q}}
\def\c{\mathbb{C}}
\def\M{\mathcal{M}}

\def\L{\mathbb{L}}
\def\R{\mathbb{R}}

\def\FT{\mathcal{C}}
\def\eFT{\widetilde{\mathcal{C}}}
\def\Mor{\mathtt{Mor}}

\def\mfa{\mathfrak{a}}
\nc{\cC}{\mathcal{C}}
\def\Com{\mathtt{Com}}

\def\Alg{\mathtt{Alg}}

\def\DGL{\mathtt{DGL}}
\def\Mod{\mathtt{Mod}}

\def\cAlg{\mathtt{CAlg}}

\def\Sets{\mathtt{Sets}}
\def\DGA{\mathtt{DGA}}

\def\cDGA{\mathtt{DGCA}}
\def\DGL{\mathtt{DGL}}
\def\DGC{\mathtt{DGC}}

\def\iDGA{\mathtt{DGA}^{\ob}}
\def\iDGC{\mathtt{DGC}^{\ob}}
\def\iAlg{\mathtt{Alg}^{\ob}}
\def\s{\mathtt{s}}

\def\D{{\mathscr D}}
\def\C{{\mathscr C}}

\def\Ho{{\mathtt{Ho}}}
\def\cb{\mathbf{\Omega}}
\def\bB{\boldsymbol{B}}
\def\Tw{\mathtt{Tw}}
\nc{\Ob}{{\rm Ob}}
\nc{\Hom}{{\rm{Hom}}}
\nc{\HOM}{\mathbf{Hom}}
\nc{\DER}{\underline{\rm{Der}}}
\nc{\END}{\underline{\rm{End}}}
\nc{\bSym}{\mathbf{\Lambda}}
\nc{\Ext}{{\rm{Ext}}}
\nc{\Rep}{{\rm{Rep}}}
\nc{\DRep}{{\rm{DRep}}}
\nc{\NCRep}{\widetilde{\rm{Rep}}}
\nc{\RAct}{{\rm{RAct}}}
\nc{\bs}{\backslash}
\nc{\cn}{{\tt{Cone}}}
\nc{\n}{\natural}
\nc{\colim}{{\tt{colim}}}
\nc{\B}{{\mathrm{B}}}
\nc{\Ba}{\overline{\mathrm{B}}}
\nc{\bC}{\overline{C}}
\nc{\EXT}{\underline{\rm{Ext}}}
\nc{\TOR}{\underline{\rm{Tor}}}
\def\H{\mathrm H}
\def\HD{\mathrm{HD}}
\def\HO{\mathrm{HO}}
\def\HRi{\mathrm{HR^{\ob}_{\bullet}}}
\def\HRim{\mathrm{HR^{\ob}_{m}}}
\def\SR{\mathrm{Tr}_{\rm sym}}
\def\HOR{\mathrm{Tr}_{\rm hOct}}
\def\rH{\overline{\mathrm{H}}}

\nc{\End}{{\rm{End}}}
\nc{\GL}{{\rm{GL}}}
\nc{\gl}{{\mathfrak{gl}}}
\nc{\PGL}{{\rm{PGL}}}
\nc{\SL}{{\rm{SL}}}
\nc{\PSL}{{\rm{PSL}}}
\nc{\ad}{{\rm{ad}}}
\nc{\Ad}{{\rm{Ad}}}
\nc{\dlim}{\varinjlim}
\nc{\plim}{\varprojlim}

\newcommand{\cA}{{\mathcal A}}

\newcommand{\Tor}{{\rm{Tor}}}

\newcommand{\Sym}{{\rm{Sym}}}

\newcommand{\id}{{\rm{Id}}}

\newcommand{\Tr}{{\rm{Tr}}}

\newcommand{\ev}{{\rm{ev}}}

\newcommand{\g}{\mathfrak{g}}

\newcommand{\into}{\,\hookrightarrow\,}

\newcommand{\onto}{\,\twoheadrightarrow\,}
\newcommand{\sonto}{\,\stackrel{\sim}{\twoheadrightarrow}\,}
\def\com{\mathtt{Com}}

\def\sComm{\mathrm{sComm}}
\def\sMod{\mathrm{sMod}}
\def\scAlg{\mathrm{sCAlg}}

\newcommand{\rar}{\rightarrow}

\numberwithin{equation}{section}
\numberwithin{theorem}{section}
\numberwithin{lemma}{section}
\numberwithin{proposition}{section}
\numberwithin{corollary}{section}
\numberwithin{example}{section}
\numberwithin{remark}{section}
\numberwithin{definition}{section}

\DeclareMathOperator*{\Moplus}{\text{\raisebox{0.25ex}{\scalebox{0.8}{$\bigoplus$}}}}

\newcommand{\bL}{\boldsymbol{\Lambda}}
 \newcommand{\HA}{B_{\rm hOct}(A)}
\newcommand{\HR}{B_{\rm hOct}(R)}
 \newcommand{\HH}{B_{\rm hOct}}
 \nc{\HI}{HR^{\ob}}
  \newcommand{\p}{\varphi}
\nc{\mc}{\mathcal}
\nc{\ob}{\ast}

\newcommand{\sinto}{\,\stackrel{\sim}{\into}\,}

\nc{\fG}{\mathfrak{G}}
\nc{\fM}{\mathfrak{M}}
\nc{\sSet}{\mathtt{sSet}}
\nc{\sGr}{\mathtt{sGr}}
\nc{\sGrh}{\mathtt{sGr}_h}
\nc{\hocolim}{{\rm hocolim}}
\nc{\SPi}{{\rm SP}^{\infty}}
\nc{\mtwo}{\mu_2}
\nc{\Ztwo}{\underline{\Z/2}}
\nc{\GLm}{{\rm GL}_n^{\!\sigma}}
\nc{\Psicyc}{\Psi_{\rm cyc}}
\nc{\Psidih}{\Psi_{\rm dih}}
\nc{\Psisym}{\Psi_{\rm sym}}
\nc{\Psihyp}{\Psi_{\rm hOct}}
\nc{\tPsihyp}{\widetilde{\Psi}_{\rm hOct}}
\title{Representation homology of involutive algebras and hyperoctaherdral homology}
\author{Yuri Berest}
\address{Department of Mathematics,
Cornell University, Ithaca, NY 14853-4201, USA}
\email{berest@math.cornell.edu}
\author{Ajay C. Ramadoss}
\address{Department of Mathematics,
Indiana University,
Bloomington, IN 47405, USA}
\email{ajcramad@indiana.edu}
\author{John Wood}
\address{Department of Mathematics,
Indiana University,
Bloomington, IN 47405, USA}
\email{wood31@iu.edu}

\begin{document}

\begin{abstract}
In this paper, we introduce and study the (involutive) representation homology of involutive algebras, following the development of representation homology of associative algebra in \cite{BKR,BFPRW}. Our main motivation is to understand the hyperoctahedral homology of involutive algebras. We show that over a field of characteristic $0$, the hyperoctahedral homology of an involutive algebra is isomorphic (as a graded commutative algebra) to its one dimensional involutive representation homology. This result enables the explicit computation of the hyperoctahedral homology of the universal enveloping algebras of involutive (DG) Lie algebras. For group algebras $k[\Gamma]$ equipped with the involution $g \mapsto g^{-1}$, we relate the involutive representation homology of $k[\Gamma]$ to the representation homology of the group $\Gamma$ with coefficients in $\mu_2$, and we show that the latter is the homology of the loop space of the mod $2$ Dold-Thom space of $B\Gamma$; more generally, we compute the representation homology of $\Gamma$ with coefficients in an arbitrary diagonalizable group scheme. As a consequence, the hyperoctahedral homology of $k[\Gamma]$ is concentrated in degree $0$ in characteristic $0$, for every $\Gamma$.
\end{abstract}
\maketitle

\section{Introduction}
The present paper is a sequel to  \cite{BR22c} and \cite{BR22}, where we study the relation between cyclic, symmetric and representation homology of associative algebras. Our goal is to extend the results of {\it loc. cit.} to  {\it involutive  algebras} (i.e., the associative $k$-algebras $A$ equipped  with an anti-automorphism $\sigma: A \to A$  such that $\sigma^2 = \id$). For such algebras, the cyclic and symmetric homology are replaced  by {\it dihedral} and {\it hyperoctahedral} homology respectively, while the representation homology  introduced in \cite{BKR} is replaced by {\it involutive representation homology} which we will define here following \cite{BR14}.

The four homology theories mentioned above --- cyclic, dihedral, symmetric and hyperoctahedral --- arise from the same general construction due to Fiedorowicz and Loday \cite{FL91} (see also \cite{FT87, Kr87}). Recall that a {\it crossed simplicial group} is a family of groups $\{G_{n}\}_{n \geq 0}$ together with a small category $\Delta G$ that contains the simplicial category $\Delta$, has the same objects $[n]$, $ n \geq 0$, as $\Delta$, and is characterized by two properties: ${\rm Aut}_{\Delta G}([n]) \cong G_{n}^{\rm op}$ for all $n$, and every morphism of $\Delta G$ factors uniquely as a morphism of $\Delta$ followed by an automorphism\footnote{The idea of a categorical crossed product of $\Delta$ with a family of groups was formalized by Loday and Fiedorowicz in \cite{FL91}; the same structure appears in the work of Krasauskas \cite{Kr87} under the name of {\it skew-simplicial groups}.}. The basic examples are the cyclic category $\Delta C$ of Connes \cite{C83} associated with the cyclic groups $\{C_{n+1}\}_{n \geq 0}$, the dihedral category $\Delta D$ associated with the dihedral groups $\{D_{n+1}\}_{n \geq 0}$, the symmetric category $\Delta S$ associated with the symmetric groups $\{S_{n+1}\}_{n \geq 0}$, and the hyperoctahedral category $\Delta H$ associated with the hyperoctahedral groups $\{H_{n+1} = (C_2)^{n+1} \rtimes S_{n+1}\}_{n \geq 0}$. Following \cite{FL91}, a crossed simplicial group $\Delta G$ gives rise to a homology theory on the category of functors $E: \Delta G \to {\rm Mod}_k$ (or $E: \Delta G^{\rm op} \to {\rm Mod}_k$): one defines
$$ {\rm HG}_\bullet(E)\,:=\, {\rm Tor}^{\Delta G}_\bullet(k, E)\,,$$
where $k$ is the trivial (constant) $\Delta G^{\rm op}$-module. To obtain a homology theory for algebras one has to feed this construction with a functor built from an algebra $A$, i.e. a {\it bar construction} over $\Delta G$. The prototypical example is Connes' {\it cyclic bar construction} $B^{\rm cyc}A: \Delta C^{\rm op} \to {\rm Mod}_k$ of an associative algebra $A$, which extends the standard simplicial module $B^{\rm hoch}A: \Delta^{\rm op} \to {\rm Mod}_k$ computing Hochschild homology; by analogy with the 'Tor' formula $\,{\rm HH}_\bullet(A) = {\rm Tor}^{\Delta^{\rm op}}_\bullet(k,B^{\rm hoch}A)$, Connes defined cyclic homology as ${\rm HC}_{\bullet}(A) := {\rm Tor}^{\Delta C^{\rm op}}_\bullet(k,B^{\rm cyc}A)$. If $A$ is involutive, then $B^{\rm cyc}A$ extends naturally to a dihedral module $B^{\rm dih}A:\Delta D^{\rm op} \to {\rm Mod}_k$, the {\it dihedral bar construction} of $A$, and ${\rm HD}_{\bullet}(A) := {\rm Tor}^{\Delta D^{\rm op}}_\bullet(k,B^{\rm dih}A)$ is the dihedral homology of $A$ (see \cite[Section 6.3.7]{L}). 
For the symmetric and hyperoctahedral categories the situation is different: as observed in \cite{FL91} (see \cite[Ex. 6.3.2]{L}), for any {\it contravariant} functor $E$ on $\Delta S$ or $\Delta H$, the Tor groups ${\rm Tor}^{\Delta S^{\rm op}}_\bullet(k, E)$ and ${\rm Tor}^{\Delta H^{\rm op}}_\bullet(k, E)$ reduce to ${\rm Tor}^{\Delta^{\rm op}}_\bullet(k, E)$, so that no contravariant bar construction can detect the action of the symmetric and hyperoctahedral groups. It was Fiedorowicz \cite{F} who noticed that the situation changes if one uses {\it covariant} bar constructions: for any associative algebra $A$, there is a natural functor $B_{\rm sym}A:\Delta S \to {\rm Mod}_k$, the {\it symmetric bar construction} of $A$, and if $A$ is involutive, $B_{\rm sym}A$ extends to a functor $B_{\rm hOct}A: \Delta H \to {\rm Mod}_k$, the {\it hyperoctahedral bar construction} of $A$. These bar constructions give rise to the homology theories
$$ {\rm HS}_\bullet(A)\,:=\,{\rm Tor}^{\Delta S}_\bullet(k, B_{\rm sym}A)\,,\qquad {\rm HO}_\bullet(A)\,:=\,{\rm Tor}^{\Delta H}_\bullet(k, B_{\rm hOct}A)\,, $$
called {\it symmetric homology} (defined for arbitrary associative algebras) and {\it hyperoctahedral homology} (defined for involutive associative algebras). 

Despite the formal resemblance to cyclic and dihedral homology, symmetric and hyperoctahedral homology are much more complicated and harder to compute. Following the program outlined in \cite{AF07}, the first computations of symmetric homology were done by S. Ault \cite{Au1, Au2}, while the systematic study of hyperoctahedral homology was initiated by D. Graves  \cite{G21, G22}. In \cite{BR22c}, the first two authors showed that, for 
associative algebras $A$ over a field of characteristic $0$, the symmetric homology ${\rm HS}_\bullet(A)$  is naturally isomorphic to the one-dimensional {\it representation homology} ${\rm HR}_\bullet(A, k)$ defined in \cite{BKR}. Using the results of the earlier work on representation homology (see \cite{BR, BFR, BFPRW, BRYI, BRYII, BR22}), this allowed us to compute ${\rm HS}_\bullet(A)$ explicitly in many interesting cases, including the symmetric algebras $ A = \Sym_k(V) $ and the universal enveloping algebras $A = U\mfa$ of Lie algebras, and to settle some old conjectures of Ault and Fiedorowicz \cite{AF07} on symmetric homology in characteristic $0$. 
The goal of the present paper is to extend this approach and computations to the hyperoctahedral homology, which requires, first of all, introducing the involutive representation homology.
Part of our motivation comes from a very interesting recent work of G. Angelini-Knoll, D. Chan, T. Gerhardt, M. Merling and M. P\'eroux \cite{AMP, ACGMP} who introduced topological (spectral) refinements of homology theories associated with crossed simplicial groups, and extended to this setting the isomorphism between symmetric and representation homology discovered in \cite{BR22c}.

In the rest of the Introduction, we outline our definition of involutive representation homology, its relation to dihedral homology, and then state our main results. Unlike in (most of) Sections~\ref{Sec2}, \ref{SecTr} and \ref{Sec3}, where $k$ will be a field of characterstic $0$, we will work here with simplicial algebras defined over an arbitrary commutative ring $k$.
Let $ \iAlg_k $ and $ \cAlg_k $ denote the categories of involutive associative and commutative algebras (with $1$), both defined over a fixed commutative ring $k$. 
For a fixed integer $n \geqslant 1$, fix an involution on the algebra $\M_n(k)$ of $ (n\times n)$-matrices with entries in $k$. Consider the matrix functor $\M_n(\mbox{--}):\, \cAlg_k \to \iAlg_k\,$ that takes a commutative $k$-algebra $C$ to the involutive $k$-algebra $\M_n(C)$ of $ (n\times n)$-matrices with entries in $C$ (with involution induced by the involution on $\M_n(k)$). It is easy to prove (following, e.g., \cite{Co79} and \cite[Prop. 2.1]{BKR} for a general result) that $\M_n(\mbox{--})$ has a left adjoint
\begin{equation}
\la{adj1}
(\,\mbox{--}\,)^{\ob}_n:\ \iAlg_k \to \cAlg_k
\end{equation}
called the {\it $n$-dimensional representation functor} on $ \iAlg_k $. This terminology is motivated by the fact that for any algebra $A$, $\,(A)^{\ob}_n = {\mathcal{O}}[\Rep^{\ob}_n(A)] \,$ is the coordinate ring of the affine
moduli scheme $ \Rep^{\ob}_n(A) $ parametrizing the involution preserving $n$-dimensional $k$-linear representations of $A$. Now, to define representation homology we simply `derive' the functor \eqref{adj1},
following the standard procedure in homotopical algebra: we prolong the adjunction $\,(\,\mbox{--}\,)^{\ob}_n: \iAlg_k \rightleftarrows \cAlg_k: \M_n(\mbox{--})\,$ to simplicial algebras, using degreewise extension:
\begin{equation}
\la{adj2}
(\,\mbox{--}\,)_n:\ \s\iAlg_k \,\rightleftarrows\, \scAlg_k\ : \M_n(\mbox{--})
\end{equation}
and then replace the adjoint functors \eqref{adj2} with their universal homotopy approximations represented by derived functors. The key point here is that \eqref{adj2} is a Quillen adjunction with respect to the standard projective model structures on simplicial 
algebras. Hence the derived functors of \eqref{adj2} exist and are well behaved: in particular, by Quillen's Adjunction Theorem, they form an adjoint pair at the level homotopy categories:
$$
\L(\,\mbox{--}\,)^{\ob}_n:\ \Ho(\s\iAlg_k) \,\rightleftarrows\, \Ho(\scAlg_k)\, : \M_n(\mbox{--})\ . 
$$
For a (simplicial) involutive $k$-algebra $ A $, we can now define its {\it $n$-dimensional involutive representation homology} by
\begin{equation}
\la{RHDefAlg} 
{\rm HR}^{\ob}_\bullet(A, k^n)\, := \, \pi_\bullet[(QA)_n]\,,
\end{equation}
where $\, QA \sonto A\,$ is a cofibrant simplicial resolution
of $A$. By definition, $ {\rm HR}_\bullet(A, n) $ is a graded commutative $k$-algebra, with $\, {\rm HR}_0(A, n) \cong (A)^{\ob}_n\,$, that does not depend on the choice of the
resolution $QA$. For $ n = 1 $, the matrix functor is just the inclusion $ \scAlg_k \into \s\iAlg_k $, and hence its left adjoint --- the $1$-dimensional derived representation functor $  \L(\,\mbox{--}\,)^{\ob}_1: \,\Ho(\s\iAlg_k) \to \Ho(\scAlg_k) $ --- coincides with the derived involutiv abelianization (`commutativization') of algebras, i.e., the derived functor of the functor sending an involutive algebra $A$ to the quotient of its abelianization $A_{\rm ab}$ by the ideal generated by image of $\id-\sigma$, where $\sigma$ is the involution induced on $A_{\rm ab}$.

Now, to relate representation homology to dihedral homology we recall that 
$\, {\rm HD}_0(A) = A/([A,A]+{\rm Im}(\id-\sigma))\,$ for any $ A \in \iAlg_k$, and observe that there is an obvious trace map $\,\Tr_n(A):\ {\rm HD}_0(A)\,\to\,(A)^{\ob}_n\,$
%
%
induced by the character of the universal $n$-dimensional representation of $A$. It turns out that this classical character map extends naturally to higher dihedral homology taking values in involutive representation homology: i.e., for any $n\ge 1$, there is a canonical graded
$k$-linear map
\begin{equation}
\la{DTr}
\Tr_n(A)_\bullet:\ {\rm HD}_\bullet(A)\,\to\,{\rm HR}^{\ob}_\bullet(A, k^n)
\end{equation}
called the {\it involutive derived character map of $n$-dimensional representations of $A$}. 
The map \eqref{DTr} was originally constructed in \cite{BKR} for associative DG algebras (without involution) over a field $k$ of characteristic zero. The analogous construction for involutive associative DG algebras over a field of characteristic $0$ appears in Section \ref{SecTr}. However, working simplicially, we can construct
\eqref{DTr} for $A$ defined over any field (or even an arbitrary commutative ring $k$, provided $A$ is flat over $k$; see Section~\ref{S3.1}). For $n=1$, we will simplify the notation \eqref{DTr} writing $\Tr_1(A)_\bullet$ as
\begin{equation}
\la{Tr*}
\Tr(A)_\bullet:\ {\rm \HD}_\bullet(A)\,\to\,{\rm HR}^{\ob}_\bullet(A, k)
\end{equation}

Next, recall that, for any involutive algebra $A$, there is also a natural map
\begin{equation}
 \la{HCS}
 i^*:\ {\rm HD}_\bullet(A)\,\to\,{\rm HO}_\bullet(A)
\end{equation}
relating dihedral homology to hyperoctahedral homology via the inclusion of categories $ i: \Delta D^{\rm op} \into \Delta H $ (The existence of this map has been noticed already in \cite{F} (see, {\it loc. cit.}, Remark 1.4), and it is studied in \cite[Sec. 3.3]{G21}; we will review its construction in Section~\ref{S3.1} below.) Our first observation (see Proposition~\ref{factorTr}) is that the derived character map \eqref{Tr*} factors 
through \eqref{HCS}, inducing a natural linear map that relates hyperoctahedral homology to representation homology:
\begin{equation}
\la{Triso}
\HOR(A)_\bullet:\ \HO_\bullet(A)\, \to\, \HRi(A, k)\ . 
\end{equation}
Finally, to state our main theorem we recall a result of Graves (see \cite[Cor. 4.5]{G22}) showing that the $ \HO_\bullet(A) $ for any $A$ admits natural 
homology operations and a Pontryagin product making it a graded commutative $k$-algebra. 

\newtheorem*{T1a}{Theorem~\ref{T1a}}
\begin{T1a}
Let $k$ be a field of characteristic $0$. For any simplicial involutive associative algebra $A \in \s\iAlg_k$, the map \eqref{Triso} is an isomorphism of graded commutative $k$-algebras
$$  \HO_\bullet(A)\,\cong\, \HRi(A, k)\ . $$
\end{T1a}

While Theorem \ref{T1a} may be viewed as a formal extension of \cite[Theorem 3.1]{BR22}, it enables us to translate involutive analogs (in Sections \ref{Sec2},\ref{SecTr} and \ref{Sec3}) of known facts about representation homology in \cite{BKR,BFPRW} to hyperoctahedral homology. An important example is the computation of the hyperoctahedral homology of the universal enveloping algebra of an involutive Lie algebra. Recall (see Definition \ref{iliealg}) that an involutive (DG) Lie algebra $\mfa$ (over $k$) is a Lie algebra with a square identity $k$-linear map $\sigma:\mfa \to \mfa$ (of complexes) such that $-\sigma$ is a (DG) Lie algebra automorphism of $\mfa$. It is easy to show that if $\mfa$ is involutive, then the involution of the tensor algebra $T_k\mfa$ which restricts to the given involution on $\mfa$ descends to the universal enveloping algebra $U\mfa$. Further the Chevalley-Eilenberg chain coalgebra $\C(\mfa;k)$, which is cocommutative, acquires a skew-involution whose co-restriction to $\mfa[1]$ is the involution on $\mfa$. This skew involution induces a square identity linear operator on the reduced Chevalley-Eilenberg homology $\rH_\bullet(\mfa;k)$. We denote its $(+1)$-eigenspace by $\rH^{\rm ev}_\bullet(\mfa;k)$.
\newtheorem*{houainv}{Theorem~\ref{houainv}}
\begin{houainv}
Let $\mfa$ be an involutive DG Lie algebra over a field$k$ of characteristic $0$. Equip $U\mfa$ with the involution induced by that on $\mfa$. Then there is an isomorphism of graded commutative $k$-algebras 
$$
{\rm HO}_\bullet(U\mfa) \,\cong\, \Lambda_k[\rH^{\rm ev}_{\bullet+1}(\mfa;k)]\,
$$
where $\Lambda_k$ stands for the graded symmetric algebra over $k$.
\end{houainv}
If $V$ is a finite dimensional $k$-vector space equipped with a $k$-linear automorphism $\sigma$ of order $2$, then $\sigma$ extends to an involution on the symmetric algebra $\Sym(V)$. Let $V_+$ and $V_-$ denote the $+1$ and $-1$ eigenspaces of the $\sigma$ action on $V$ respectively. As a consequence of Theorem \ref{houainv}, we have
\newtheorem*{hosym}{Corollary~\ref{hosym}}
\begin{hosym}
There is an isomorphism of graded commutative $k$-algebras
$${\rm HO}_\bullet(\Sym(V))\,\cong\, \bigwedge_k\left[\big(\Lambda^{\rm odd}(V_+[1]) \otimes \Lambda(V_-[1])\big)[-1]\right]\,,$$
where $\Lambda^{\rm odd}$ stands for the direct sum of odd graded symmetric powers (over $k$).
\end{hosym}
In particular, if $k[x_1,\ldots,x_n]$ is equipped with the involution that fixes $x_i$ for $1 \leq i \leq p$ and sends $x_j$ to $-x_j$ for $p <j \leq n$, then 
$${\rm HO}_\bullet(k[x_1,\ldots,x_n])\,\cong\, \bigwedge_k\left[\big(\Lambda^{\rm odd}(sx_1,\ldots,sx_p) \otimes \Lambda(sx_{p+1},\ldots,sx_n)\big)[-1]\right]\,,$$
where $sx_i$ denotes $x_i$ placed in homological degree $1$.

Another consequence of Theorem \ref{houainv} is the explicit computation of ${\rm HO}_\bullet(U\gl_n(k))$ for any involution on $U\gl_n(k)$ induced by a trace-preserving involution on $\M_n(k)$ (for example, any conjugate by an invertible matrix of the transpose). Such an involution also restricts to an involution on $\mathfrak{sl}_n(k)$ and hence, on $U\mathfrak{sl}_n(k)$. For any such involution, we have
\newtheorem*{hougl}{Corollary~\ref{hougl}}
\begin{hougl}
For all $ n\ge 1$, there are isomorphisms of graded commutative $k$-algebras
\begin{align*}
   {\rm HO}_\bullet(U\gl_n(k)) &\cong\,\bigwedge_k\left[\big(\Lambda^{\rm odd}(\xi_1,\xi_5,\ldots,\xi_{4\lfloor \frac{n-1}{4}\rfloor+1}) \otimes \Lambda(\xi_3,\xi_7,\ldots,\xi_{4\lfloor \frac{n-3}{4}\rfloor+3})\big)[-1]\right]\,,\\
    {\rm HO}_\bullet(U\mathfrak{sl}_n(k)) &\cong\,\bigwedge_k\left[\big(\Lambda^{\rm odd}(\xi_5,\ldots,\xi_{4\lfloor \frac{n-1}{4}\rfloor+1}) \otimes \Lambda(\xi_3,\xi_7,\ldots,\xi_{4\lfloor \frac{n-3}{4}\rfloor+3})\big)[-1]\right]\ .
\end{align*}  
\end{hougl}

Our last application of Theorem~\ref{T1a} concerns group algebras. 
To compute the hyperoctahedral homology of such algebras we prove a new general result on representation homology of spaces, which is of independent interest. Recall that 
representation homology can be defined for any pointed connected topological space $X$ with coefficients in an affine algebraic group scheme $G$ over a commutative ring $k$. Following \cite{BRYI, BRYII}, we denote this homology
by $ {\rm HR}_\bullet(X, G) $, which is, by definition, a graded commutative associative $k$-algebra.
In \cite{BR22}, we showed that, for $G = \mathbb{G}_m(k)$, the multiplicative group of $k$, there is a natural isomorphism  (\cite[Theorem 1.1]{BR22})
\begin{equation} 
\label{DoldTh}
{\rm HR}_\bullet(X, \,\mathbb{G}_m(k))\,\cong\,\H_\bullet\big(\Omega\,\SPi(X);\,k\big)\,,
\end{equation}
where $ \SPi(X) $ is the Dold-Thom space (infinite symmetric product) of $X$. By a well-known theorem of Kan, the pointed connected homotopy types can be modelled by simplicial groups:
if we take $ X = B \Gamma $ for such a group $\Gamma $, then 
the left hand side of \eqref{DoldTh} is isomorphic to the representation homology of the simplicial algebra $ k[\Gamma]$, at least when $k$ is a field of characteristic $0$ (see \cite[Theorem 2.1]{BR22c}). In this case, the symmetric trace 
$ {\rm HS}_\bullet(k[\Gamma]) \to {\rm HR}_\bullet(k[\Gamma], k)$
is represented by a natural map of spaces studied in \cite{BR22}:
\begin{equation}
\label{Symtoptr}
\Omega \,\Omega^{\infty} \Sigma^{\infty}(B \Gamma)\,\to\,
\Omega\,\SPi(B \Gamma)
\end{equation}
In this paper, we extend these results to the involutive case.
First, we generalize the isomorphism \eqref{DoldTh}
to arbitrary abelian groups $\cA$ as follows. For a fixed basering $k$, we associate to $\cA$ the diagonalizable group scheme $ D_k(\cA) := {\rm Spec}\,k[\cA]$, where $k[\cA]$ is the group algebra of $\cA$ over $k$. On the other hand, for a pointed connected space $(X,*)$, we consider the function space
\begin{equation}
\label{spxa}
\SPi(X, \cA) := \{f: X \to \cA\ : \ f(\ast) = 0\quad \mbox{and}\quad |\,{\rm supp}(f)\,| < \infty\}
\end{equation}
equipped with the weak (colimit) topology for the subspace filtration by functions supported on finite configurations of points of fixed cardinality. This space is the classical {\it McCord model}\, \cite{McC69}\, for the Dold-Thom space of $X$ with coefficients in $\cA$ in the sense that 
$\,\pi_i\,\SPi(X;\cA) \cong \tilde{\H}_i(X; \cA)\,$ for all $ i \ge 0 $; it is an abelian topological monoid 
that is natural (bifunctorial) in $X$ and $\cA$ (see \cite{Ku03}). In Section~\ref{Sgroups}, we prove
\newtheorem*{Tgroups}{Theorem~\ref{Tgroups}}
\begin{Tgroups}
Let $k$ be an arbitrary commutative ring. For any abelian group $\cA$ and any pointed connected space $X$, there is a natural isomorphism
\begin{equation} 
\la{isoDA}
{\rm HR}_\bullet(X, D_k(\cA))\,\cong\,\H_\bullet\big(\Omega\,\SPi(X, \cA);\,k\big)\ , 
\end{equation}
where $\Omega\,\SPi(X, A)$ is the based loop space of the McCord-Dold-Thom space \eqref{spxa}. Consequently,
$$ {\rm HR}_\bullet(X, D_k(\cA))\,\cong\,\H_\bullet\Big(\prod_{n \geq 1}K\big(\tilde{\H}_n(X, \cA),\,n-1\big);\,k\Big)\,. 
$$
\end{Tgroups}
We remark that the second isomorphism of Theorem~\ref{Tgroups}
follows from the first, because any path-connected abelian topological monoid is a product of Eilenberg–MacLane spaces (see \cite[Cor. 4K.7]{Hat02}). 

For $ \cA = \Z $, the isomorphism \eqref{isoDA}  specializes to 
 \eqref{DoldTh}.
The case relevant to involutive algebras is $ \cA = \Z/2$. Indeed, if $\Gamma$ is a (simplicial) group, its group algebra $k[\Gamma]$ carries the canonical involution defined on generators by $g \mapsto g^{-1}$. When
$k$ is a field of characteristic zero, we prove  (see  Proposition~\ref{grcomp}) 
\begin{equation} 
\la{isoGLm}
\HRi(k[\Gamma], k^n) \,\cong\, 
{\rm HR}_\bullet(B\Gamma, \GLm(k))\ ,\quad \forall \, n \ge 1 \ ,
\end{equation}
where  $\,\GLm(k) := \{g \in \GL_n(k)\,:\,\sigma(g)=g^{-1}\} $ denotes the subgroup of $ \GL_n(k) $ determined by a chosen involution\footnote{Recall (see \cite[Thm. 1]{AA}) that every $k$-linear involution on $\M_n(k) $ has the form $ \sigma(X) = J^{-1} X^T J $ with $ J \in \GL_n(k) $ satisfying $ J^{T} = \pm J$. For $J$ symmetric,  $\GLm(k) = {\rm O}(J)$, while  for $J$ skew-symmetric, $\,\GLm(k) = {\rm Sp}(J)$.}\, $\sigma$ on $\M_n(k)$.
For $n=1$, we have $\, \GL_1^{\sigma}(k) = \mu_2 = D_k(\Z/2) $; hence, combining the isomorphism \eqref{isoDA} for $ \cA = \Z/2 $ with \eqref{isoGLm} yields
\begin{equation}
\HRi(k[\Gamma], k) \,\cong\, {\rm HR}_\bullet(B\Gamma, \mtwo) 
\,\cong\,\H_\bullet\big(\Omega\,\SPi(B\Gamma, \Z/2);\,k\big)
\end{equation}
It turns out that with above isomorphisms,
the hyperoctahedral trace map \eqref{Triso} for $ A = k[\Gamma] $ is induced on homology by a natural map of spaces
\begin{equation}
\label{Hyptoptr}
E(\Z/2)_{+} \wedge_{\Z/2} \Omega \,\Omega^{\infty} \Sigma^{\infty}(B \Gamma) \, \to\, \Omega\,\SPi(B\Gamma, \Z/2)
\end{equation}
where the domain is the bar construction $B(C_\infty E, C_1, \Gamma)$  of \cite[Theorem 7.20]{G21} computing ${\rm HO}_\bullet(k[\Gamma])$. This is our generalization of \eqref{Hyptoptr}.

Now, if $k$ is a field of characteristic $0$, we can use the second isomorphism of Theorem~\ref{Tgroups} (for $ \cA = \Z/2$) together with the  isomorphism of Theorem~\ref{T1a} to conclude 
\newtheorem*{Cgroups}{Corollary~\ref{Cgroups}}
\begin{Cgroups}
let $k$ be a field of characteristic zero. For any (homotopy) simplicial group $\Gamma$, there is an isomorphism of graded commutative $k$-algebras
$$ 
{\rm HO}_\bullet(k[\Gamma])\,\cong\, k[\H_1(B\Gamma;\Z/2)]\,. 
$$
In particular, $\,{\rm HO}_{i}(k[\Gamma])\,$ vanishes for all $ i > 0 $.
\end{Cgroups}
%


This paper is organized as follows. In Sections \ref{Sec2} we introduce and obtain basic results regarding the representation homology of involutive DG algebras mostly following the analogous results for the representation homology of associative DG algebras in \cite{BKR}. In Section \ref{simpdgcomp}, we compare the DG and simplicial definitions of involutive representation homology, showing the they are equivalent over a field of characteristic $0$ (Proposition \ref{simp}). In Section \ref{SecTr}, we construct derived character maps relating dihedral homology to involutive representation homology. In Section \ref{Sec3}, we show that the involutive representation homology of an augmented, involutive DG algebra is isomorphic to the Chevalley-Eilenberg homology of a current DG Lie coalgebra constructed from its Koszul dual (Theorem \ref{invmain}). The statements and proofs of our main results appear in Section \ref{S3}. In Section \ref{S3.1}, we show that the derived character map \eqref{Tr*} factors through hyperoctahedral homology (Proposition \ref{factorTr}). In Section \ref{S3.2}, we prove Theorem \ref{T1a}. The key fact behind this result is Proposition \ref{pC}. In Section \ref{S3.3}, we prove Theorem \ref{houainv} and obtain Corollary \ref{hosym} and Corollary \ref{hougl} as consequences thereof. Finally, in Section \ref{Sgroups}, we turn to group algebras: we extend the dihedral and hyperoctahedral bar constructions to homotopy simplicial groups, construct the dihedral and hyperoctahedral character maps, prove Theorem \ref{Tgroups} and deduce Corollary \ref{Cgroups}. 

\subsection*{Historical remark} The definition of representation homology of involutive algebras  that we use in this paper (see Sections~\ref{Sec2}) is borrowed from the unpublished preprint \cite{BR14} of the first two authors. 
Motivated by the  construction of representation homology  in \cite{BKR}, the goal of \cite{BR14} was to introduce and study the derived extension of skein algebras of $3$-manifolds that play a fundamental role in knot theory and geometric topology (see, e.g., \cite{PS}).
(Among other things, it was shown in \cite{BR14} that for
the complex group algebra $A = \c[\pi_1(M)]$ of the fundamental group of a compact $3$-manifold $M$ equipped with canonical involution $ g \mapsto g^{-1}$, the invariant part $ {\rm HR}^{\ob}_0(A, \c^2)^{{\rm GL}_2} $ of the $0$-dimensional involutive representation homology of $A$ coincides with the classical skein algebra of $M$. The main theorem of  \cite{BR14} then provides a derived version of a theorem of Brumfiel and Hilden \cite{BH}, identifying the skein algebra of an involutive algebra $A$  with its universal trace quotient.)
This historical background partly explains our use of the model-categorical language in the present paper. 
Although representation homology of involutive algebras could be naturally (probably, more concisely) defined in the setting of $\infty$-categories, we decided to keep the original framework of \cite{BR14} as it appears to be most suitable for the proof of Theorem \ref{invmain} and our computations in Section~\ref{S3}.

\subsection*{AI Disclosure}{
We used Claude (Fable 5.1) to provide an independent verification of the proof of 
Theorem~\ref{Tgroups}. As mentioned in the Introduction, this theorem is a generalization of \cite[Theorem 1.1]{BR22} to arbitrary abelian group coefficients (McCord spaces). We have verified that the arguments of \cite{BR22} extend {\it mutatis mutandis} to this more general situation and take full responsibility for this content. All results in other sections were obtained without any AI assistance.
}

\section{Derived representation schemes for involutive algebras} \la{Sec2}
In this section, we construct the (derived) representation scheme parametrizing involution preserving representations of an involutive DG algebra. The constructions and basic results of this section, as well as the proofs thereof, parallel those of~\cite[Section 2]{BKR} (also see \cite[Section 3]{BFR}).

\subsection{The classical representation functor}

\begin{definition}
An {\it involution} on a $k$-algebra $A$ is a $k$-algebra homomorphism $\sigma\,:\, A \rar A^{\mathrm{op}}$ such that $\sigma^{\mathrm{op}} \circ \sigma\,=\,\id$. Equivalently, $\sigma$ is a square identity $k$-linear map from $A$ to itself such that for all $a,b \,\in\,A$, $\sigma(ab)\,=\,\sigma(b)\sigma(a)$. 
\end{definition}

Let $\iAlg_k$ denote the category of $k$-algebras equipped with involution (which we shall also refer to as involutive algebras). Let $\M_n(k)$ denote the algebra of $n \times n$ matrices with entries in $k$. Let $\sigma$ be an involution on $\M_n(k)$ (a choice of which we fix for the rest of this section). For any $A\,\in\,\iAlg_k$, it turns out that there is an affine scheme $\Rep^{\ob}_n(A)$ parametrizing {\it involution preserving} representations of $A$ in $\M_n(k)$. To be precise, the involution $\sigma$ induces an involution on $\M_n(B)$ for any $B \,\in \,\cAlg_k$. One therefore obtains a functor
$$\M_n(\mbox{--})\,:\,\cAlg_k \rar \iAlg_k\,,\,\,\,\, B \mapsto \M_n(B)\,\text{.}$$
The functor $\M_n(\mbox{--})$ can be shown to have a left adjoint $(\,\mbox{--}\,)^{\ob}_n\,:\,\iAlg_k \rar \cAlg_k$ (see Lemma \ref{prep1} below). In particular, $A^{\ob}_n$ represents the functor
$$\Rep_n^{\ob}(A)\,:\,\cAlg_k \rar \mathtt{Sets}\,,\,\,\,\, B \mapsto \Hom_{\iAlg_k}(A, \M_n(B))\,\text{.}$$
In other words, $A^{\ob}_n$ is the algebra $k[\Rep^{\ob}_n(A)]$ of functions on the affine scheme $\Rep^{\ob}_n(A)$.

\subsection{The derived representation functor} \la{SDRep}

The notion of an involutive algebra extends to the differential graded (DG for short) setting: an {\it involution} on a DG algebra $A$ is a $k$-linear map $\sigma\,:\, A \rar A$ such that $\sigma^2\,=\,\id$ and $\sigma(ab)\,=\,{(-1)}^{|a||b|} \sigma(b)\sigma(a)$. More concisely,
an involution on $A$ is simply a DG algebra isomorphism $\sigma\,:\,A \rar A^{\mathrm{op}}$ such that
$$ \sigma^{\mathrm{op}} \circ \sigma\,=\,\id \,\text{.}$$
Denote the category of (unital) DG algebras (resp., $k$-augmented DG algebras) with involution by $\iDGA_k$ (resp., $\iDGA_{k/k}$).

\subsubsection{The matrix algebra functor and its adjoint}

 For an involutive DG algebra $S$, let $\iDGA_S$ denote the category of
morphisms $S \rar A$ of involutive DGAs. Consider the natural functor
\begin{equation} \la{G} \M_n(\mbox{--}) \,:\, \iDGA_k \rar \iDGA_{\M_n(k)}\,,\,\,\,\,\, A \mapsto \M_n(A)\simeq M_n(k)\otimes A \,,\,\,\,\,\, a\in A \mapsto I_n\otimes a \in M_n(k)\otimes A \,\text{.}\end{equation}
Here, the involution on $\M_n(A)$ is the tensor product of the involutions on $\M_n(k)$ and $A$ respectively. We abuse notation and use the symbol $\sigma$ to denote the
involution on $A$ as well as the involutions on $\M_n(k)$ and $\M_n(A)$.

\begin{lemma} \la{eqcat}
 The functor~\eqref{G} is an equivalence of categories.
\end{lemma}

\begin{proof}

We claim that the inverse of $\M_n(\mbox{--})$ is the functor
\begin{equation} \la{Ginv} (\mbox{--})^{\M_n(k)} \,:\, \iDGA_{\M_n(k)} \rar
\iDGA_k\,,\,\,\,\, A \mapsto A^{\M_n(k)} \,\text{.}\end{equation}
Here, $A^{\M_n(k)}$ denotes the DG subalgebra of elements in $A$ commuting with all elements in $\M_n(k)$. Observe that $(\mbox{--})^{\M_n(k)}$ is well defined: if $A \in \iDGA_{\M_n(k)}$ with involution $\sigma$, $\sigma$ induces
an involution on $A^{\M_n(k)}$: indeed, if $[m\,,\,a]=0$ for all $m \in \M_n(k)$, $[\sigma(m)\,,\,\sigma(a)]\,=\,-\sigma([m\,,\,a])=0$
for all $m \in \M_n(k)$. Hence, $a \in A^{\M_n(k)}$ implies that $\sigma(a) \in A^{\M_n(k)}$. 
That $(\mbox{--})^{\M_n(k)} \circ \M_n(\mbox{--}) \,=\,\id$
amounts to checking that $(\M_n(B))^{\M_n(k)} \cong B$. The left hand side is n by n matrices $b=\Sigma b_{ij}$ with coefficients in $B$ such that $[b,e_{k\ell}]=0$ for every elementary n by n matrix $e_{k\ell}$. For any $k\not = \ell$, $[b,e_{k\ell}]=0$ implies $b_{kk}=b_{\ell\ell}$ and $b_{k\ell}=0$, showing that $b$ can be written as $b_0\otimes \id\in B\otimes \M_n(k)$. This subalgebra is easily seen to be isomorphic to B. 
To show that $\M_n(\mbox{--}) \circ (\mbox{--})^{\M_n(k)} \,=\,\id$, we need to check that the natural map
$$ \varphi\,:\,\M_n(k) \otimes A^{\M_n(k)} \rar A\,,\,\,\,\, m \otimes a \mapsto \iota(m).a $$
is an isomorphism of involutive DG-algebras. That it is an isomorphism of DG-algebras is shown in the proof
of~\cite[Lemma 2.1]{BKR}. Further,
$$\sigma(\iota(m).a)\,=\, \sigma(a)\sigma(\iota(m)) \,=\, \sigma(a)\iota(\sigma(m)) \,=\, \iota(\sigma(m))\sigma(a) \,,$$
the last equality because $a \in A^{\M_n(k)}$. Hence, $\varphi$ is an isomorphism of involutive DG-algebras as desired.
\end{proof}

Let $S$ be an involutive DGA. Fix a homomorphism $\rho:S \rar \M_n(k)$ of involutive DGAs. This defines a natural functor $\rho^*\,:\,\iDGA_{\M_n(k)} \rar \iDGA_S$ and exhibits $M_n(k)$ as an element of $\iDGA_S$ by pre-composition. Hence the coproduct $(\mbox{--}) \ast_S \M_n(k)$ is well-defined as an element of $\iDGA_{k}$. We also note that is it naturally an involutive algebra over $\M_n(k)$, so it is well-defined in  $\iDGA_{\M_n(k)}$.

\begin{lemma} \la{freeprod}
The functors
$$ (\mbox{--}) \ast_S \M_n(k)\,:\, \iDGA_S \rightleftarrows \iDGA_{\M_n(k)}  \,:\,\rho^* $$
form an adjoint pair.
\end{lemma}

\begin{proof}

Given $A\in \iDGA_{S}$ and $B\in \iDGA_{\M_n(k)}$, with structure maps $P: S\to A$ and $Q: \M_n(k)\to B$, we want to show that there is a natural bijection of sets $$\Hom_{\iDGA_{\M_n(k)}}  (A\ast_S M_n(k),B)= \Hom_{\iDGA_S}(A,\rho^*(B))\ .$$

An element $f\in \Hom_{\iDGA_S}(A,\rho^*(B))$ is a map $f\in \Hom_{\iDGA_k}(A,B)$ such that $f\circ P = Q\circ \rho$. Since $A\ast_S \M_n(k)$ is the coproduct, a pair of maps $f: A\to B$ and $Q: \M_n(k)\to B$ such that $f\circ p=Q\circ \rho$ determine a unique map $A\ast_S \M_n(k)\to B$. This yields a natural map of sets 

\begin{equation} \la{forw} \Hom_{\iDGA_S}(A,\rho^*(B))\to \Hom_{\iDGA_{\M_n(k)}}(A\ast_S\M_n(k),B)\ . \end{equation}

Conversely, for any $\p \in \Hom_{\iDGA_{\M_n(k)}}(A \ast_S \M_n(k), B)$, its restriction to $A$ via the homomorphism $A \to A \ast_S \M_n(k)$ in $\iDGA_S$ determines a unique homomorphism of involutive DG algebras $A \to B$ over $S$. This gives a natural map of sets 

\begin{equation} \la{back}\Hom_{\iDGA_{\M_n(k)}}(A\ast_S\M_n(k),B) \to \Hom_{\iDGA_S}(A,\rho^*(B))\ . \end{equation}

It is easy to check that the maps \eqref{forw} and \eqref{back} are mutual inverses. 


\end{proof}

  Let $\cDGA^{\ob}_k$ denote the category of commutative DGAs with involution. We now introduce the functors

\begin{eqnarray}
\la{r1} \sqrt[n]{S \backslash -}\,:\, \iDGA_S \rar \iDGA_k \,, \,\,\,\, S \backslash A \,\mapsto\, (A \ast_S \M_n(k))^{\M_n(k)}\\
 \la{r2} (S \backslash -)^{\ob}_{n}\,:\, \iDGA_S \rar \cDGA_k\,,\,\,\,\, S \backslash A \,\mapsto\, (\sqrt[n]{ S \backslash A})_{{{\mathrm{ab}}},{\sigma}}
\end{eqnarray}
Here,
$$(\mbox{--})_{{{\mathrm{ab}}}, \sigma}\,:=\, (\mbox{--})_{\sigma} \circ (\mbox{--})_{{\mathrm{ab}}}$$
where
$(\mbox{--})_{{\mathrm{ab}}}\,:\,\iDGA_k \rar \cDGA^{\ob}_k $ denotes Abelianization (i.e, taking quotient by the two sided commutator ideal) and $(\mbox{--})_{\sigma} \,:\,
\cDGA^{\ob}_k \rar \cDGA_k$ denotes quotient by the ideal generated by the image of $\id -\sigma$. For $B \in \cDGA_k$, equip $\M_n(B)$ with the involution induced by the given involution on $\M_n(k)$ and the trivial involution on $B$. \
This gives a functor $$\M_n(\mbox{--})\,:\, \cDGA_k \rar\iDGA_S\,,\,\,\, B \mapsto \M_n(B)\,\text{.}$$ 

\begin{proposition} \la{prep}\la{S2P1}
$\mathrm{(a)}$ The functors $$ \sqrt[n]{S \backslash \mbox{--}}\,:\,\iDGA_S \rightleftarrows \iDGA_k \,:\, \M_n(\mbox{--})$$
form an adjoint pair.

$\mathrm{(b)}$ The functors
  $$ (S \backslash -)^{\ob}_{n}\,:\, \iDGA_S \rightleftarrows \cDGA_k \,:\, \M_n(\mbox{--}) $$
  form an adjoint pair.
\end{proposition}

\begin{proof}

The functor $\M_n(\mbox{--}): \iDGA_k \rar \iDGA_S$ is the composition of the functors
$$\begin{diagram} \iDGA_k & \rTo^{\M_n(\mbox{--})} & \iDGA_{\M_n(k)} & \rTo^{\rho^*} & \iDGA_S \end{diagram}\,\text{.}$$
By lemma~\ref{freeprod}, the left adjoint of the second functor in the above composition is the functor $(\mbox{--}) \ast_S \M_n(k)\,:\,\iDGA_S \rar \iDGA_{\M_n(k)}$. By Lemma~\ref{eqcat}, the inverse (and hence, left adjoint) of the first functor in the above composition is the functor $(\mbox{--})^{\M_n(k)}$. Hence, the left adjoint of $\M_n(\mbox{--})\,:\,\iDGA_k \rar \iDGA_S$ is the composition of functors
$$\begin{diagram} \iDGA_S & \rTo^{(\mbox{--}) \ast_S \M_n(k)} & \iDGA_{\M_n(k)} & \rTo^{(\mbox{--})^{\M_n(k)}} & \iDGA_k \end{diagram}\,,$$
which is exactly the functor $\sqrt[n]{S \backslash \mbox{--}}$. This proves (a).

It remains to show (b). The functor $\M_n(\mbox{--})\,:\,\cDGA_k \rar \iDGA_S$ is the composition of functors
$$ \begin{diagram} \cDGA_k & \rTo^{(1)} & \cDGA^{\ob} & \rTo^{(2)} & \iDGA_k & \rTo^{\M_n(\mbox{--})} & \iDGA_S \end{diagram}\,,$$
where the functor labelled $(1)$ takes $B \in \cDGA_k$ to $B$ equipped with the trivial involution and where the functor labelled $(2)$ is the forgetful functor. The left adjoint of $(2)$ is the functor $(\mbox{--})_{\mathrm{ab}}$. The left adjoint of $(1)$ is $(\mbox{--})_{\sigma}$. Hence, (a) implies that the left adjoint of the functor $\M_n(\mbox{--})\,:\cDGA_k \rar \iDGA_S$ is the functor
$$ (S \backslash -)^{\ob}_{n}\,:=\, (\mbox{--})_{\sigma} \circ (\mbox{--})_{\mathrm{ab}} \circ \sqrt[n]{S \backslash \mbox{--}} \,\text{.}$$ This proves (b).
\end{proof}

Note that if $S \in \iAlg_k$, then the functors $\sqrt[n]{S \backslash \mbox{--}}: \iDGA_S \rar \iDGA_k$ and $(S \backslash \mbox{--})_n^{\ob}:\iDGA_S \rar \cDGA_k$ restrict to the corresponding categories of `honest' algebras, yielding functors 
$\sqrt[n]{S \backslash \mbox{--}}:\iAlg_S \to \iAlg_k$ and $(S \backslash \mbox{--})_n^{\ob}:\iAlg_S \to \cAlg_k$ respectively. Since $\iAlg_S$ (resp., $\iAlg_k$, $\cAlg_k$) is a full subcategory of $\iDGA_S$ (resp., $\iDGA_k$, $\cDGA_k$), Proposition \ref{prep} immediately implies

\begin{lemma} \la{prep1}
$\mathrm{(a)}$ The functors $$ \sqrt[n]{S \backslash \mbox{--}}\,:\,\iAlg_S \rightleftarrows \iAlg_k \,:\, \M_n(\mbox{--})$$
form an adjoint pair.

$\mathrm{(b)}$ The functors
  $$ (S \backslash -)^{\ob}_{n}\,:\, \iAlg_S \rightleftarrows \cAlg_k \,:\, \M_n(\mbox{--}) $$
  form an adjoint pair.

\end{lemma}

\subsubsection{Model structure} \la{MoStr}

By \cite[Section 5.2]{B14} (also see \cite[Section 3.1]{B12}), the category $\iDGA_S$ is the category of algebras over an operad (over $k[\Z_2]$, hence over $k$). Hence, by \cite[Theorem 4.1.1]{H}, $\iDGA_k$ has a model structure whose weak equivalences are quasi-isomorphisms and whose fibrations are degreewise surjections (hence, for any $S \in \iDGA_k$, $\iDGA_S$ acquires the inherited model structure). The category $\cDGA_k$ is equipped with a similar model structure. The proof of \cite[Prop. B.2]{BKR} goes through with minor modifications in the involutive setting. For later use in this paper, we recall below the following analog of a part of {\it loc. cit.}. Letting $\Omega := k[t,dt]$ with $|t|=0$ and $d(t):=dt$ denote the de-Rham algebra of the affine line $\mathbb A_k^1$ (equipped with the trivial involution). For any map of complexes $H:C \to D \otimes \Omega$, we let $H(a):C \to D$ denote the composition of $H$ with the quotient map $D \otimes \Omega \onto D \otimes \Omega/(t-a)$ for any $a \in k$. Clearly, $H(0)$ and $H(a)$ are homotopic as maps of complexes for any $a \in k$ (see \cite[Remark B.4.4]{BKR}). 
\bprop
\la{lhom}
Let $ f,\, g:\, A \to B $ be two morphisms in $ \iDGA_S $, where $A$ is a cofibrant object in $ \iDGA_S $.
The following conditions are equivalent.\\
$(i)$ $\, f \sim g\,$ in $ \iDGA_S \, $.\\
$(ii)$\ There exists a morphism $h:A \rar B \otimes \Omega$ in $\iDGA_S $ such that $h(0)=f$ and $h(1)=g$.\\
\eprop
In addition, the proof of \cite[Prop. B.3]{BKR} goes through with minor modifications to show that
\bprop \la{dgproper}
The model category $\iDGA_k$ is proper.
\eprop

We now return to the adjunction of Lemma \ref{prep1} $(b)$. Clearly, the functor $\M_n(\mbox{--})\,:\,\cDGA_k \rar \iDGA_S$ preserves fibrations (degreewise surjections) and weak equivalences (quasi-isomorphisms).

\begin{theorem} \la{trep}
 (a) The functors $(S\backslash \mbox{--})^{\ob}_n\,:\,  \iDGA_S \rightleftarrows \cDGA_k \,:\, \M_n(-)$ form a Quillen pair.\\
  (b) The functor $(S \backslash \mbox{--})^{\ob}_n$ has a total left derived functor
      $$\L(S \backslash \mbox{--})^{\ob}_n\,:\,\Ho(\iDGA_S) \rar \Ho(\cDGA_k) \,,\,\,\,\, (S \backslash A) \mapsto Q(S \backslash A)_n\,,\,\,\,\, \gamma f \mapsto \gamma(\tilde{f}^{\ob}_n) \,\text{.}$$
       Here, $Q(S \backslash A)$ denotes any choice of cofibrant resolution of $(S \backslash A)$ in $\iDGA_S$ and
$\tilde{f}$ denotes any cofibrant lifting of $f$ in $\iDGA_S$. \\
(c) For any $S \backslash A$ in $\iDGA_S$ and for any $B \in \cDGA_k$,
\begin{equation*}
\la{S2E15}
\Hom_{\Ho(\cDGA_k)}(\L(S \bs A)^{\ob}_n,\, B) \cong
\Hom_{\Ho(\iDGA_S)}(A,\M_n(B))\ .
\end{equation*}
\end{theorem}

\begin{proof}
Since the functor $\M_n(\mbox{--})$ preserves fibrations and weak equivalences, Lemma A.1 of~\cite{BKR} implies that the functors $(S \backslash \mbox{--})^{\ob}_n\,:\,  \iDGA_S \rightleftarrows \cDGA_k \,:\, \M_n(-)$ form a Quillen
pair. Quillen's fundamental theorem (see Theorem A.2 of~\cite{BKR}) then implies parts (b), (c) of the desired theorem.
\end{proof}

\begin{definition}
We define the (relative) {\it derived representation scheme} (parametrizing involution preserving representations) of $A$ by
$$\DRep^{\ob}_n(S \backslash A)\,:=\, \L(S \bs A)^{\ob}_n \,\in\,\Ho(\cDGA_k) $$
and define the (relative) {\it representation homology} by
$$\mathrm{HR}^{\ob}_{\bullet}(S \bs A, k^n)\, :=\,\mathrm{H}_{\bullet}(\L(S \bs A)^{\ob}_n)\,\text{.}$$
\end{definition}

When $S\,=\,k$, we omit the notation $S \bs \mbox{--}$ and denote the functors $(S \bs \mbox{--})^{\ob}_n$, $\L(S \bs \mbox{--})^{\ob}_n$ and $\mathrm{HR}^{\ob}_{\bullet}(S \bs \mbox{--}, k^n)$  by $(\,\mbox{--}\,)^{\ob}_n$,  $\L(\mbox{--})^{\ob}_n$ and  $\mathrm{HR}^{\ob}_{\bullet}(\mbox{--}, k^n)$ respectively.

\subsubsection{Augmented involutive DGAs}

Let $\iDGA_{k/k}$ denote the category of involutive DG $k$-algebras that are {\it augmented} over $k$. Let $A \,\in\, \iDGA_{k/k}$. Applying the functor $(\mbox{--})^{\ob}_n$ to the augmentation morphism $\varepsilon_A\,:\, A \rar k$, one sees that $A^{\ob}_n$ is also augmented over $k$. The functor $(\mbox{--})^{\ob}_n$ therefore restricts to a functor 
\begin{equation} \la{augrepfunct} (\mbox{--})^{\ob}_n\,:\,\iDGA_{k/k} \rar \cDGA_{k/k}\,,\,\,\,\, A \mapsto A^{\ob}_n \,\text{.} \end{equation} 
The right adjoint of the functor~\eqref{augrepfunct} is the functor 
\begin{equation} \la{mataug} \bar{\M}_n(\mbox{--})\,:\,\cDGA_{k/k} \rar \iDGA_{k/k}\,,\,\,\,\, B \mapsto \bar{\M}_n(B)\,:=\,k \oplus \M_n(\bar{B}) \,,\end{equation}
where the involution on the augmentation ideal of $\bar{\M}_n(B)$ is induced by the chosen involution on $\M_n(k)$.

The categories $\iDGA_{k/k}$ and $\cDGA_{k/k}$ have model structures arising out of the model structures on $\iDGA_k$ and $\cDGA_k$   respectively. A morphism is a cofibration (resp., fibration/weak equivalence) in $\iDGA_{k/k}$ (resp., $\cDGA_{k/k}$) iff it is a cofibration (resp., fibration/weak equivalence) in $\iDGA_{k}$ (resp., $\cDGA_{k}$). Since the functor $\bar{\M}_n(\mbox{--})$ clearly preserves weak equivalences (quasi-isomorphisms) and fibrations (degreewise surjections), the adjoint pair 
$$ (\,\mbox{--}\,)^{\ob}_n\,:\,\iDGA_{k/k} \leftrightarrows \cDGA_{k/k} \,:\,\bar{\M}_n(\mbox{--}) $$ 
is Quillen. Consequently, the functor $(\mbox{--})^{\ob}_n\,:\,\iDGA_{k/k} \rar \cDGA_{k/k}$
 has a (total) left derived functor 
 $$\L(\mbox{--})^{\ob}_n\,:\,\Ho(\iDGA_{k/k}) \rar \Ho(\cDGA_{k/k})\,,\,\,\,\, A \mapsto (QA)^{\ob}_n\,,$$
 where $QA \stackrel{\sim}{\rar} A$ is any cofibrant resolution of $A$ in $\iDGA_{k/k}$. Since any cofibrant resolution of $A$ in $\iDGA_k$ is also a cofibrant resolution in $\iDGA_{k/k}$, the functor $\L(\mbox{--})^{\ob}_n$ is just the ``restriction" of the corresponding functor in the unaugmented setting. We therefore, continue to denote 
 $\L A^{\ob}_n\,\in\,\Ho(\cDGA_{k/k})$ by $\DRep^{\ob}_n(A)$ and leave other related notation unchanged in the augmented setting.

\subsection{Basic properties}

We now discuss certain basic properties of derived representation functors for involutive DG algebras.

\subsubsection{Base change and a cofibration spectral sequence}

Suppose that $S \stackrel{f}{\rar} A \stackrel{g}{\rar} B$ in $\iDGA_k$. Fix an involution preserving algebra homomorphism $\rho_A\,:\, A \rar \M_n(k)$. Let $\rho:=\rho_A \circ f$. The proof of Theorem 2.3 of~\cite{BKR} goes through word for word (modulo the obvious minor changes) to give the following theorem. After the statement of the theorem we reproduce a lemma from \cite{BKR} and provide the proofs of both results for the reader's convenience.

\begin{theorem}
\la{hopushout}
There is a homotopy commutative diagram
\begin{equation}
\la{drepd2}
\begin{diagram}[small, tight]
\L(S \bs A)^{\ob}_n           &  \rTo^{\,\L(f)_n^{\ob}\,}        & \L(S \bs B)^{\ob}_n\\
\dTo^{\L(\rho)_n} &                    & \dTo \\
k                &  \rInto            & \L( A \bs B)^{\ob}_n
\end{diagram}
\end{equation}
which is a cocartesian square in $ \Ho(\cDGA_k) $.
\end{theorem}
\blemma
\la{S2L9}
The commutative diagram
\begin{equation}
\la{drepd11}
\begin{diagram}[small, tight]
(S\bs A)^{\ob}_n               &  \rTo^{\,(f)^{\ob}_n  \,}            & (S\bs B)^{\ob}_n  \\
\dTo^{(\varrho)^{\ob}_n  } &                                   & \dTo \\
k                       &  \rInto                           & (A\bs B)^{\ob}_n  
\end{diagram}
\end{equation}
is a cocartesian square in $ \cDGA_k $.
\elemma
\bproof
To simplify the notation set $ \C := \cDGA_k $ and denote by $ \C^\D $ the category of diagrams of shape $ \{\ast \leftarrow \ast \rightarrow \ast \}\,$ in $ \C $. The pushout construction defines a functor $\, \colim:\,\C^\D \to \C \,$ which is left adjoint to the diagonal functor
$ \Delta:\,\C \to \C^\D $ ({\it cf.} Example A.2 in \cite{BKR}). We need to show that
\begin{equation}
\la{coliso}
(A\bs B)^{\ob}_n   \cong \colim[\,k \xleftarrow{\,\varrho^{\ob}_n  \,} (S\bs A)^{\ob}_n    \xrightarrow{\,f^{\ob}_n  \,} (S\bs B)^{\ob}_n   \,]\ .
\end{equation}
To this end, we denote the above colimit by $ \colim(\varrho^{\ob}_n  ,\,f^{\ob}_n  ) $,
choose an arbitrary $ C \in \cDGA_k $ and identify
\begin{eqnarray*}
\Hom_{\C}(\colim(\varrho^{\ob}_n  ,\,f^{\ob}_n  ),\,C) \,\, & \cong & \Hom_{\C^{\D}}((\varrho^{\ob}_n  ,\,f^{\ob}_n  ),\,\Delta C)\\
                          &\cong  & \{\alpha \in \Hom_{\C}((S \bs B)^{\ob}_n  ,\,C)\ :\ \alpha\,f^{\ob}_n   = \varrho^{\ob}_n  \}\\
                          &\cong& \{\tilde{\alpha} \in \Hom_{\iDGA_S}(B,\,\M_n(k) \otimes C)\ :\ \tilde{\alpha}\,f = \varrho\}\\
                          &\cong& \Hom_{\DGA_A^{\ob}}(B,\,\M_n(k) \otimes C) \\
                          &\cong& \Hom_{\cDGA_k}((A \bs B)^{\ob}_n  ,\,C)\ ,
\end{eqnarray*}
where the third and the last bijections are given by Proposition~\ref{prep}$(b)$.
The isomorphism \eqref{coliso} follows now by Yoneda's Lemma.
\eproof

\begin{proof}[Proof of Theorem ~\ref{hopushout}]
By Lemma~\ref{S2L9}, there is an isomorphism of functors $\, \DGA_A^{\ob} \to \cDGA_k \,$:
\begin{equation}\la{lccol0}
(A\bs\, \mbox{--}\,)^{\ob}_n    \cong \colim[\,k \xleftarrow{} (S\bs A)^{\ob}_n   \xrightarrow{\,(S\bs\, \mbox{--}\,)^{\ob}_n  \,} (S\bs \, \mbox{--}\,)^{\ob}_n   \,]\ .
\end{equation}
By uniqueness of the derived functors, \eqref{lccol0} induces an isomorphism of
functors $\, \Ho(\DGA_A^{\ob}) \to \Ho(\cDGA_k)
\,$:
\begin{equation}\la{lccol}
\L(A\bs \, \mbox{--}\,)^{\ob}_n    \cong \L\colim[\,k \xleftarrow{\,\varrho^{\ob}_n  \,} (S\bs A)^{\ob}_n  \xrightarrow{\,(S\bs\, \mbox{--}\,)^{\ob}_n  \,} (S\bs \, \mbox{--}\,)^{\ob}_n   \,] \ .
\end{equation}
Now, to compute $\,\L \colim \,$ in \eqref{lccol} we should replace $ f:\, A \to B $ by
its cofibrant resolution $ i:\, A \to Q(A\bs B) $ in $ \iDGA_A $; then
\begin{equation}
\la{lccol1}
\L\colim[\,k \xleftarrow{\,\varrho^{\ob}_n  \,} (S\bs A)^{\ob}_n    \xrightarrow{\,f^{\ob}_n  \,} (S\bs B)^{\ob}_n   \,]
\cong \gamma\,\colim[\,k \xleftarrow{\,\varrho^{\ob}_n  \,} (S \bs A)^{\ob}_n    \xrightarrow{\,i^{\ob}_n  \,} (S\bs Q(A \bs B))^{\ob}_n  \,]\ ,
\end{equation}
where the last `$ \colim $' is taken in $ \cDGA_k $ and $\,\gamma: \cDGA_k \to \Ho(\cDGA_k) \,$.

We will use a specific resolution $ i:\, A \to Q(A\bs B) $ in \eqref{lccol1}, which we construct in the following way.
First, we choose a cofibrant resolution $\,p_{S\bs A}:\, Q(S\bs A) \sonto A $ of $ A $ in $\,\iDGA_S\,$.
Then, by ~\cite[Prop. 3.1]{FHT}, there is a cofibration
$\,\tilde{f}:\, Q(S\bs A) \into  Q(S \bs B) \,$ in $ \DGA_k^{\ob} $ and a surjective quasi-isomorphism
$\,p_{S\bs B}:\, Q(S\bs B) \sonto B $ such that $\, p_{S\bs B} \circ \tilde{f} =
f \circ p_{S\bs A} \,$. We take $ p_{S\bs B} $ to be a cofibrant resolution of $ B $ in $ \iDGA_S$.
Now, we define $\, i: A \to  Q(A\bs B) $ by pushing out $ \tilde{f} $ along $ p_{S\bs A} $ in $ \iDGA_S \,$:
\begin{equation}
\la{drepd3}
\begin{diagram}[small, tight]
Q(S \bs A)                  &  \rInto^{\ \tilde{f}\ }           & Q(S \bs B) \\
\dOnto^{\,p_{S \bs A}\,}   &                                   & \dOnto^{\tilde{\varphi}} \\
A                           &  \rInto^{i\ }                     & Q(A\bs B)
\end{diagram}
\end{equation}
By the universal property of pushouts, there is a (unique) morphism $\,p_{A \bs B}:\,Q(A\bs B) \to B \,$ 
in $ \iDGA_S $ such that $\,p_{A \bs B} \circ i = f \,$ and  $\,p_{A\bs B} \circ \tilde{\varphi} = p_{S\bs B}\,$. We claim that $ p_{A \bs B} $ is a cofibrant resolution 
of $ f $ in $ \iDGA_A $. Indeed, by \cite[Prop. 3.14]{DS}, $\,i\,$ is a cofibration in $ \iDGA_k $, 
so $ Q(S \bs A) $ is a cofibrant object in $ \iDGA_S $. On the other hand,  $\, \tilde{\varphi} \,$  is a pushout of a quasi-isomorphism along a cofibration in $ \iDGA_k $, hence, by Proposition~\ref{dgproper}, this map is a quasi-isomorphism as 
well. Finally, $\,p_{A\bs B} \circ \tilde{\varphi} = p_{S\bs B}\,$ implies that
$ p_{A\bs B} $ is a surjective quasi-isomorphism, since so is $ p_{S \bs B} $.

Now, since $\,(S\bs \, \mbox{--}\,)^{\ob}_n  : \iDGA_S \to \cDGA_k\,$ is a left adjoint functor,
it preserves colimits: hence, it follows from \eqref{drepd3} that
\begin{equation}
\la{drepd33}
(S \bs Q(A\bs B))^{\ob}_n    \cong \colim[\,(S\bs A)^{\ob}_n   \xleftarrow{} Q(S\bs A)^{\ob}_n   \xrightarrow{\tilde{f}^{\ob}_n  } Q(S\bs B)^{\ob}_n   \,]
\end{equation}
On the other hand, since $\,(S\bs \, \mbox{--}\,)^{\ob}_n   \,$ is a left Quillen functor, it preserves cofibrations:
hence $ \tilde{f}^{\ob}_n   $ in \eqref{drepd33} is a cofibration in $ \cDGA_k $ since so is $ \tilde{f} $ in \eqref{drepd3}.
Combining \eqref{lccol}, \eqref{lccol1} and \eqref{drepd33}, we get
\begin{eqnarray*}
\L(A\bs B\,)^{\ob}_n  
&\cong&
\gamma\,\colim[\,k \xleftarrow{ \tilde{\varrho}^{\ob}_n   } Q(S\bs A)^{\ob}_n    \xrightarrow{\tilde{f}^{\ob}_n  } Q(S\bs B)^{\ob}_n   \,] \nonumber \\*[1ex]
&\cong&
\colim[\,k \xleftarrow{ \gamma \tilde{\varrho}^{\ob}_n   } \gamma\,Q(S\bs A)^{\ob}_n  
\xrightarrow{\gamma \tilde{f}^{\ob}_n  } \gamma\, Q(S\bs B)^{\ob}_n   \,] \la{drepd34} \\*[1ex]
&\cong &
\colim[\,k \xleftarrow{\,\L(\varrho)^{\ob}_n  \,} \L(S\bs A)^{\ob}_n    \xrightarrow{\,\L(f)^{\ob}_n  \,} \L(S\bs B)^{\ob}_n   \,] \la{drepd35} \ ,
\end{eqnarray*}
where the second isomorphism is a consequence of ~\cite[Prop. A.2]{BKR}, and the third holds by Theorem~\ref{trep}.
Note that ~\cite[Prop. A.2]{BKR} applies in our situation, since $ \cDGA_k $ is a proper model category
({\it cf.} Proposition~\ref{dgproper}). 
\end{proof}

Let ${\iDGA_k}^+$ denote the model category of nonnegatively graded involutive DGAs. The proof of Theorem 2.4 of~\cite{BKR} also goes through to give us the following result.

\begin{theorem}
\la{EMsps}
Given $\, S \xrightarrow{} A \xrightarrow{} B\,$ in ${\iDGA_k}^+$ and a representation $ A \to \M_n(k) $,
there is an Eilenberg-Moore spectral sequence with
$$
E^2_{\ast,\,\ast} = \Tor_{\ast,\,\ast}^{\mathrm{HR}^{\ob}_\bullet(S\bs A, k^n)}(k,\,\mathrm{HR}^{\ob}_\bullet(S \bs B, k^n))
$$
converging to $ \mathrm{HR}^{\ob}_{\bullet}(A\bs B, k^n) $.
\end{theorem}

\begin{proof}
To construct this spectral sequence we will use the cofibrant resolutions
constructed in the proof of Theorem~\ref{hopushout}.
In addition, we will assume that $ Q(R\bs S) $ is chosen to be almost free
in $ \DGA_R^{\ob,+} $ ({\it cf.}  \cite[Prop. 3.1]{FHT}).
Then the pushout diagram \eqref{drepd3} gives the pushout
in the category $ \cDGA_k^+ $:
\begin{equation}
\la{dred}
\begin{diagram}[small, tight]
Q(R\bs S)^{\ob}_n     &  \rTo^{\tilde{f}^{\ob}_n  } & Q(R \bs A)^{\ob}_n   \\
\dTo^{\tilde{\varrho}^{\ob}_n  } &          & \dTo \\
k             &  \rInto             & Q(S \bs A)^{\ob}_n  
\end{diagram}
\end{equation}
where $\, \tilde{\varrho}^{\ob}_n   :\,Q(R\bs S)^{\ob}_n   \to  k \,$ is the map adjoint to
$\,\tilde{\varrho} := \varrho \circ p_{R\bs S}: Q(R\bs S) \to \M_n(k) \,$.
It follows from \eqref{dred} that
$$
k \otimes_{Q(R \bs S )^{\ob}_n  } Q(R \bs A)^{\ob}_n   \cong Q(S\bs A)^{\ob}_n   \ .
$$
Moreover, for the almost free DG resolutions $\, Q(R\bs S),\,Q(R \bs A)\,$ and $\, Q(S \bs A)\,$,
the commutative DG algebras $ Q(R\bs S)^{\ob}_n  ,\,Q(R \bs A)^{\ob}_n   $ and $ Q(S \bs A)^{\ob}_n   $
are almost free. The existence of the Eilenberg-Moore spectral is now standard (see, e.g.,
\cite{McC}, Theorem~7.6).
\end{proof}

\subsubsection{$0$-th homology}

Suppose that $S, A \,\in\,\iAlg_k$. The following proposition shows that $\DRep^{\ob}_n(S \bs A)$ is indeed the ``higher derived functor" of $\Rep^{\ob}_n(S \bs A)$ from the homological point of view. Its proof coincides with that of~\cite[Theorem 2.5]{BKR} modulo a few modifications.

\begin{proposition} \la{zerohom}
$$ {\rm HR}^{\ob}_0(S \bs A, k^n)\,\cong\, k[\Rep^{\ob}_n(S \bs A)]\,\text{.}$$
\end{proposition}

\begin{proof}

Fix a cofibrant resolution $Q(S\bs A) \to S\bs A \in \iDGA_S$. Then, for any $C\in \cAlg_k^*$, there are natural isomorphisms

\begin{eqnarray*}
\Hom_{\cAlg_k^*}((S\bs A)_n,\,C) \,\,& \cong & \Hom_{\iAlg_S}(A,\,\M_n(k) \otimes C)\\
                          &\cong  & \Hom_{\DGA^{\ast,+}_S}(Q(S \bs A),\, \M_n(k) \otimes C)\\
                          &\cong& \Hom_{\cDGA^{\ast,+}_k}(Q(S \bs A)_n ,\, C)\\
                          &\cong& \Hom_{\cAlg_k^*}(\H_0[Q(S \bs A)_n],\, C)\ ,
\end{eqnarray*}
The first isomorphism is a degree 0 (underived) application of \ref{prep}(a).
The second isomorphism was due to the fact that $\M_n(k) \otimes C$ is concentrated in degree 0.
The third isomorphism is \ref{prep}(b)
The fourth ismorphism is again a statement about $C$ being concentrated in degree 0.

\end{proof}

\subsubsection{Invariants} \la{secinv}
Let $G$ be the subgroup of $\GL_n(k)$ consisting of matrices $g$ satisfying $\sigma(g^{-1})=\lambda g$ for some scalar
$\lambda$ (depending on $g$), with $\sigma$ the specified involution. Note that for any $A \in \iDGA_k$, $G$ acts on $A_n$: the automorphism corresponding to
$g \in G$ is adjoint to the composite map
$$ A \stackrel{\pi}{\rar} \M_n(k) \otimes A \xrightarrow{\text{Ad}_g \otimes \id} \M_n(A) $$
of involutive DGAs. One may therefore, consider the $G$-invariant subfunctor $(-)_n^{\ast,G}\,:\, \iDGA_k \rar \cDGA_k$
of $(-)^{\ob}_n$. Further, the Quillen adjunction in Theorem~\ref{trep} (c) gives a $G$-action on $\L A^{\ob}_n$. One may therefore,
consider the
$G$-invariant subfunctor $[\L (-)^{\ob}_n]^{G}$ of $\L (-)^{\ob}_n$.

Proposition~\ref{lhom} enables us to imitate the proof of~\cite[Theorem 2.6]{BKR} to obtain:

\begin{theorem} \la{inv}\la{S2P4}
$(a)$ The functor $(-)_n^{\ast,G}$ has a total left derived functor
$$
\L[(\,-\,)_n^{\ast, G}]:\,\Ho(\iDGA_k) \to \Ho(\cDGA_k)\ .
$$

$(b)$ There is an isomorphism of functors
$\,
\L[(\,-\,)_n^{\ast,G}] \cong [\L(\,-\,)^{\ob}_n]^{G}
\,$.
 \end{theorem}

\begin{proof}[Proof of Theorem~\ref{inv}]
$(a)$ By ~\cite{BKR} Lemma A.2, it suffices to show that $\,(-)_n^{\ast,G}$ maps any acyclic cofibration $\,i: A \sinto B \,$ between cofibrant objects in $\iDGA_k$ to a weak equivalence in $\cDGA_k$. By ~\cite{BKR} A4.8, given $\,i: A \sinto B \,$, there is a map $\,p:B \rar A\,$ in $\iDGA_k$ such that $\,pi=\id_A\,$ and $ip$ is homotopic to $\id_B$. Hence, $p^{G} \circ i^G= \id_{(A)_n^{\ast,G}}$. By Proposition \ref{lhom}, there is a morphism $\,h:B \rar B \otimes \Omega $ in $\iDGA_k$ with $\,h(0)=\id_B\,$ and $\,h(1)=ip\,$. By the obvious analog (in the involutive setting) of \cite[Lemma 2.5]{BKR}, there is a morphism $h^G:(B)_n^{\ast,G} \rar (B)_n^{\ast,G} \otimes \Omega $ such that $h^G(0)=\id_{(B)_n^{\ast,G}}$ and $h^G(1)=i^G \circ p^G$. Finally, by Proposition~\ref{lhom}, $i^G \circ p^G$ induces the identity map at the level of homology. Hence $i^G$ is a weak equivalence in $\cDGA_k$. This proves $(a)$.

$(b)$ The functor $ \L[(-)_n^{\ast,G}] $ maps $\,\gamma A \in \Ho(\iDGA_k)\,$ to the class of the invariant subalgebra $ (QA)_n^{\ast,G} $ of $ (QA)_n^{\ob}$, where $ QA $ is a cofibrant replacement of $A$. By $(a)$, this map is well-defined, i.e. does not depend on the choice of $ QA $. On the other hand, $ \mathrm{HR}_{\bullet}^{\ob}(\mbox{--})_n^{\ast,G} $ maps $ \gamma A $ to $\mathrm{HR}_{\bullet}^{\ob}( [\gamma(QA)^{\ob}_n])^G $. Since $k$ is of characteristic $0$ and since $\GL(V)$ is reductive, $\,\mathrm{HR}_{\bullet}^*(\gamma[((QA)_n^{\ob})^G]) \cong [\mathrm{HR}_{\bullet}^*(\gamma((QA)_n^{\ob}))]^G \,$ for all $A \in \iDGA_k$. This proves part $(b)$.
\end{proof}
 
 We often denote the homotopy commutative DG algebra $\L A^{\ast, G}_n$ by $\DRep^{\ob}_n(A)^G$.  By Theorem~\ref{inv}, its homology is exactly $\mathrm{HR}_{\bullet}^{\ob}(A, k^n)^G$.

\subsection{The tangent complex}


\subsubsection{}
Fix a (involution preserving) representation $\rho\,:\,A \rar \M_n(k)$. This corresponds to a homomorphism
$\rho\,:\,A_n^* \rar k$ of commutative DG algebras (by Propositon~\ref{prep}). Note that the complex $\mathrm{Der}_{\bullet}^{\rho}(A,\M_n(k))$ of $\rho$-derivations
from $A$ to $\M_n(k)$ is equipped with a $\Z_2$ action given by
$$\tau.D(x)\,=\, \sigma_{\M_n(k)}(D(\sigma_A(x)))\,,\,\,\,\forall\,x\,\in\, A $$
where $\tau$ is the generator of $\Z_2$.

\begin{prop} \la{ptan1}
 The tangent complex to the DG-scheme $\mathrm{Spec }A_n^*$ at the closed point $\rho$ is precisely
the subcomplex $\mathrm{Der}^{\rho}_{\bullet}(A,\M_n(k))^{\Z_2}$ of elements of $\mathrm{Der}^{\rho}_{\bullet}(A,\M_n(k))$
that are fixed by the above $\Z_2$-action.
\end{prop}

\begin{proof}
 Indeed, the required tangent complex is the complex $\mathrm{Der}^{\rho}_{\bullet}(A_n^*,k)$. Note that for a fixed (homological)
degree $i$, $\mathrm{Der}_{-i}^{\rho}(A_n^*,k)\,=\,\Hom_{\cDGA_k}^{\rho}(A_n^*,k[\epsilon]/\epsilon^2)$ where $\epsilon$ is
homogenous of degree $i$ and where the superscript $\rho$ indicates that we consider only those homomorphisms which equal $\rho\,:A_n^* \rar k$
after composition with the augmentation $k[\epsilon]/\epsilon^2 \rar k$.

By Proposition~\ref{prep}, we have the identification
$$ \Hom_{\cDGA_k}^{\rho}(A_n^*,k[\epsilon]/\epsilon^2)\,=\,\Hom_{\iDGA_k}^{\rho}(A,\M_n(k) \otimes k[\epsilon]/\epsilon^2)\,\text{.}$$
On the other hand,
$\Hom_{\iDGA_k}^{\rho}(A,\M_n(k) \otimes k[\epsilon]/\epsilon^2)$ is precisely the set of $\rho$-derivations $D$ of degree $-i$
satisfying
$$ D \circ \sigma_A\,=\,\sigma_{\M_n(k)} \circ D\,\text{.}$$
Thus, $\mathrm{Der}^{\rho}_{\bullet}(A_n^*,k)\,\cong\,\mathrm{Der}^{\rho}_{\bullet}(A,\M_n(k))^{\Z_2}$ as graded vector spaces.
That the above identification commutes with differentials is easily verified.
The desired proposition is now proven.
\end{proof}

 \subsection{Representation homology of simplicial involutive algebras} \la{simpdgcomp}

The constructions we've given above in the DG setting can be mimicked in the simplical setting, thereby defining representation homology of a simplicial involutuve algebra (see below). To compare the simplicial and DG constructions, we need to recall that involutive algebras are algebras over a $k$-linear operad (see Section \ref{MoStr}). This enables us to make use of the Monoidal Dold-Kan correspondance developed in the appendix of ~\cite{BRYII} to prove Proposition \ref{simp} below. 

\begin{definition} \la{defsialg}
A simplicial algebra with involution is a simplicial object in $\iAlg_k$, i.e, a functor $\Delta^{\rm op} \to \iAlg_k$. Equivalently, it is a simplicial associative algebra $A$ equipped with an {\it involution}, i.e., a simplicial algebra isomorphism $\sigma\,:\,A \rar A^{\mathrm{op}}$ such that
$$ \sigma^{\mathrm{op}} \circ \sigma\,=\,\id \,\text{.}$$
Denote the category of (unital) simplicial algebras (resp., $k$-augmented simplicial algebras) with involution by $\s\iAlg_k$ (resp., $\s\iAlg_{k/k}$). Note that $A \in \s\iAlg_k$ if and only if 
\end{definition}

Since $\iAlg_k$ (resp., $\cAlg_k$) is the category of simplicial algebras over an operad, the category $\s\iAlg_k$ (resp., $\s\cAlg_k$) has a model structure in which the fibrations (resp., weak-equivalences) are fibrations (resp., weak-equivalences) of the underlying simplicial sets (see \cite[Sec. II.4, Theorem 4]{Q1}). 

\begin{prop} \la{simpadj}
\la{simpad}   
$(i)$ The adjunction in Lemma~\ref{prep1} extends simplicially to a Quillen adjunction $(\mbox{--})_n^{\ob}: \s\iAlg_k\, \leftrightarrows \, \s\cAlg_k\,:\M_n(\mbox{--})$.\\

$(ii)$ Consequently, there is an adjunction between homotopy categories 
$$\L (\mbox{--})_n^{\ob}: \Ho(\s\iAlg_k)\, \leftrightarrows \, \Ho(\s\cAlg_k): \M_n(\mbox{--})\ .$$
\end{prop}

\begin{proof}
Since $\s\iAlg_k=\Hom_{Cat}(\Delta^{op},\iAlg_k)$ and $\s\cAlg_k=\Hom_{Cat}(\Delta^{op},\cAlg_k)$, the first assertion in Part $(i)$ follows from the fact that for a small category $I$, an adjunction $F\, : C \leftrightarrows D:\, G$ induces an adjunction $F\, : \Hom_{Cat}(I,C) \leftrightarrows \Hom_{Cat}(I,D):\, G$. It remains to verify that this adjunction is Quillen. For this, we check that the right adjoint respects fibrations and acyclic fibrations. Indeed, a morphism $f: A \to B$ in $\s\cAlg_k$ that is a Kan fibration of simplicial sets is sent to $\M_n(f):\M_n(A) \to \M_n(B)$. Since $\M_n(A)$ is naturally isomorphic to $A^{n^2}$ as a simplicial set and since a cartesian product of Kan fibrations is a Kan fibration, $\M_n(f)$ is a Kan fibration of simplicial sets. Thus, $\M_n(\mbox{--})$ preserves fibrations. Since $\M_n(\mbox{--})$ also preserves weak-equivalences, it preserves acyclic fibrations as well. This shows that the above adjunction is Quillen, proving Part $(i)$. Part $(ii)$ follows immediately from Part $(i)$.  
\end{proof}
\begin{definition} \la{defsrep}
For $A\in \s\iAlg_k$, we define $\mathrm{HR}_{\bullet}^{\ob}(A, k^n)^{\s}=\pi_{\bullet}(\L(A)_n^{\ob})$.
\end{definition}
Note that there are fully faithful functors $\iAlg_k \hookrightarrow \iDGA_k$ and $\iAlg_k \hookrightarrow \s\iAlg_k$, where the second inclusion views an involutive algebra as a discrete simplicial involutive algebra. 

\begin{proposition} \la{simp}
For any $A\in \iAlg_k$, there is a natural isomorphism 
$$ \mathrm{HR}_{\bullet}^{\ob}(A, k^n)\cong \mathrm{HR}_{\bullet}^{\ob}(A, k^n)^{\s} \,.$$
\end{proposition}

\begin{proof}
Since $\iAlg_k$ is the category of algebras over an operad, the monoidal Dold-Kan correspondence (see \cite[Theorem A.1]{BRYII}) implies that there is a Quillen equivalence 

$$ N^{\ast}\,:\,{\iDGA_k}^+ \rightleftarrows \s\iAlg_k\,:\,N\,,\qquad N^{\ast}\,:\,\cDGA_k^+ \rightleftarrows \s\cAlg_k\,:\,N$$

where ${\iDGA_k}^+$ (resp., $\cDGA_k^+$) stands for the category of involutive DG algebras (resp., commutative DG algebras) concentrated in nonnegative homological degree and $N$ stands for the normalized chain complex functor. These Quillen equivalences fit into a commutative diagram of functors 
\[
\begin{diagram}[small]
 {\iDGA_k}^+ & \rTo^{N^{\ast}} & \s\iAlg_k\\
  \dTo^{(\mbox{--})^{\ob}_n} & & \dTo_{(\mbox{--})^{\ob}_n}\\
 \cDGA_k^+ & \rTo^{N^{\ast}} & \s\cAlg_k\\
\end{diagram}\ .
\]
View $A\in \iAlg_k$ as a (discrete) object in $\s\iAlg_k$. Then $N^{\ast}(QA)$ is a cofibrant resolution of $A$ in $\s\iAlg_k$ for any cofibrant resolution $QA$ of $A$ in ${\iDGA_k}^+$. Hence, 
\begin{align*}
\mathrm{HR}_\bullet^{\ob}(A,k^n)^{\s} & \cong  \pi_\bullet[(N^{\ast}(QA))^{\ob}_n]\\
                                    & \cong \pi_\bullet[N^{\ast}((QA)^{\ob}_n))]\\
                                    &\cong \H_\bullet[NN^{\ast}((QA)^{\ob}_n))]\\
                                    & \cong \H_\bullet[(QA)^{\ob}_n]\\
                                    & \cong \mathrm{HR}_\bullet^{\ob}(A,k^n)
\end{align*}
This proves the desired proposition.
\end{proof}

\section{Relative dihedral homology and traces} \la{SecTr}

In the first part of this section, we recall the definition of relative dihedral homology (see \cite[Sec. 5.2]{L} and \cite{Go}). We then recall the construction of dihedral homology as a nonabelian derived functor on the category of involutive associative algebras (see \cite[Section 3]{Go}). This construction parallels the original Feigin-Tsygan construction of cyclic homology as a nonabelian derived functor on the category of algebras (see \cite{FT85,FT87}). Our exposition, which follows our approach to relative cyclic homology \cite{BFR,BKR}, is slightly more general than in \cite{Go} (for instance, we work over an arbitrary involutive DG algebra) and our proofs are different from the proofs therein. Subsequently, in Section \ref{TraceM}, we construct canonical trace maps $\,\Tr(S \bs A)_m:\,\mathrm{HD}_{m-1}(S \bs A) \to {\rm HR}^{\ob}_m(S \bs A, k^n)\,$ relating the dihedral homology of an involutive $S$-algebra
$ A \in \iAlg_S $ to its representation homology. 

Recall (see \cite[Sec. 5.2]{L} and \cite{Go}) that if $A \in \iAlg_k$, that the (absolute) {\it dihedral homology} of $A$, denoted by ${\rm HD}_\bullet(A)$ is the homology of the complex 
\begin{equation} \la{absdih} {\rm CD}_\bullet(A)\,:=\, \big( {\rm C}_{\bullet}(A)/({\rm Im}(1-t)+{\rm Im}(1-y)), {b}\big)\,,\end{equation}
where ${\rm C}_\bullet$ stands for the Hochschild chain complex of $A$, $b$ stands for the (differential induced by) the Hochschild boundary, and the operators $t$ and $y$ act on Hochschild chains by
\[ t(a_0 \otimes \ldots \otimes a_n) = (-1)^n a_n \otimes a_0 \otimes \ldots \otimes a_{n-1}\,,\qquad y(a_0 \otimes \ldots \otimes a_n)= (-1)^{\frac{n(n+1)}{2}} \sigma(a_0) \otimes \sigma(a_n) \otimes \ldots \otimes \sigma(a_1)\ .\]
This definition has an obvious generalization to arbitrary $A \in \iDGA_k$. For $A \in \iDGA_k$, the {\it reduced} dihedral homology $\overline{\rm HD}_\bullet(A)$ is the homology of the complex 
${\rm CD}_\bullet(A)/{\rm CD}_\bullet(k)$

\subsection{The dihedral functor}
For $A \in \iDGA_k$ let $A_{\natural} := \frac{A}{[A,A] + \mathrm{Im}(\id -\sigma)} \in \com_k$, where $[A,A]$ is the commutator subspace and $\mathrm{Im}(\id -\sigma)$ is all elements not fixed by the involution. The assignment $A\mapsto A_{\natural}$ is functorial, even being seen to have a right adjoint described by giving a trivial involution and then taking the free with involution from that involutive chain complex.

Since this is a functor, it acts on morphisms $(f\, : \, S\mapsto A) \to (f_{\natural}\, : \, S_{\natural}\to A_{\natural})$. As a result, it also respects triangles of composable morphisms:
\[
\begin{diagram}[small, tight]
  &       &     S     &        &  & & &       &     S_{\natural}     &        &\\
  &\ldTo  &           &  \rdTo &   & \implies & &\ldTo  &           &  \rdTo &  \\
A &       & \rTo^{f}  &        & B & &A_{\natural} &       & \rTo^{f_{\natural}}  &        & B_{\natural}
\end{diagram}
\]
If we fix an $S\in \iDGA_k$, then a map $f\, :\, S\mapsto A$ in involutive DGAs is an element in $\iDGA_S$. Indeed, we will often denote a map $S\mapsto A$ by $S\bs A$. If we wish to associate a chain complex to an element of this category, we can exploit $f\mapsto f_{\natural}$ to get a morphism in chain complexes, and then we can take the cone construction to get a chain complex.

We define the involutive Feigin-Tsygan functor
$$ \FT\,:\, \iDGA_S \rar\com_k \,,\,\,\,\,\,\, S \bs A \mapsto \cn(S_{\natural} \rar A_{\natural}) \,\text{.}$$
Letting $\Omega := k[t,dt]$ with
$|t|=0$ and $d(t):=dt$ denote the De-Rham algebra of the affine line $\mathbb A_k^1$ (equipped with the trivial involution), and using our definition of $A_{\natural}$ instead of $A/[A,A]$ in ~\cite{BKR}, we find the proof of their lemma 3.1 goes through exactly as written. It is reproduced for the reader's convenience.

\begin{lemma} \la{goodhom}

If $h: A \rar B \otimes \Omega$ is a morphism in $\iDGA_S$ such that $h(0)=f$ and $h(1)=g$, then
$\FT(f) \sim \FT(g)$ in $\com_k$.

\end{lemma}

\bproof
We recall that $\, \Omega = k[t] \oplus k[t] dt \,$ denotes the algebraic de Rham complex of $ \mathbb{A}^1_k $, which is a commutative
DG algebra. The natural maps $\,B \into B \otimes \Omega \hookleftarrow \Omega \,$ induce then an isomorphism of complexes $\,B_\n \otimes \Omega \xrightarrow{\sim} (B \otimes \Omega)_\n \,$.
Identifying these complexes, we have the commutative diagram
\begin{equation*}
\begin{diagram}[small, tight]
S_{\natural} & \rTo  & B_{\natural} \otimes \Omega\\
             & \rdTo & \dTo_{\id \otimes \ev_a}\\
             &       & B_{\natural}
\end{diagram}
\end{equation*}
which induces a map of c\^{o}nes $\,\widetilde{\ev}_a:\,
\cn(S_\n \to B_\n \otimes \Omega)\,\to \,\cn(S_\n \to B_\n)\,$
for each $ a \in k $. (Here, $ \ev_a:\, \Omega \to k $ stands for
the evaluation map $\,f(t) + g(t) dt \mapsto f(a)\,$, where $ a \in k $.)

If we now apply $ \FT $ to $ h $, we get a morphism
$$
\FT(h):\,\cn(S_\n \to A_\n)\,\to \,\cn(S_\n \to B_\n \otimes \Omega)\ ,
$$
such that $\,\FT(h)(0) := \widetilde{\ev}_0 \circ \FT(h) = \FT(f) \,$ and $\,\FT(h)(1) := \widetilde{\ev}_1 \circ \FT(h) =
\FT(g) \,$. To show that $\,\FT(f) \sim \FT(g) \,$ it thus suffices to show that
$\,\widetilde{\ev}_1 \sim \widetilde{\ev}_0\,$. But this last equivalence is immediate from
the fact that, for all $ a \in k \,$, $\,\ev_a \sim \ev_0\,$ as morphisms $ \Omega \to k $;
the corresponding homotopy $ \Omega \to k[1] $ is given by $\,f(t)+ g(t)\,dt \mapsto \int_0^a\,g(\tau)\,d\tau\,$.
\eproof

By Proposition~\ref{lhom}, the proof of Theorem 3.1 of~\cite{BKR} goes through (with obvious modifications) to prove its involutive
analog:

\begin{theorem}\la{dihedral}
 The functor $\FT\,:\,\iDGA_S \rar \com_k$ has a total left derived functor
$$ \L \FT\,:\, \Ho(\iDGA_S) \rar \mathcal D(k)\,,\,\,\,\, \gamma S \bs A \mapsto \gamma \FT(S \bs QA) $$
where $S \bs QA $ is a cofibrant resolution of $S \bs A$ in $\iDGA_S$.
\end{theorem}

\begin{proof}
The proof is analogous to the proof of Theorem~\ref{inv}. By Brown's Lemma (~\cite[Proposition A.2]{BKR}), it suffices to show that $\,\FT\,$ maps
acyclic cofibrations $\,i: A \sinto B \,$ between cofibrant objects in $\iDGA_S $ to weak equivalences
(quasi-isomorphisms) in $\Com(k) $. Given $\,i: A \sinto B \,$, there is a map $\,p:\,B \to A\,$ in $\iDGA_S$ such that
$\,p\,i=\id_A\,$ and $\,i\,p \sim \id_B\,$ ({\it cf.} \cite[Proposition A.1]{BKR}). Hence $ \FT(p) \,\circ\, \FT(i) = \id_{\FT(A)} $, and
by Proposition~\ref{lhom}, there is a morphism $\,h: B \rar B \otimes \Omega $ in $\iDGA_S $
with $\,h(0)=\id_B\,$ and $\,h(1)=i\,p\,$. Lemma~\ref{goodhom} now implies that $ \FT(ip) \sim \id_{\FT(B)} $.
Hence $ \FT(i p) $ induces the identity map on homology, and therefore $ \FT(i) $ is a quasi-isomorphism.
\end{proof}

We define the (relative) {\it dihedral homology} by setting
$$\mathrm{HD}_{m-1}(S \bs A)\,:=\, \mathrm{H}_m(\L\FT(S \bs A)) \,\text{.}$$
We note that when $S$ and $A$ are involutive algebras concentrated in degree $0$, our definition coincides (upto shift)
with that of~\cite{Go}. The proof of \cite[Thm 3.2]{Go} extends to the DG setting to yield the following minor generalization. 

\begin{prop} \la{HDAto0}
For $A \in \iDGA_k$, there is an isomorphism of homologies 
\[ {\rm HD}_\bullet(A \bs 0)\,\cong\, {\rm HD}_\bullet(A)\ .\]
\end{prop}

\subsection{The fundamental exact sequence for relative dihedral homology }

In order to study the properties of the dihedral functor, it is convenient to extend it to the category
$\Mor(\iDGA_k)$ of morphisms of involutive DG-algebras. Define
$$\widetilde{\FT}\,:\,\Mor(\iDGA_k) \rar \Com_k \,,\,\,\,\,(S \stackrel{f}{\rar} A) \mapsto \cn(S_{\natural} \stackrel{f_{\natural}}\rar A_{\natural})\,\text{.}$$
The arguments in Section 3.2 of~\cite{BKR} go through with the obvious modifications to give us the following results.

\begin{theorem} \la{mordih}

The functor $\widetilde{\FT}$ has a total left derived functor
$$\L\widetilde{\FT}\,:\,\Ho(\Mor(\iDGA_k)) \rar \mathcal D(k)\,,\,\,\,\, (S \bs A) \mapsto \FT(Q(QS \bs A)) \,\text{.}$$
 Here, $QS$ is a cofibrant resolution of $S$ in $\iDGA_k$ and $Q(QS \bs A)$ is a cofibrant resolution of $A$ in
$\iDGA_{QS}$.  Unpacking notation, in ~\cite{BKR} this was equivalently written 
\begin{equation}
\la{isff}
\L\eFT(S \bs A) = \cn[(QS)_\n \to Q(QS\bs A)_\n]\ ,
\end{equation}
\end{theorem}

\begin{proof}
We verify the assumptions of Brown's Lemma for $ \eFT $. Let $ A \to B $ be a
cofibrant object in $ \Mor(\iDGA_k)$, which means that $ A \to B $ is a cofibration
in $ \iDGA_k $, with $A$ being a cofibrant object.
Let
\begin{equation*}
\begin{diagram}[small, tight]
A                 & \rTo^{\alpha}    &  C         \\
\dTo          &                      & \dTo   \\
B                 & \rTo^{\beta}     &  D
\end{diagram}
\end{equation*}
be an acyclic cofibration in $ \Mor(\iDGA_k) $, which means that
$ \alpha $ is an acyclic cofibration, $\, \beta $ is a weak equivalence, and
the induced map $\, B \coprod_A C \to D \,$ is a cofibration in $ \iDGA_k $.
Since $ A \into B $ is a cofibration, the pushout $ B \to B \coprod_A C $ of $ \alpha $
along $ A \to B $ is an acyclic cofibration. Since
$\,\beta:\, B \to B \coprod_A C \to D $ is acyclic, the cofibration
$\, B \coprod_A C \to D \,$ is also acyclic. It follows that
$ \beta $ itself is an acyclic cofibration. The standard argument
(see the proof of Theorem~\ref{mordih} above) implies now that $ \alpha_\n $ and
$ \beta_\n $ are quasi-isomorphisms. Hence, the pair
$(\alpha_\n,\,\beta_\n) $ induces a quasi-isomorphism
$\,\eFT(A \bs B) \xrightarrow{\sim} \eFT(C \bs D) \,$.
By Brown's Lemma, $\, \L\eFT $ exists and is given by formula
\eqref{isff}, which is independent of the choice of resolutions.
\end{proof}

\begin{lemma} \la{3.2}\la{LtP}
Given involutive DG-algebra maps $R \stackrel{f}{\rar} S \stackrel{g}{\rar} A$, there is a distinguished triangle
in $\mathcal D(k)$ of the form:
$$ \L\widetilde{\FT}(S \bs R) \rar \L\widetilde{\FT}(R \bs A) \rar \L\widetilde{\FT}(S \bs A) \rar \L\widetilde{\FT}(S \bs R)[1]\,\text{.}$$
 \end{lemma}

\begin{proof}
Fix a cofibrant resolution $\, p_R: QR \sonto R \,$ in $ \iDGA_k $. Then choose cofibrant resolutions $\, p_S: QS \sonto S \,$ of $ g \circ p_R $ in $ \iDGA_{QR} $ and
$p_A: QA \sonto A\,$ of $ f\circ p_S $ in $ \iDGA_{QS}$. Note that both $QS$ and $QA$ are cofibrant
in $ \iDGA_k $, and there are liftings $ \tilde{g}: QR \into QS $ and  $ \tilde{f}: QS \into QA $ which are cofibrations.
Now, by Octahedron Axiom, associated to $ (QR)_\n \xrightarrow{\tilde{g}} (QS)_\n \xrightarrow{\tilde{f}} (QA)_\n $
is a distinguished triangle in $ \D(k)$ of the form
$$
\la{isct}
\cn(\tilde{g}_\n) \to \cn(\tilde{f}_\n \circ \tilde{g}_\n) \to \cn(\tilde{f}_\n) \to \cn(\tilde{g}_\n)[1] \ .
$$
By \eqref{isff}, $\,\L\eFT(g) = \cn(\tilde{g}_\n)$, and similarly for $f$ and $ f \, g $. The above triangle is thus isomorphic to \eqref{isct}.
\end{proof}

\begin{lemma} \la{3.3}
 For any $S \bs A$ in $\iDGA_k$, $\L\widetilde{\FT}(S \bs A) \,\cong\, \L\FT(S \bs A)$.
\end{lemma}

\begin{proof}
The inclusion functor $i\, : \, \iDGA_S \to \Mor(\iDGA_k)$ maps weak equivalences to weak equivalences and hence descends to the homotopy categories $\overline{i}\, :\, \Ho(\iDGA_S)\to \Ho(\Mor(\iDGA_k))$. The universal property of derived functors then yields a natural transformation 
$$\L \widetilde{\FT}\circ \overline{i}\to \L \FT$$
Our claim is that this is an isomorphism of functors.

Let $ p: Q \sonto S $ be a cofibrant resolution of $S$ in $\iDGA_k$. Then $\L\widetilde{\FT}(S\bs 0) \cong Q\langle x \rangle_\n/Q_\n\,$
and $\,\L\FT(S\bs 0) \cong S\langle x \rangle_\n/S_\n\,$, where $ dx =1 $, and
$ p $ induces an isomorphism $ \L\eFT(S\bs 0) \stackrel{\sim}{\to} \L\FT(S\bs 0)$. More generally,
if $ p: R \to S $ is any quasi-isomorphism in $ \iDGA_k $, then $ p $ induces
$ \L\FT(R\bs 0) \stackrel{\sim}{\to} \L\FT(S\bs 0)$.
Hence, for $ S \to A \to 0 $, the same argument as in the proof of Lemma~\ref{LtP} shows that there is an exact triangle

$$ \L\FT(S \bs R) \rar \L\FT(R \bs A) \rar \L\FT(S \bs A) \rar \L\FT(S \bs R)[1]\,\text{.}$$

\end{proof}

By Lemma~\ref{3.2} and Lemma~\ref{3.3}, we obtain (via a different proof) the following result proven in~\cite{Go}.

\begin{theorem} \la{t3.3}
 Given involutive DG-algebra maps $R \stackrel{f}{\rar} S \stackrel{g}{\rar} A$, there is a long exact sequence
$$ .. \rar \mathrm{HD}_n(S \bs R) \rar \mathrm{HD}_n(R\bs A) \rar \mathrm{HD}_n(S \bs A) \rar \mathrm{HD}_{m-1}(S \bs R) \rar ...$$
of relative dihedral homologies.
\end{theorem}

\begin{proof}
Apply homology to the distinguished triangle in Lemma~\ref{3.2}. 
\end{proof}

Applying Theorem \ref{t3.3} to the sequence of involutive DG algebra maps $k \to A \to 0$ we obtain

\begin{cor} \la{t3.3c}
For $A \in \iDGA_k$, there is a (natural) isomorphism of homologies for all $n$
\[ {\rm HD}_{n-1}(k \bs A)\,\cong\, \overline{\rm HD}_n(A)\ .\]
\end{cor}



\subsection{Trace construction}\la{TraceM}
For any involution $\sigma$ on $\M_n(k)$, the natural map
$\, k \into \M_n(k) \onto \M_n(k)_\n \,$ is an isomorphism of complexes, which we can use to
identify $\, \M_n(k)_\n \cong k \,$. This isomorphism takes the (class of) a matrix to its trace. Fix $\, S \in \iDGA_k \,$ and an involutive DG algebra map
$\,\varrho:\,S \to \M_n(k) \,$.
For an involutive $S$-algebra $ A \in \iDGA_S $, consider the relative DG representation
scheme $ \Rep_n^{\ob}(S \bs A) $, and let $\,\pi_n^{\ob}:\,A \to \M_n(k) \otimes (S \bs A)_n^{\ob} \,$
denote the universal representation of $A$ corresponding to the identity map
in the adjunction of Proposition~\ref{S2P1}$(b)$. Consider the morphism of complexes
\begin{equation}
\la{E1S4}
A \xrightarrow{\pi_n^{\ob}} \M_n(k) \otimes (S \bs A)_n^{\ob} \xrightarrow{\Tr_n \otimes \id } (S \bs A)_n^{\ob}\ .
\end{equation}
Since $ \pi_n^{\ob} $ is a map of involutive $S$-algebras, and the $S$-algebra structure on $ \M_n(k) \otimes (S \bs A)_n^{\ob} $
is of the form $ \varrho \otimes \id $, \eqref{E1S4} induces a map
$\, \Tr_n \circ \pi_n^{\ob}: \, A_\n \to (S \bs A)_n^{\ob} $, which fits in the commutative diagram
\begin{equation}
\la{D1S4}
\begin{diagram}[small]
  S_\n                  &  \rTo^{}  & A_\n  \\
\dTo^{\Tr_n \circ \varrho} &           & \dTo_{\Tr_n  \circ \pi_n^{\ob}} \\
  k                     &  \rTo^{}  & (S \bs A)_n^{\ob}
\end{diagram}
\end{equation}
This, in turn, induces a morphism of complexes
\begin{equation}
\la{E2S4}
\cn(S_\n \to A_\n) \to \overline{(S \bs A)}_n^{\ob}\ ,
\end{equation}
where we write $\, \overline{R} = R/k\cdot 1_R \,$ for a unital DG algebra $R$.
The family of morphisms \eqref{E2S4} defines a natural transformation of functors
$\,\Tr_n:\, \FT \to \overline{(\mbox{---})}_n^{\ob}\,$, which descends to a natural
transformation of the derived functors by the universal property of derived functors:
\begin{equation}
\la{E3S4}
\Tr_n :\, \L \FT \to  \overline{\L(\mbox{---})}_n^{\ob} \ .
\end{equation}

Next, observe that, for any (non-acyclic) unital DG algebra $ R $ (with or without involution),
\begin{equation}
\la{hvan}
\H_{n}(\overline{R}) \cong
\left\{
\begin{array}
{lcl}
\overline{\H_0(R)} & , & n = 0 \\*[1ex]
\H_{n}(R)           & , & n \not= 0
\end{array}
\right.
\end{equation}
This follows from the long homology sequence arising from
$ 0 \to k \to  R \to \overline{R} \to 0 $ and the fact that $\H_n(k)=0$ for $n\not =0$
Hence, if $ A \in \iAlg_S $ is an ordinary involutive algebra, applying the natural transformation
\ref{E3S4} to $ S \bs A $ and using \eqref{hvan}, we can define
\begin{equation}
\la{char}
\Tr_n(S \bs A)_m:\ \mathrm{HD}_{m-1}(S\bs A) \to \HRim(S \bs A, k^n)\ ,\ n \ge 1\ .
\end{equation}
Assembled together, these trace maps define a homomorphism of
graded commutative algebras
\begin{equation}
\la{lchar}
\bL \Tr_n(S \bs A)_\bullet :\ \bL(\mathrm{HD}(S\bs A)[1]) \to \mathrm{HR}^{\ob}_{\bullet}(S \bs A, k^n)
\ ,
\end{equation}
where $ \bL $ denotes the graded symmetric algebra of a graded $k$-vector space $W$.

\subsection{The absolute case}
\la{abscase}

We examine the trace maps \eqref{char} and \eqref{lchar} in the special case
when $S = k$. In this case, by Corollary \ref{t3.3c}, the maps \eqref{char} relate the reduced dihedral homology of $A$ to the (absolute) representation
homology:
\begin{equation}
\la{char1}
\Tr_n(A)_m:\ \overline{\mathrm{HD}}_{m}(A) \to \mathrm{HR}^{\ob}_{m}(A,k^n)\ ,\ n \ge 1\ .
\end{equation}

Now, for each $m$, there is a natural map $ \mathrm{HD}_m(A) \to \overline{\mathrm{HD}}_m(A)$. Combining this map with \eqref{char1}, we get
%
\begin{equation}
\la{char2}
\Tr_n(A)_m:\ \mathrm{HD}_{m}(A) \to \mathrm{HR}^{\ob}_{m}(A, k^n)\,,\qquad \forall \,m\geq 0 \ .
\end{equation}

Notice that whem $m=0$, the map \eqref{char2} becomes (see Proposition \ref{zerohom}) the map $\, \Tr(A)_0:\, A_\n \to (A)_n^{\ob}\,$ induced by the compopsite map $\,A \xrightarrow{\pi_n^{\ob}} \M_n(k) \otimes  (A)_n^{\ob}  \xrightarrow{\Tr \otimes \id} (A)_n^{\ob}\,$.

The linear maps \eqref{char2} define an algebra homomorphism
\begin{equation}
\la{lchar1}
\bL \Tr(A)_\bullet :\ \bL[\mathrm{HD}(A)] \to \mathrm{HR}_{\bullet}(A, k^n)\ .
\end{equation}
Since, for $n\ge 1$, \eqref{char2} factor through reduced dihedral homology, \eqref{lchar1} induces
\begin{equation}
\la{trhomr}
\overline{\Tr}(A)_\bullet :\ \Sym(\mathrm{HD}_0(A)) \otimes \bL(\overline{\mathrm{HD}}_{\ge 1}(A)) \to \mathrm{HR}^{\ob}_\bullet(A, k^n)\ ,
\end{equation}
where $ \overline{\mathrm{HD}}_{\ge 1}(A) := \bigoplus_{m \ge 1} \overline{\mathrm{HD}}_{m}(A) $.
\bprop
The image of the maps \eqref{lchar1} and \eqref{trhomr} is contained in $ \mathrm{HR}^{\ob}_\bullet(A, k^n)^{G}$, where $G$ is the subgroup of ${\rm GL}_n(k)$ from Section \ref{secinv}
\eprop
\bproof
Recall that, when $V$ is concentrated in degree $0$, we have the invariant subfunctor of $ (\mbox{---})_n^{\ob} $:
$$
((\mbox{---})_n^{\ob})^{G} :\ \DGA_k \to \cDGA_k\ ,\quad A \mapsto ((A)_n^{\ob})^{G}\ .
$$
It is clear from \eqref{E1S4} that the natural transformation $ \Tr_n$ maps $ \FT $ to this subfunctor:
\begin{equation}
\la{E8S4}
\Tr_n:\ \FT \to \overline{(\mbox{---})}_n^{\ob, G}
\end{equation}
Now, by Theorem~\ref{inv}, $\, ((\mbox{---})_n^{\ob})^G $ has a total left derived functor, which is isomorphic to the invariant subfunctor $ (\L(\mbox{---})_n^{\ob}) ^G$ of $ \L(\mbox{---})_n^{\ob} $. Hence,
\eqref{E8S4} descends to the homotopy categories, i.e. induces a natural transformation
$\, \Tr_n^{\ob}:\, \L\FT \to \overline{(\L(\mbox{---})}_n^{\ob})^G\,$. On the homology level, this gives
\begin{equation}
\la{E9S4}
\bL \Tr_n^{\ob}(A)_\bullet :\ \bL(\mathrm{HD}(A)) \to \mathrm{HR}^{\ob}_{\bullet}(A, k^n)^{G}\ ,
\end{equation}
which is the statement of proposition.
\eproof

\section{Representation homology and Lie cohomology} \la{Sec3}

In this section, we show (see Theorem \ref{invmain}) that for any augmented involutive DG algebra $A$, ${\rm DRep}^{\ob}_n(A)$ is isomorphic as a (homotopy) commutative algebra to the Chevalley-Eilenberg algebra of the DG Lie coalgebra $\gl_n^{+,\ast}(\overline{\bB}A)$ obtained by taking the $+1$-eigenspace of the skew involution induced on the current Lie coalgebra $\gl_n^\ast(\overline{\bB}A)$ by he involutions on $A$ and on $\M_n(k)$. Here, $\overline{\bB}A$ stands for the reduced bar construction of $A$. This result is the involutive analog of one of the main results of \cite{BFPRW} (see {\it loc. cit.}, Theorem 3.1).

\subsection{Skew involutive DG coalgebras and involutive DG algebras}

\subsubsection{The cobar-bar formalism}

Let $\Delta_{-1}\,:\,k[-1] \rar k[-1] \otimes k[-1]$ be the map taking $s^{-1}$ to $-s^{-1} \otimes s^{-1}$, where $s^{-1}$ denotes $1 \in k$ viewed as an element of $k[-1]$.  Recall that for $C \in \DGC_{k/k}$, the cobar construction
$\cb(C)$ is the tensor algebra $T(\overline{C}[-1])$ equipped with differential $d_1+d_2$ where $d_1$ is the derivation induced by the differential on $\overline{C}[-1]$ and $d_2|_{\overline{C}[-1]}$ is given by the composite map
\begin{equation} \la{d2}
\begin{diagram} k[-1] \otimes\overline{C} & \rTo^{\Delta_{-1} \otimes \Delta} & k[-1] \otimes k[-1] \otimes \overline{C} \otimes \overline{C} & \rTo^{{\rm Sw}}_{\cong} & (k[-1] \otimes C)^{\otimes 2} \end{diagram}
\end{equation}

For $R \in \iDGA_{k/k}$, recall that the bar construction $\bB(R)$ is the cofree coalgebra $T^c(\overline{R}[1])$ on $\overline{R}[1]$ equipped with the degree $-1$ codifferential $d_1+d_2$ where the coderivation $d_1$ is induced by the differential on $\overline{R}[1]$ and where the corestriction of $d_2$ to $\overline{R}[1]$ is given by the composite map
\begin{diagram} \overline{R}[1]^{\otimes 2} \,\cong\, \overline{R} \otimes \overline{R} \otimes k[1] \otimes k[1] & \rTo^{m \otimes \mu} & \overline{R} \otimes k[1]\,\cong\, \overline{R}[1] \,\text{.} \end{diagram}
Here, $m$ denotes the multiplication on $R$ and $\mu\,:\,k[1] \otimes k[1] \rar k[1]$ is the degree $-1$ map taking $1 \otimes 1$ to $1$.

Further recall that the hom-complex $\HOM(\overline{C}, \overline{R})$ has the structure of a DG algebra, with product given by
$$ f \ast g \,:=\, m_R \circ (f \otimes g) \circ \Delta_C \,\text{.}$$
Let $\Tw(\overline{C},\overline{R})$ denote the set of Maurer-Cartan elements in $\HOM(\overline{C}, \overline{R})$, i.e, the set of degree $-1$ elements in $\HOM(\overline{C}, \overline{R})$ satisfying the Maurer-Cartan equation
$$ df+\frac{1}{2}[f,f] \,=\, df+ f \ast f \,=\,0 \,\text{.}$$
It can be shown without difficulty, a good reference is ~\cite{LV}, that one has natural bijections of sets
\begin{equation} \la{twbij} \Hom_{\DGA_{k/k}}(\cb(C), R)\,\cong \,\Tw(\overline{C},\overline{R})\,\cong\, \Hom_{\DGC_{k/k}}(C, \bB(R)) \,\text{.}\end{equation}
Thus, $\cb\,:\,  \DGC_{k/k} \rightleftarrows \DGA_{k/k}\,:\,\bB$ are adjoint functors.

The category $\DGA_{k/k}$ has a natural model structure. The weak equivalences in this model structure are the quasi-isomorphisms and the fibrations are the degree-wise surjections. The cofibrations are characterised as morphisms satisfying the left lifting property with respect to acyclic fibrations. More explicitly, every cofibrant object in $\DGA_{k/k}$ is isomorphic to a retract of $\cb(C)$ for some $C \in \DGC_{k/k}$ (see~\cite[Theorem 4.3]{Ke}). The category $\DGC_{k/k}$ has a dual model structure, where the weak equivalences are the morphisms $f$ such that $\cb(f)$ is a quasi-isomorphism. The following theorem is well known (see, for example,~\cite{LV}).

\begin{theorem}
{\cite[Theorem 4.3]{Ke}}
The pair of adjoint functors
\begin{equation} \la{cbb} \cb\,:\,  \DGC_{k/k} \rightleftarrows \DGA_{k/k}\,:\,\bB \end{equation}
is a Quillen pair. In particular, the functors~\eqref{cbb} induce mutually inverse equivalences between homotopy categories
$$ \L \cb\,:\,\Ho(\DGC_{k/k}) \,\cong\, \Ho(\DGA_{k/k})\,:\, \R \bB \,\text{.}$$
\end{theorem}

\subsubsection{Skew involutive coalgebras and involutive algebras}

Modifying a dualized version of the definition of an involution on a DG algebra, we define a {\it skew involution} $\sigma:C \rar C$ on a (non-counital) DG coalgebra $C \in \DGC$ to be a morphism of complexes such that $\sigma^2=\id_C$ and
$$\Delta \circ \sigma \,=\, -\tau \circ (\sigma \otimes \sigma) \circ \Delta \,,$$
where $\tau: C \otimes C \rar C \otimes C$ denotes the swap map. Clearly, when $C=k$, the only possible skew involution on $C$ is the linear map such that $\sigma(1)=-1$. We remark that if $\sigma$ is a skew involution on $C$, then $-\sigma$ is an involution on $C$. 

\begin{definition} \la{skc}
A counital (resp., coaugmented) DG coalgebra $C \in \DGC_k$ (resp., $C \in \DGC_{k/k}$) is {\it skew involutive} if $C$ is equipped with a skew involution such that the counit $C \to k$ is a morphism of skew involutive DG coalgebras (resp., the counit $C \to k$ and the coaugmentation $k \to C$ are morphisms of skew involutive DG coalgebras).
\end{definition}

Clearly, if $C \in \DGC_{k/k}$ is skew involutive with skew involution $\sigma$ then $\sigma(1)=-1$. The proof of the following lemma is an easy exercise that we leave to the reader.

\begin{lemma} \la{basskc}
$(i)$ The skew involution on a skew involutive coaugmented DG coalgebra $C \in \DGC_{k/k}$ restricts to a skew involution on $\overline{C}\in \DGC$.\\

$(ii)$ A skew involution on (the non-counital) DG coalgebra $\overline{C}\in \DGC$ extends by linearity to a skew involutive structure on $C \in \DGC_{k/k}$ such that $\sigma(1)=-1$.

\end{lemma}

\begin{prop} \la{skinvtoinv}
For $C \in \DGC_{k/k}$ skew involutive, $\cb(C)$ is an involutive $k$-augmented DG algebra. Dually, for $A \in \iDGA_{k/k}$, $\bB(A)$ is a skew involutive $k$-augmented DG coalgebra. 
\end{prop}

\begin{proof}
Let $\sigma$ denote the skew involution on $C$. One can define an involution on the graded algebra $T(\overline{C}[-1])$ whose restriction to the (graded) generating space $\overline{C}[-1]$ is $\sigma$. Denote this extended involution on $T(\overline{C}[-1])$ by $\sigma$ as well. It remains to check that $\sigma \circ d=d \circ \sigma$, where $d$ is the differential on $\cb(C)$.
Since $\sigma \circ d$ and $d \circ \sigma$ are both derivations from $\cb(C)$ to $\cb(C)^{\rm op}$ with respect to the graded algebra homomorphism $\sigma$, it suffices to check that they coincide on generators, i.e, when restricted to $\overline{C}[-1]$. Recall that $d=d_1+d_2$. That $\sigma$ (restricted to $\overline{C}[-1]$ commutes with $d_1$ follows from the fact that $\sigma:C \rar C$ is a map of complexes. That $\sigma$ (restricted to $\overline{C}[-1]$ commutes with $d_2$ follows from the fact that it is a skew-involution on $C$: indeed, for $x \in \overline{C}$ we have (using Sweedler notation for coproducts)

\begin{align*}
\sigma(d_2(s^{-1}x)) &=\sigma \big({\rm Sw}(-s^{-1} \otimes s^{-1} \otimes x^{(1)} \otimes x^{(2)})\big) \text{ (by \eqref{d2})} \\
 &= -(-1)^{|x^{(1)}|}\sigma(s^{-1}x^{(1)}\otimes s^{-1}x^{(2)}) \\   
 &= (-1)^{(|x^{(1)}|-1)|x^{(2)}|}(\sigma(s^{-1}x^{(2)})\otimes \sigma(s^{-1}x^{(1)}))\\
 &= (-1)^{|x^{(1)}||x^{(2)}|}{\rm Sw}\big(s^{-1} \otimes s^{-1} \otimes \sigma(x^{(2)}) \otimes \sigma(x^{(1)})\big)\\
&= {\rm Sw}\big(-s^{-1} \otimes s^{-1} \otimes -\tau \circ (\sigma \otimes \sigma)(x^{(1)} \otimes x^{(2)})\big)\\
&= {\rm Sw}\big(-s^{-1} \otimes s^{-1} \otimes \Delta(\sigma(x))\big) \text{ (since $C$ is skew involutive)}\\
&=  d_2(s^{-1}\sigma(x))=d_2(\sigma(s^{-1}x)) \ .\\
\end{align*}

This proves the first assertion. Similarly, one verifies that $\bB(A)$ has a skew involution whose corestriction to $\bar{A}[1]$ is the involution on $A$ and which commutes with the differential on $\bB(A)$.
\end{proof}

Let $\DGC^{\ob}_{k/k}$ denote the category of skew involutive coaugmented conilpotent  DG coalgebras.  Let $C \in \iDGC_{k/k}$ and let $R \in \iDGA_{k/k}$. Then, the complex $\HOM(\overline{C}, \overline{R})$ is equipped with a $\Z_2$-action given by
$$\gamma(f)\,:=\, \sigma_R \circ f \circ \sigma_C\,,$$
where $\gamma$ is the generator of $\Z_2$. 

\begin{lemma} \la{invHom}
 The $\Z_2$ action on $\HOM(\overline{C}, \overline{R})$ is a skew-involution: $$ \gamma(f \ast g)\,=\, -(-1)^{|f||g|} \gamma(g) \ast \gamma(f)\,, $$
\end{lemma}
\begin{proof}
\begin{align*}
\gamma (f \ast g) \, =\, \sigma_R \circ (f\ast g) \circ \sigma_C &=\, \sigma_R \circ  m_R \circ (f \otimes g) \circ \Delta_C \circ \sigma_C\\
   &= - \sigma_R \circ  m_R \circ \left[ (f \otimes g) \circ \tau \circ (\sigma_C \otimes \sigma_C) \circ \Delta_C  \right]\\
   &= - m_R \circ \tau \circ (\sigma_R\otimes \sigma_R)\circ  (f \otimes g) \circ \tau \circ (\sigma_C \otimes \sigma_C) \circ \Delta_C\\
   &= -m_R \circ (\sigma_R \otimes \sigma_R) \circ \tau \circ (f \otimes g) \circ \tau \circ (\sigma_C \otimes \sigma_C) \circ \Delta_C\\
   &=-(-1)^{|f||g|} m_R \circ (\sigma_R \otimes \sigma_R) \circ (g \otimes f) \circ (\sigma_C \otimes \sigma_C) \circ \Delta_C\\
   &=-(-1)^{|f||g|} m_R \circ [(\sigma_R \circ g \circ \sigma_C) \otimes (\sigma_R \circ f \circ \sigma_C)] \circ \Delta_C\\
   &= -(-1)^{|f||g|} \gamma(g) \ast \gamma(f)\ .
\end{align*}
\end{proof}
By Lemma \ref{invHom}, 
$$ \gamma([f\,,\,g])\,=\, [\gamma(f), \gamma(g)]\,\text{.}$$
Hence, the DG {\it Lie algebra}  $\HOM(\overline{C}, \overline{R})$ has a natural $\Z_2$-action. As a result, $\Tw(C,R)$ has a natural $\Z_2$-action.
\begin{lemma} \la{twbijinv}
There are natural bijections of sets
$$ \Hom_{\iDGA_{k/k}}(\cb(C), R)\,\cong\,\Tw(\overline{C},\overline{R})^{\Z_2}\,\cong\, \Hom_{\iDGC_{k/k}}(C, \bB(R)) \,\text{.}$$
\end{lemma}

\begin{proof}
Let $C \in \iDGC_{k/k}$ and let $R \in \iDGA_{k/k}$. Then, the augmented DGA homomorphisms $\cb(C) \to R$ are in natural bijection with elements of $\Tw(\overline{C},\overline{R})$ (see \eqref{twbij}). This bijection takes a homomorphism $\phi:\cb(C) \to R$ to its restriction to the generating complex $\overline{C}[-1]$, viewed as an element of $\HOM(\overline{C},\overline{R})_{-1}$. We denote this twisting cochain by $\overline{\phi}$. Now, $\phi$ is a homomorphism of involutive DG algebras if and only if 
\begin{align*}
 \phi \circ \sigma_{\cb(C)} = \sigma_R \circ \phi & \iff \phi \circ \sigma_{\cb(C)}|_{\overline{C}[-1]} = \sigma_R \circ \phi |_{\overline{C}[-1]}\\
  & \iff \overline{\phi} \circ \sigma_C = \sigma_R \circ \overline{\phi}\\
  & \iff \gamma(\overline{\phi})= \overline{\phi}\ .
\end{align*}
Thus, the involution preserving (augmented) DG algebra homomorphisms $\cb(C) \to R$ are in natural bijection with $\Tw(\overline{C},\overline{R})^{\Z_2}$. This proves the first natural bijection in the desired lemma. The second natural bijection in the desired lemma is shown similarly.
\end{proof}

Recall (see Section \ref{MoStr}) that $\iDGA_{k/k}$ has a model structure where the weak equivalences are the quasi-isomorphisms and where the fibrations are the degree wise surjections. Following~\cite[Section 1.3]{Le}, one can show that $\iDGC_{k/k}$ has a dual model structure whose weak equivalences are the morphisms $f$ such that $\cb(f)$ is a quasi-isomorphism. It can also be shown without difficulty that every cofibrant object in $\iDGA_{k/k}$ is a retract of $\cb(C)$ for some  $C \in \iDGC_{k/k}$. Further following~\cite[Section 1.3]{Le}, one can show that

\begin{theorem} \la{kosdualinv}
The pair of adjoint functors
\begin{equation} \la{icbb} \cb\,:\,  \iDGC_{k/k} \rightleftarrows \iDGA_{k/k}\,:\,\bB \end{equation}
is a Quillen pair. In particular, the functors~\eqref{cbb} induce mutually inverse equivalences between homotopy categories
$$ \L \cb\,:\,\Ho(\iDGC_{k/k}) \,\cong\, \Ho(\iDGA_{k/k})\,:\, \R \bB \,\text{.}$$
\end{theorem}

\subsection{Lie algebras and Lie coalgebras from involutions} \la{CEThm}

If $\sigma$ is a skew involution on $C \in \DGC$, $C$ breaks into a direct sum $C^{+} \oplus C^{-}$ of complexes such that $\sigma$ acts by multiplication by $1$ (resp., $-1$) on $C^{+}$ (resp., $C^{-}$). Note that $\sigma$ commutes with the Lie cobracket $] \mbox{--}[:= (\id -\tau) \circ \Delta$: indeed, 
\begin{align*}
(\id -\tau) \circ \Delta \circ \sigma &= -(\id -\tau) \circ \tau \circ (\sigma \otimes \sigma) \circ \Delta\\
&= (\id -\tau) \circ (\sigma \otimes \sigma) \circ \Delta\\
&= (\sigma \otimes \sigma) \circ (\id -\tau) \circ \Delta\ .
\end{align*}
Hence, the projection $C \otimes C \onto C_+ \otimes C_+$ composed with the restriction of the Lie cobracket $]\mbox{--}[$ to $C_+$ defines a Lie cobracket on $C^{+}$. More generally, if $\M_n(k)$ is equipped with an involution, its dual $\M_n^{\ast}(k)$ becomes a coalgebra equipped with an involution. Tensoring this involution with a given skew involution on $C$ makes
${\M}_n^{\ast}(C)$ a (DG) coalgebra with skew involution $\sigma$. We let $\gl_n^{\ast}(C)$ denote $\M_n^\ast(C)$ equipped with the Lie cobracket $]\mbox{--}[:=(\id-\tau)\circ \Delta$, of which $\sigma$ is an automorphism of order $2$. We let $\gl_n^{+,\ast}(C)$ denote the $+1$-eigenspace of $\sigma$. As described above, $\gl_n^{+,\ast}(C)$ acquires the structure of a (DG) Lie coalgebra.
Note that if $C=k$ equipped with trivial skew involution, $\gl_n^{+,\ast}(k)$ is the linear dual of $\gl_n^{-}(k)$. One always has a surjection $\gl_n^{+,\ast}(C) \twoheadrightarrow \gl_n^{+,\ast}(k)$ of DG coalgebras for $C \in \iDGC_{k/k}$.

Next, we remind the reader that if $\mathfrak{C}$ is any DG Lie coalgebra, then one has the {\it Chevelley-Eilenberg algebra} $\C(\mathfrak{C};k):=(\Sym_k(\mathfrak{C}[-1]),d_1+d_2)$ of $C$, whose differential if the sum of two derivations $d_1$ and $d_2$. The restriction of $d_1$ to $\mathfrak{C}[-1]$ is the differential on $\mathfrak{C}[-1]$, while the restriction of $d_2$ to $\mathfrak{C}[-1]$ is the composite map
\[ \begin{diagram}k[-1] \otimes \mathfrak{C} & \rTo^{\Delta_{-1} \otimes ]\mbox{--}[} & k[-1] \otimes k[-1] \otimes \wedge^2 \mathfrak {C} \cong \Sym^2(k[-1] \otimes \mathfrak{C}) & \rInto & \Sym(k[-1] \otimes \mathfrak{C}) \end{diagram}\ .\]
If $\mathfrak{C}=\g^\ast$, where $\g$ is any finite dimensional Lie algebra, then $\C(\g^{\ast};k)$ is precisely the Chevalley-Eilenberg cochain algebra of $\g$. We are now ready to state the mai theorem of this section.

Let $A \in \iDGA_{k/k}$ be an involutive algebra. Suppose that there exists a quasi-isomorphism $\cb(C) \stackrel{\sim}{\twoheadrightarrow} A$ in $\iDGA_{k/k}$ for some skew involutive $C$ in $\DGC_{k/k}$. Fix an involution on $\M_n(k)$. Then,

\begin{theorem}
\la{invmain}
For any $n \geq 1$ there are isomorphisms in $\Ho(\cDGA_{k/k}$:
$$\DRep_n^{\ob}(A)\,\cong\, \C({\gl}^{+,\ast}_n(\overline{C});k)\,,\qquad \DRep_n^{\ob}(A)^G \,\cong\, \C(\gl^{+,\ast}_n(C), \gl^{+,\ast}_n(k);k)\ .$$
Consequently,
$$ {\rm HR}^{\ob}_\bullet(A,k^n)\,\cong\, \H_\bullet({\gl}^{+,\ast}_n(\overline{C});k)\,,\qquad {\rm HR}^{\ob}_\bullet(A,k^n)^G\,\cong\, \H_\bullet(\gl^{+,\ast}_n(C), \gl^{+,\ast}_n(k);k)\ . $$
\end{theorem}

\begin{proof}
We begin by proving the first part. Since $\cb(C)$ is cofibrant in $\iDGA_{k/k}$ (see the preliminaries before Theorem \ref{kosdualinv}), it suffices to check that $(\cb(C))^{\ob}_n\,\cong\, \C({\gl}^{+,\ast}_n(\overline{C});k)$ in $\cDGA_{k/k}$. Now, for any $B \in \cDGA_{k/k}$,
\begin{equation} \la{eqinvadj1} \Hom_{\cDGA_{k/k}}((\cb(C))^{\ob}_n\,,\,B)\,\cong\, \Hom_{\iDGA_{k/k}}( \cb(C)\,,\,\widetilde{\M}_n(B)) \,\text{.} \end{equation}
On the other hand, by Lemma~\ref{twbijinv}
\begin{equation} \la{eqinvadj2}\Hom_{\DGA_{k/k}}( \cb(C)\,,\,\widetilde{\M}_n(B))\,\cong\,\mathtt{Tw}(\overline{C}\,,\,\M_n(\bar{B}))^{\Z_2} \,,\end{equation}
where the $\Z_2$-action on $\mathtt{Tw}(\overline{C}\,,\,\M_n(\bar{B}))$ is induced by the $\Z_2$-action on $\HOM(\overline{C}\,,\,\M_n(\bar{B}))$ (see Lemma \ref{invHom}). Since the natural isomorphism 
$$\HOM(\overline{C}\,,\,\M_n(\bar{B})) \,\cong\, \HOM(\M_n^{\ast}(\overline{C}),\overline{B})$$
is an isomorphism of involutive DG-algebras (and therefore, an isomorphism of DG Lie algebras with $\Z_2$-action), there are natural bijections 
\begin{equation} \la{eqinvadj3} \mathtt{Tw}(\overline{C}\,,\,\M_n(\bar{B}))^{\Z_2} \,\cong\, \mathtt{Tw}(\M_n^{\ast}(\overline{C})\,,\,\bar{B})^{\Z_2}\,\cong\, \mathtt{Tw}(\gl_n^{+,\ast}(\overline{C})\,,\,\bar{B}) \,,\end{equation}
where the last bijection is a consequence of the DG Lie algebra isomorphism
$$\HOM(\M_n^{\ast}(\overline{C}),\overline{B})^{\Z_2} \,\cong\, \HOM(\gl_n^{+,\ast}(\overline{C}),\bar{B})\ .$$
It follows from \eqref{eqinvadj1}, \eqref{eqinvadj2} and \eqref{eqinvadj3} that there is a natural bijection 
\begin{equation} \la{eqinvadj4} \Hom_{\cDGA_{k/k}}((\cb(C))^{\ob}_n\,,\,B)\,\cong\, \mathtt{Tw}(\gl_n^{+,\ast}(\overline{C})\,,\,\bar{B})\ . \end{equation}
The right hand side of \eqref{eqinvadj4} is in natural bijection with $\Hom_{\cDGA_{k/k}}(\C({\gl}^{+,\ast}_n(\overline{C});k), B)$ (see, for example, \cite[Sec. 6.3.4]{BFPRW}). The formula for $\DRep_n^{\ob}(A)$ then follows from the Yoneda lemma. To prove the formula for $\DRep_n^{\ob}(A)^G$, note that the Lie algebra of $G$ is precisely $\gl_n^{-}(k)$ (which is the $k$-linear dual of
$\gl_n^{+,*}(k)$). Then, as in the proof of~\cite[Theorem 3.2]{BFPRW}, note that $\C({\gl}^{+,*}_n(\overline{C});k)^G$ is precisely $\C(\gl^{+,*}_n(C), \gl^{+,*}_n(k);k)$. 
The remaining formulas for representation homologies then follow immediately.
\end{proof}





\section{main results} \la{S3}

\subsection{The dihedral and hyperoctahedral bar constructions} \la{bardihhoct} 

Recall that the {\it dihedral (resp., hyperoctahedral) crossed simplicial category} $\Delta D$ (resp., $\Delta H$) is defined to be an extension of the simplicial category $\Delta$ that has the same objects as $\Delta$ and morphisms characterized by the following two properties (see \cite[6.1.4,6.3]{L}):\\

$(1)$ For each $n \geq 0$, ${\rm Aut}_{\Delta D}([n]) \cong D_{n+1}$ (resp., ${\rm Aut}_{\Delta H}([n]) \cong H_{n+1}$, where $D_{n+1}$ (resp., $H_{n+1}$) is the $(n+1)$-st dihedral (resp., hyperoctahedral) group.\\

$(2)$ Any morphism $f:[n] \to [m]$ in $\Delta D$ (resp., $\Delta H$) can be uniquely written as a composite $g \circ \sigma$ where $g \in \Hom_{\Delta}([n],[m])$ and $\sigma \in \Hom_{\Delta D}([n],[n])\cong D_{n+1}$ (resp., $\sigma \in \Hom_{\Delta H}([n],[n])\cong H_{n+1}$.\\

Note that the generating maps of the category $\Delta D^{\rm op}$ are the face maps $d_i:[n] \to [n-1], 0 \leq i \leq n$ for all $n \geq 1$, the degeneracies $s_j:[n] \to [n+1], 0 \leq j \leq n$ for all $n \geq 0$ and the generators $r_n, t_n$ of the dihedral group $D_{n+1}$ (satisfying $r_n^2=t_n^{n+1}=1$, $r_nt_nr_n=t_n^{-1}$ and viewed as automorphisms of $[n]$ for all $n \geq 0$). For $A \in \iAlg_k$, the {\it dihedral bar construction} of $A$ (see \cite[Defn. 1.4.1]{Lo}) is the functor $B^{\rm dih}A:\Delta D^{\rm op} \to {\rm Mod}_k$ given on objects by $[n] \mapsto A^{\otimes (n+1)}$ and on the generating morphisms by 
\begin{align*}
s_j(a_0 \otimes \ldots \otimes a_n) &=  a_0 \otimes \ldots \otimes a_j \otimes 1 \otimes a_{j+1} \otimes \ldots \otimes a_n\,,\qquad 0 \leq j \leq n\\
d_i(a_0 \otimes \ldots \otimes a_n) &= a_0 \otimes \ldots \otimes a_ia_{i+1} \otimes \ldots \otimes a_n\,,\qquad 0 \leq i \leq n-1\\
d_n(a_0 \otimes \ldots \otimes a_n) &= a_na_0 \otimes a_1 \otimes \ldots \otimes a_{n-1}\\
t_n(a_0 \otimes \ldots \otimes a_n) &= a_n \otimes a_0 \otimes \ldots \otimes a_{n-1}\\
r_n(a_0 \otimes \ldots \otimes a_n) &= \sigma(a_0) \otimes \sigma(a_n) \otimes \ldots \otimes \sigma(a_1)\,,
\end{align*}
where $\sigma$ is the involution on $A$. By \cite[Sec. 6.3.7]{L}, 
\begin{equation} \la{hdfh} \HD_\bullet(A)\,\cong\, {\rm Tor}^{\Delta D^{\rm op}}_\bullet(k, B^{\rm dih}A)\ . \end{equation}
For $A \in \s\iAlg_k$, the dihedral bar construction $B^{\rm dih}A$ defines a simplicial left $\Delta D^{\rm op}$-module, i.e, a functor $\Delta D^{\rm op} \to \s{\rm Mod}_k$. In this case, \eqref{hdfh} holds with $\HD_\bullet(A)$ understood to be the dihedral homology of the associated chain complex ${\rm C}(A)$ of $A$, and the ${\rm Tor}$ on the right defined as the homology $\H_\bullet[k \otimes^{\L}_{\Delta D^{\rm op}} {\rm C}(B^{\rm dih}A)]$.

In order to describe the hyperoctahedral bar construction, it convenient to use the following notation for morphisms in $\Delta H$ mimicking \cite[Sec. 2]{Au1}. Recall that $H_{n+1}:=(C_2)^{n+1} \rtimes S_{n+1}$, with $S_{n+1}$ acting on $(C_2)^{n+1}$ on the right by permuting factors. Explicitly, any element in $H_{n+1}$ is of the form $(a_0,\ldots, a_n;\tau)$ with $a_i \in C_2$ and $\tau \in S_{n+1}$, with 
$$(a_0,\ldots, a_n;\tau) \cdot (b_0,\ldots,b_n;\mu):= (a_{\mu(0)}b_0,\ldots, a_{\mu(n)}b_n; \tau\mu) \ .$$
Any morphism $f:[n] \to [m]$ in $\Delta H$ can be uniquely written as a ``tensor product" of $m+1$ noncommutative monomials in the $2m+2$ formal variables $x_0,\ldots,x_n,x_0^s,\ldots,x_n^s$. Explicitly, if $f=g \circ \sigma$ where $g \in \Hom_{\Delta}([n],[m])$ and $\sigma=(a_0,\ldots,a_n;\tau)$ then we write $f=X_0 \otimes X_1 \otimes \ldots\otimes X_m$, where 
$$X_i:= \begin{cases}
       x_{\tau^{-1}(j_1)}^{a_{\tau^{-1}(j_1)}} \ldots  x_{\tau^{-1}(j_r)}^{a_{\tau^{-1}(j_r)}} & \text{ if } g^{-1}(i)=\{j_1< \ldots <j_r\} \subseteq [n]\\
       1 & \text{ if } g^{-1}(i) = \emptyset \ .
\end{cases}$$
For example, if $f:[4] \to [3]$ is given by $f=g \circ \sigma$, where $g:[4] \to [3]$ is given by $g(0)=g(1)=0$, $g(2)=g(3)=1$ and $g(4)=3$ and $\sigma=(1,s,s,s,1;\tau)$ where $\tau=(0 \,1)(2\,4\,3)$ (as a product of disjoint cycles in $S_5$), then 
$f$ is represented by $x_1^sx_0 \otimes x_3^sx_4 \otimes 1 \otimes x_2^s$. Composition is then defined by a substitution rule ( keeping in mind that if $X=x_{i_1}^{\alpha_{i_1}}\ldots x_{i_r}^{\alpha_{i_r}}$ where $\alpha_{i_j} \in C_2$ for all $j$, then $X^s=x_{i_r}^{s\alpha_{i_r}} x_{i_{r-1}}^{s\alpha_{i_{r-1}}}\ldots x_{i_1}^{s\alpha_{i_1}}$): for instance, if $h:[3] \to [3]$ is defined by $h=l \circ \varphi$, where $l(0)=l(1)=1$, $l(2)=l(3)=2$ and $\varphi=(1,s,1,s;(0\,1\,2))$ then $h=1 \otimes x_2x_0 \otimes x_1^sx_3^s \otimes 1$, and 
$$h \circ f = (1 \otimes X_2X_0 \otimes X_1^sX_3^s \otimes 1) \circ (\overbrace{x_1^sx_0}^{X_0} \otimes \overbrace{x_3^sx_4}^{X_1} \otimes \overbrace{1}^{X_2} \otimes \overbrace{x_2^s}^{X_3}) = 1 \otimes x_1^sx_0 \otimes x_4^sx_3x_2 \otimes 1 \ .$$

Using the above notation, the {\it hyperoctahedral bar construction} of an involutive algebra $A \in \iAlg_k$ is defined as the functor $B_{\rm hOct}(A):\Delta H \to {\rm Mod}_k$ that is given on objects by
\begin{equation} \la{hobarobj} B_{\rm hOct}(A)([n]) = A^{\otimes (n+1)}\,, \qquad \forall\, n \geq 0\,,\end{equation}
and on morphisms by the following `substitution formula'
\begin{equation} \la{hobarmor} B_{\rm hOct}(A)(f)(a_0 \otimes \ldots \otimes a_n) = a_{i_1}^{\alpha_{i_1}}\ldots a_{i_r}^{\alpha_{i_r}} \otimes \ldots \otimes a_{k_1}^{\alpha_{k_1}}\ldots a_{k_s}^{\alpha_{k_s}}  \end{equation}
if $f \in \Hom_{\Delta H}([n],[m])$ is represented by $x_{i_1}^{\alpha_{i_1}}\ldots x_{i_r}^{\alpha_{i_r}} \otimes \ldots \otimes x_{k_1}^{\alpha_{k_1}}\ldots x_{k_s}^{\alpha_{k_s}}$. Here, the exponents $\alpha_j$ lie in $C_2$, and for $a \in A$, $a^s:=\sigma(a)$, where $\sigma$ is the involution on $A$. More generally, for $A \in \s\iAlg_k$, $B_{\rm hOct}(A):\Delta H \to \s{\rm Mod}_k$ is a {\it simplicial} $\Delta H$-module. 

\begin{definition} \la{defho}
For $A$ in $\s\iAlg_k$, the {\it hyperoctahedral homology} ${\rm HO}_{\bullet}(A)$ is the homology 
$${\rm HO}_{\bullet}(A)\,:=\,{\rm Tor}^{\Delta H}_{\bullet}(k, B_{\rm hOct}(A)) \ . $$
Recall that for an arbitrary $A \in \s\iAlg_k$, the latter Tor is defined as the homology $\H_\bullet[k \otimes^{\L}_{\Delta H} {\rm C}(B_{\rm hOct}(A))]$, where ${\rm C}$ stands for the associated chain complex.
\end{definition}

\subsection{Derived character maps revisited}
\la{S3.1}
We now construct derived character maps relating hyperoctahedral homology to (involutive) representation homology. 

Let $R \stackrel{\sim}{\to} A$ be a semi-free simplicial resolution of $A \in  \s\iAlg_k$. Applying the unit of the adjunction of Proposition \ref{simpadj} $(i)$ to $R$, we obtain a map of involutive simplicial  $k$-algebras $\,\pi_n: R \to M_n[(R)_n]\,$. Composing $\pi_n$ with the trace $\,M_n[(R)_n] \to (R)_n\,$ gives a map of simplicial $k$-modules $\,R \to (R)_n \,$. The latter
factors through the quotient $R_{\natural}:=\,R/\left( [R,R]\oplus (R-\sigma R)\right)\,$, inducing the simplicial $k$-linear map
\begin{equation}\la{TrR}
\Tr_n(R):\ R_{\natural} \to (R)^{\ob}_n \ .
\end{equation}
\begin{lemma} \la{hc0}
There is a natural isomorphism of simplicial $k$-modules 
$$
k \otimes_{\Delta D^{\rm op}} B^{\rm dih}R\, \cong\, R_{\natural} \ .
$$
\end{lemma}
\begin{proof}
It suffices to check that $\,k \otimes_{\Delta D^{\rm op}} B^{\rm dih}R_m \cong (R_m)_{\natural}\,$ for each $m$, but this follows immediately from formula \eqref{hdfh} and the fact that $\HD_0(B) \cong (B)_{\natural}$ for any $B \in \iAlg_k$.
\end{proof}
By Lemma \ref{hc0}, we obtain the following composition of morphisms in $\mathscr{D}(k)$, where ${\rm C}$ stands for the associated chain complex,
\begin{equation} \la{derivedch1} k \otimes^{\L}_{\Delta D^{\rm op}} {\rm C}(B^{\rm dih}R) \to  k \otimes_{\Delta D^{\rm op}} {\rm C}(B^{\rm dih}R) \cong {\rm C}(R_{\natural}) \, \xrightarrow{\Tr_n(R)}\, {\rm C}[(R)^{\ob}_n] \,,
\end{equation}
%
where the last map is defined from functoriality of $\rm C$ by applying it to \eqref{TrR}. Note that, if $k$ is a field, $B^{\rm dih}R$ is weakly equivalent to $B^{\rm dih}A$ by K\"{u}nneth's Theorem. In this case, \eqref{derivedch1} induces on homology a natural graded $k$-linear map
\begin{equation} 
\la{derivedch2}
\Tr_n(A)_\bullet: 
\HD_\bullet(A) \to \HRi(A, k^n)\ , 
\end{equation}
which is the map \eqref{char2}.
On the other hand, there is an inclusion functor $i:\Delta D^{\rm op} \hookrightarrow \Delta H$ defined in  \cite[3.3]{G21}. On the generating morphisms of $\Delta D^{\rm op}$ (and with morphisms in $\Delta H$ described as in Section \ref{bardihhoct} above, it is given by 
\begin{align*}
s_j & \mapsto   x_0 \otimes \ldots \otimes x_j \otimes 1 \otimes x_{j+1} \otimes \ldots \otimes x_n\,,\qquad 0 \leq j \leq n\\
d_i &\mapsto x_0 \otimes \ldots \otimes x_ix_{i+1} \otimes \ldots \otimes x_n\,,\qquad 0 \leq i \leq n-1\\
d_n &\mapsto x_nx_0 \otimes x_1 \otimes \ldots \otimes x_{n-1}\\
t_n  &\mapsto x_n \otimes x_0 \otimes \ldots \otimes x_{n-1}\\
r_n  &\mapsto x_0^s \otimes x_n^s \otimes \ldots \otimes x_1^s\ .
\end{align*}
Clearly,
\begin{equation} \la{cyctosym}B^{\rm dih}A = i^{*} \HA  \end{equation}
for any algebra $A \in \s\iAlg_k$. Hence, the functor $i$ induces a natural map 
\begin{equation} \la{hctohs} i^{*}:\ \HD_\bullet(A) = {\rm Tor}^{\Delta D^{\rm op}}_{\ast}(k, B^{\rm dih}A) \,\to\, {\rm Tor}^{\Delta H}_\bullet(k, \HA) = {\rm HO}_\bullet(A)\ .\end{equation}
relating dihedral homology to hyperoctahedral homology. 

Next, we recall that, for any involutive algebra $ A \in \iAlg_k$,
 Graves \cite[Corollary 4.5]{G22} constructed a graded commutative $k$-algebra structure on $\HO_\bullet(A)$. The following lemma extends this structure to simplicial involutive algebras, i.e., to any $A \in \s\iAlg_k$.
\begin{lemma} \la{algstr}
Graves' structure of a graded commutative algebra on ${\rm HO}_\bullet(A)$ extends naturally to simplicial involutive $k$-algebras $($defined over an arbitrary commutative ring $k)$.  
\end{lemma}
\begin{proof}
Recall \cite[Def. 5.2, Thm. 5.4]{G22} that Graves constructed a functor that we will call $G(\mbox{--}):\iAlg_k \to \mathrm{sE}_\infty\Alg_k$, taking $A\to C(\Delta H_+,\H_{A+})$, such that  $\pi_\bullet[GA] \cong {\rm HO}_\bullet(A)$ for any
$A \in \iAlg_k$. Here, $\mathrm{sE}_\infty\Alg_k$ denotes the category of $E_{\infty}$-algebra objects in the category of simplicial $k$-modules (which we shall refer to as {\it simplicial $E_{\infty}$-algebras}). Here, the term `$E_{\infty}$-algebra' refers to an algebra over the simplicial $E_{\infty}$-operad $\mathcal{D}_{\bf Mod}$, an analogue of the Barratt-Eccles operad in \cite{G21} (with notation mirroring \cite[Remark 2.8]{Au2}). Now let $A \in \s\iAlg_k$ be an arbitrary simplicial involutive $k$-algebra. Then, $GA$ is a bisimplicial $E_\infty$-algebra whose diagonal $k$-module is an $E_{\infty}$-algebra. Let $\Delta H_+$ be the category obtained by adjoining an initial object $[-1]$ to $\Delta H$. Now recall that, by definition (see \cite[Defn. 2.5, Defn. 6.3]{G21} and \cite[Theorem 5.4]{G22}),
$$GA = k[N(\mbox{--}\downarrow \Delta H_+)] \otimes_{\Delta H_+} B_{\rm hOct,+}A \,,
$$
where $N(\mbox{--})$ stands for the simplicial nerve functor and $B_{\rm hOct,+}A$ is the extension of $\HA$ to a $\Delta H_+$-module as in \cite[Def. 6.3]{G21}. Since the simplicial right $\Delta H_+$-module $k[N(\mbox{--}\downarrow \Delta H_+)]$ is a levelwise projective resolution of $k$, the homology of (the Dold-Kan normalization of) the diagonal of $GA$ is precisely
$${\rm Tor}^{\Delta H_+}_\bullet(k, B_{\rm hOct,+}A)\, \cong\, {\rm Tor}^{\Delta H}_\bullet(k, \HA) = : \HO_\bullet(A)\ . $$
The above isomorphism is \cite[Thm 6.4]{G21}. This proves the desired lemma.
\end{proof}

Now, we come to the main result of this section.
\begin{proposition} \la{factorTr}
Assume that $k$ is a field. The one-dimensional derived character map \eqref{derivedch2} factors through \eqref{hctohs}:
\begin{equation*}
 \la{HSR2} 
 \begin{diagram}[small]
 \HD_\bullet(A)& & \rTo^{\ \Tr(A)_\bullet\ } & & \HRi(A, k)\\
& \rdTo_{i^*} &                         & \ruDotsto &\\
&                  & \HO_\bullet(A)        & &
 \end{diagram}
 \end{equation*}
inducing a {\rm homomorphism} of graded commutative $k$-algebras with respect to the Graves structure on $\HO_\bullet(A)$:
\begin{equation}
\la{SRAA}
\HOR(A)_\bullet:\ \HO_\bullet(A)\, \to\, \HRi(A, k) \ 
\end{equation}
\end{proposition}
\begin{proof}
Given an algebra $A$, we can construct the map \eqref{SRAA} as follows. Fix a semi-free simplicial resolution $R \stackrel{\sim}{\to} A$ in $\s\iAlg_k $ and consider the natural map
\begin{equation}\la{dttot} k \otimes^{\L}_{\Delta H} {\rm C}(\HR) \to  k \otimes_{\Delta H} {\rm C}(\HR)  \end{equation}
where ${\rm C}$ stands for the associated chain complex. By the K\"{u}nneth Theorem, $\HR \xrightarrow{\sim} \HA$ is a weak equivalence. Hence $\H_\bullet[k \otimes^{\L}_{\Delta H} {\rm C}(\HR)] \cong \HO_\bullet(A)$. Note that there are natural isomorphisms of simplicial $k$-vector spaces $k \otimes_{\Delta H} \HR \cong R_{\rm ab,\sigma} \cong (R)_1^{\ob}$, where $R_{\rm ab,\sigma}$ denotes the quotient of the abelianization $R_{\rm ab}$ of $R$ by the ideal generated by the image of $\id -\sigma$ ($\sigma$ being the induced involution on $R_{\rm ab}$). The first of these isomophisms is by \cite[Thm 5.8]{G22}, while the second isomorphism is by definition (see Section \ref{SDRep}). Hence $\H_\bullet[k \otimes_{\Delta H} {\rm C}(\HR)] \cong \HRi(A,k)$.Thus, \eqref{dttot} induces on homology a graded $k$-linear map $\SR(A)_\bullet:\HO_\bullet(A) \to \HRi(A,k)$. 
 
 Next, note that the map $\Tr_1(R):R_{\natural} \to (R)^{\ob}_1$ inducing $\Tr(A)_\bullet$ is just the canonical quotient to $(R)^{\ob}_1:=R_{\rm ab, \sigma}$, and the map $i^{\ast}$ coincides with $ \Tr_1(R) : k \otimes_{\Delta D^{\rm op}} B^{\rm dih}R \to k \otimes_{\Delta H} \HR$ when restricted to the non-derived tensor products. Hence the following diagram commutes in 
 $ {\mathscr D}(k) $: 
 $$  
 \begin{diagram}[small]
   k \otimes^{\L}_{\Delta D^{\rm op}} {\rm C}(B^{\rm dih}R)  & \rTo^{i^{\ast}} & k \otimes^{\L}_{\Delta H} {\rm C}(\HR)\\
    \dTo & & \dTo\\
   k \otimes_{\Delta D^{\rm op}} {\rm C}(B^{\rm dih}R) & \rTo^{\Tr_1(R)} &  k \otimes_{\Delta H} {\rm C}(\HR)\\
 \end{diagram}
 $$
 which gives the natural isomorphism $ \HOR(A)_\bullet \circ i^* \cong \Tr(A)_\bullet $ on homology.

 It remains to prove that the map \eqref{SRAA} constructed above is an algebra homomorphism. To this end,
consider the bisimplicial $E_{\infty}$-algebra $G{R}_\bullet$, where $G(\mbox{--})_\bullet:\iAlg_k \to \mathrm{sE}_\infty\Alg_k$ is Grave's functor (see the proof of Lemma \ref{algstr}) .The total (i.e. diagonal) simplicial vector space of $GR_\bullet$ is a (simplicial) $E_{\infty}$-algebra.
Recall that 
$$GR_{\bullet}  = k[(N(\mbox{--}\downarrow \Delta H_+))_i] \otimes_{\Delta H_+} B_{\rm hOct,+}R \ .
$$
This is a (bi)simplicial vector spaceoverl, and as in (the proof of) \cite[Lemma 3.1]{Au2}, 
$$GR_{i} \,\cong\, \bigoplus_{[n]\stackrel{\id}{\to} [n] \stackrel{f_1}{\to} [n_1] \stackrel{f_2}{\to} \ldots \stackrel{f_i}{\to} [n_i]} R^{\otimes n}\,,$$
where the direct sum runs over chains of morphisms $[n]\stackrel{\id}{\to} [n] \stackrel{f_1}{\to} [n_1] \stackrel{f_2}{\to} \ldots \stackrel{f_i}{\to} [n_i]$ of length $i+1$ in $\Delta H_+$. Since 
$$ GA_{i} \,\cong\, \bigoplus_{[n]\stackrel{\id}{\to} [n] \stackrel{f_1}{\to} [n_1] \stackrel{f_2}{\to} \ldots \stackrel{f_i}{\to} [n_i]}  A^{\otimes n} \,,$$
$GR_i \to GA_i$ is a weak equivalence for all $i$ by the K\"{u}nneth theorem. Hence, $GR_\bullet \to GA_\bullet$ is a weak equivalence of simplicial $E_{\infty}$-algebras.
Hence, ${\tt diag}(G{R}_\bullet) \xrightarrow{\sim} {\tt diag}(GA_\bullet)$ is a quasi-isomorphism of $E_{\infty}$-algebras. On the other hand, the natural map $G(R_i) \to \HO_0(R_i) \cong (R_i)_{\rm ab,\sigma}$ is an $E_{\infty}$-algebra homomorphism by \cite[Thm 6.8]{G22}, where $R_i$ is the algebra of $i$-simplices on $R$. 
Hence the map ${\tt diag}(GR_\bullet) \to R_{\rm ab, \sigma}$ is a map of $E_{\infty}$-algebras.
Since the map $C(R_i) \to \HO_0(R_i) \cong (R_i)_{\rm ab,\sigma}$ represents the map $k \otimes^{\L}_{\Delta H} \HR_i \to k \otimes_{\Delta H} \HR_i$, the map induced on homology (of the associated chain complexes) by the map ${\tt diag}(GR_\bullet) \to R_{\rm ab, \sigma}$ is the same as that induced on homologies by \eqref{dttot}, which is $\HOR(A)_\bullet$. Since the map ${\tt diag}(GR) \to R_{\rm ab, \sigma}$ is a homomorphism of simplicial $E_{\infty}$-algebras, $\HOR(A)_\bullet$ is a morphism of graded commutative algebras. This finishes the proof of the proposition.
\end{proof}
%

%
%

\subsection{Hyperoctahedral homology vs representation homology} 
\la{S3.2}
{\it From now on, we will assume once again that $k$ is a field of characteristic $0$.}

\begin{theorem} \la{T1a}
For any simplicial algebra $A \in \s\iAlg_k$, the map $\HOR(A)_\bullet$ of Proposition~\ref{factorTr} is an isomorphism of graded commutative $k$-algebras.
\end{theorem}
The proof of Theorem~\ref{T1a} is based on the following key proposition.
\begin{proposition} \la{pC}
Let V be a k-vector space. Then
$$\Tor_i^{\Delta H}(k, \HH(F(V)))\simeq  \begin{cases}
\Sym(V)
   & \text{ if }  i=0 \\
0
   & \text{ if } i\not = 0
\end{cases}$$
\end{proposition}
Our proof mirrors \cite{BR22c} and relies heavily on results of Graves in \cite{G21}, themselves referencing rather technical computations of Lemma~33 and Lemma~36 of \cite{Au1}.
\begin{proof}
Choose a basis in V, say a set $X\subset V$, so that $F(V)\cong F(k[X]) \cong k[J(X)]$, where $J(X)$ denotes the free involutive monoid generated by the set, that is, the classical James construction applied to the (based) set $(C_2 \times X)_+:= (C_2 \times X) \sqcup \{\ast\}$ (with $\ast$ regarded as the basepoint). We regard $J(X)$ as the $C_2$-James construction of the discrete (based) space $X_+$, where, as in \cite{G21}, the classical James construction as an endofunctor of $Top_*$ has been extended an endofunctor of $C_2-Top_*$ according to the rule $t(x_1,...,x_n)=(tx_n,...,tx_1)$. Then, by \cite[Lemma 7.13]{G21} and (the proof of) \cite[Lemma 15]{G21}, we have 
$$
{\rm HO}_\bullet(F(V))\, \cong\, {\rm HO}_\bullet(k[J(X)])\, \cong \, \H_\bullet[\mathcal{N}X; k]\ ,
$$ 
where 
$$\mathcal{N}X \,:=\, \coprod_{n \geq -1} N([n] \downarrow \Delta H_+) \times_{H_{n+1}} (C_2 \times X)^{n+1}\ . $$
%
%
Here, $\Delta H_+$ is the category obtained by adding to $\Delta H$ the initial object $[-1]$, and $N(\mbox{--})$ in the formula above stands for the nerve. By \cite[Lemma 7.16]{G21}, there is a weak equivalence $\mathcal{N}X \simeq \mathcal{D}X$, where 
$$\mathcal{D}X\,:=\, \coprod_{n \geqslant 0} E_{\ast}H_{n} \times_{H_{n}} (C_2 \times  X)^{n}\ . $$
Hence,
\begin{align*} 
{\rm HO}_\bullet(F(V)) &\cong\,\H_\bullet[\mathcal{D}X; k] \\
                       & \cong\, \Moplus_{n=0}^{\infty} \H_\bullet\left(k[E_{\ast}H_{n}] \otimes_{k[H_n]} k[C_2 \times X]^{\otimes n}\right)\\
                       & \cong\,  \Moplus_{n=0}^{\infty} \H_\bullet \left( k[E_{\ast}H_{n}] \otimes_{k[H_n]} k[k[C_2] \otimes V]^{\otimes n}\right)\\
                       & \cong\, \Moplus_{n=0}^{\infty} \H_\bullet \left( k \otimes^{\L}_{k[H_n]} k[k[C_2] \otimes V]^{\otimes n}\right)\\
                      & \cong\,  \Moplus_{n=0}^{\infty} \H_\bullet \left(H_n,(k[C_2] \otimes V)^{\otimes n}\right)\\
                      &\cong\, \Moplus_{n=0}^{\infty} \H_0 \left(H_n,(k[C_2] \otimes V)^{\otimes n}\right)\\
                      & \cong\, {\rm HO}_0(F(V))\\
                      & \cong\,\Sym(V)\ .\\
\end{align*}
The third isomorphism is above is because $X$ is a discrete basis of $V$. Since $E_{\ast}H_n$ is contractible for all $n$, $k[E_\ast H_n]$ is weakly equivalent to $k$. Further, $k[E_\bullet \H_n]$ is a levelwise projective simplicial right $\Sigma_n$-module. This explains the fourth isomorphism above, from which the fifth isomorphism above follows by definition.
The sixth isomorphism above holds because  $k$ is a field of characteristic $0$ (and the groups $H_n$ are finite), and the last isomorphism above is by \cite[Theorem 6.8]{G22}. This completes the proof of the desired lemma.
\end{proof}
Now, we can proceed with

\begin{proof}[Proof of Theorem \ref{T1a}]
Let $R \stackrel{\sim}{\to} A$ be a semi-free simplicial resolution. Then, the (left) $\Delta H$-module of $n$-simplices in $\HR$ is of the form $\HH(F(V))$ for some involutive vector space $V$. It follows from Proposition \ref{pC} that the following natural chain map is a quasi-isomorphism:
\begin{equation} \la{srr} k \otimes^{\L}_{\Delta H} {\rm C}(\HR) \to k \otimes_{\Delta H} {\rm C}(\HR) \cong {\rm C}(R_{\rm ab, \sigma})\,, \end{equation}
where last isomorphism is a consequence of \cite[Theorem 6.8]{G22}. Since \eqref{srr} is precisely the map that induces  $\HOR(A)_\bullet$ on homology, the $\HOR(A)_\bullet$ is an isomorphism of graded $k$-vector spaces.
By Proposition \ref{factorTr}, we then conclude that this isomorphism of vector spaces
is actually an isomorphism of $k$-algebras.
\end{proof}
For further applications, it will be convenient for us to restate the result of Theorem~\ref{T1a} for differential graded algebras. Recall once again (see proof of Proposition \ref{simp}) that since $\iAlg_k$ is the category of algebras over an operad, the monoidal Dold-Kan correspondence (see \cite[Theorem A.1]{BRYII}) implies that there is a Quillen equivalence 

$$ N^{\ast}\,:\,{\iDGA_k}^+ \rightleftarrows \s\iAlg_k\,:\,N\,,\qquad N^{\ast}\,:\,\cDGA_k^+ \rightleftarrows \s\cAlg_k\,:\,N$$

where ${\iDGA_k}^+$ (resp., $\cDGA_k^+$) stands for the category of involutive DG algebras (resp., commutative DG algebras) concentrated in nonnegative homological degree and $N$ stands for the normalized chain complex functor.
\begin{proposition}\la{tA}
For any $A \in {\iDGA_k}^+$, there is an isomorphism of graded $k$-vector spaces $\HO_\bullet(A) \cong \HO_\bullet(N^{\ast}A)$. As a consequence, there is an isomorphism of graded commutative $k$-algebras
$${\rm HO}_\bullet(A) \cong \HRi(A,k)\,,$$
where $\HO_\bullet(A)$ inherits its algebra structure from $\HO_\bullet(N^*A)$.
\end{proposition} 
\begin{proof}
There are natural quasi-isomorphisms of complexes of $\Delta H$-modules
$$\HA \stackrel{\sim}{\to} \HH (NN^*A) \stackrel{\sim}{\to} N(\HH (N^* A))\,, $$
where the the first arrow --- induced by the natural weak equivalence $A \to NN^*A$ of DG algebras --- is a weak equivalence by the K\"{u}nneth Theorem, and the second arrow
is a Eilenberg-Zilber map, which is also known to be a quasi-isomorphism by \cite[Sect. A.2]{BRYII}.  Thus, on homology, we get natural isomorphisms of graded $k$-vector spaces
$$\HO_\bullet(A):= \H_\bullet[k \otimes^{\L}_{\Delta H} \HA] \,\cong\, \H_\bullet[k \otimes^{\L}_{\Delta H} N(\HH (N^* A))] =: \HO_\bullet(N^*A)\ . $$
proving the first statement in the desired proposition. 

Next, since the functor $(\,\mbox{--}\,)_{\rm ab,\sigma}$ is left adjoint, it commutes with $N^*$: i.e., there is a commutative diagram of functors 
\begin{equation} \la{abnst}
\begin{diagram}[small]
{\iDGA_k}^{+} & \rTo^{N^*} & \s\iAlg_k\\
 \dTo^{(\mbox{--})_{\rm ab,\sigma}} & &  \dTo_{(\mbox{--})_{\rm ab,\sigma}}\\
 \cDGA_k^{+} & \rTo^{N^*} & \s\cAlg_k
 \end{diagram}\ .
\end{equation}
Let $R \stackrel{\sim}{\to} A$ be a semifree resolution of $A$ in ${\iDGA_k}^{+}$. There are isomorphisms (that are natural in $R$) 
\begin{equation} \la{nnstrh} \HRi(A, k) \,\cong\, \H_\bullet(R_{\rm ab,\sigma}) \,\cong\, \pi_{\bullet}[N^*(R_{\rm ab,\sigma})] \,\cong\, \pi_\bullet[(N^* R)_{\rm ab,\sigma}]\,\cong\,\HRi(N^*A, k)\ . \end{equation}
The first isomorphism above is by definition (see Section \ref{SDRep}), the second isomorphism above is by \eqref{abnst}, and the last isomorphism above is by definition (see Definition \ref{defsrep}) and the fact that $N^*R \to N^*A$ is a semi-free resolution of $A$ (by \cite[Prop. A.2]{BRYII} and since $N^*$ preserves weak-equivalences).  This shows that there is an isomorphism of graded commutative $k$-algebras
\begin{equation} \la{hrlnast} \HRi(A,k) \cong \HRi(N^*A, k)\ . \end{equation}
The second statement of the desired proposition now follows immediately from Theorem \ref{T1a}.
\end{proof}

\subsection{Hyperoctahedral homology of universal enveloping algebras}
\la{S3.3}

Let $\mfa \in \DGL_k$ be a DG Lie algebra. 
\begin{definition} \la{iliealg}
An {\it involution} $\sigma: \mfa \to \mfa$ is a map of complexes such that $\sigma^2=\id$, and 
$\sigma \circ [\mbox{--},\mbox{--}]\,=\,- [\mbox{--},\mbox{--}] \circ (\sigma \otimes \sigma)\,:\,\mfa \otimes \mfa \to \mfa$. Explicitly, for all $a,b \in \mfa$, we have 
$$\sigma([a,b])\,=\,-[\sigma(a), \sigma(b)]\ .$$
\end{definition}

Note that if $\mfa \in \DGL_k$ is equipped with an involution $\sigma$, then the $-1$-eigenspace $\mfa^{-}$ of the involution on $\mfa$ is a DG Lie subalgebra of $\mfa$. For example, if $\mfa=\mathfrak{gl}_n(k)$ with the transpose as involution, then $\mfa^{-}=\mathfrak{o}_n(k)$.

\begin{remark} \la{indinv}
Clearly, if $\sigma$ is an involution on $\mfa \in \DGL_k$, then the involution on the tensor algebra $T_k \mfa \in \DGA_k$ whose restriction on $\mfa$ is $\sigma$ descends to an involution on the universal enveloping algebra $U\mfa$ of $\mfa$. Thus, $U\mfa \in \iDGA_k$. Abusing notation, we shall continue to denote this notation by $\sigma$.
\end{remark}

\begin{lemma} \la{indinvs}
Let $\mfa \in \DGL_k$. Then, the Chevalley-Eilenberg coalgebra $\C(\mfa;k)$ is a co-commutative skew involutive DG coalgebra. Consequently, the map $\cb(\C(\mfa;k)) \stackrel{\sim}{\to} U\mfa$ is a (cofibrant) resolution in the category $\iDGA_{k/k}$ of (augmented) involutive DG $k$-algebras.
\end{lemma}
\begin{proof}
There is a skew involution on the graded (co-commutative) coalgebra $\C(\mfa;k)$ whose corestriction to $\mfa[1]$ is the involution on $\mfa$. That this skew involution is compatible with the differential is shown by a direct computation similar to (the proof of) Proposition \ref{skinvtoinv}. It follows that the twisting cochain $\tau: \C(\mfa;k)[-1] \to \mfa$ intertwines the involution on $\mfa$ with the skew-involution on $\C(\mfa;k)$. Consequently, 
$T_k\tau: \cb(\C(\mfa;k)) \to T_k\mfa$ is a morphism of involutive graded algebras. Since the map $\cb(\C(\mfa;k)) \to U\mfa$ factors (as a map of graded algebras, forgetting differentials) through the map $T_k\tau$, it is a map of involutive graded algebras. That this map is a quasi-isomorphism of DG algebras is well known (see, e.g., \cite[Theorem 6.1]{BFPRW} and the references therein). Finally $\cb(C)$ is cofibrant in $\iDGA_{k/k}$ by Theorem \ref{kosdualinv}. 
\end{proof}

It is easily seen that if $\mfa \in \DGL_k$ is involutive, then the induced skew involution on $C:=\C(\mfa;k)$ restricts to a skew involution on its co-augmentation coideal $\overline{C}$. There is therefore an induced linear map on the (reduced) Chevalley-Eilenberg homology $\rH_\bullet(\mfa;k)$ which squares to the identity. We let $\rH^{\rm ev}_{\bullet}(\mfa;k)$ denote the $+1$-eigenspace of this operator induced by the involution on $\mfa$. We now prove the following important consequence of Theorem \ref{T1a}.
\begin{theorem} 
\la{houainv}
Let $\mfa \in \DGL^+_k$ be an involutive DG Lie algebra over $k$ and let $ U\mfa \in \DGA^+_k $ denote its universal enveloping algebra (equipped with the involution induced by that on $\mfa$). 
There is an isomorphism of graded commutative $k$-algebras 
$$
{\rm HO}_\bullet(U\mfa) \,\cong\, \left(\Lambda_k[\rH^{\rm ev}_{\bullet+1}(\mfa;k)]\right)\,\ .
$$
\end{theorem}
\begin{proof}
Since $\overline{C}$ is cocommutative, the Lie cobracket on $\overline{C}$ vanishes (see Section \ref{CEThm}). Hence, the Lie cobracket on $\overline{C}^+$ vanishes. Further note that $\gl^{+,\ast}_1(\overline{C}) \cong \overline{C}^+$ (as Lie coalgebras). Hence, there is an isomorphism of graded commutative algebras $\C(\gl_1^{+,\ast}(\overline{C});k) \cong \Lambda_k(\overline{C}^+[-1])$. By Theorem \ref{invmain},
\begin{equation} \la{hr1ua} \HRi(U\mfa,k) \,\cong\, \H_\bullet[\Lambda_k(\overline{C}^+[-1])]\ .\end{equation}
 It follows that there is a chain of natural isomorphisms 
\begin{equation}
\la{isoLie}
{\rm HO}_\bullet(U\mfa) \,\cong\, \HRi(U\mfa,k) \,\cong\,  \H_\bullet[\Lambda_k(\overline{C}^+[-1])] \,\cong\, 
\Lambda_k[\rH^{\rm ev}_{\bullet+1}(\mfa;k)]\ .
\end{equation}
The first isomorphism in \eqref{isoLie} follows from Proposition \ref{tA}, the second is \eqref{hr1ua}, and the last one is a consequence of the well-known fact that the graded symmetric algebra functor $\Lambda_k$ commutes with homology when $k$ is a field of characteristic $0$ (see, e.g. \cite[Appendix B, Prop. 2.1]{Q}). This proves the desired theorem.
\end{proof}

\begin{example}
Let $V$ be an finite dimensional $k$-vector space equipped with an involution (i.e., a linear map $\sigma: V \to V$ such that $\sigma^2=\id$). Let $V_+$ and $V_{-}$ denote the $+1$ and $-1$ eigenspaces of $\sigma$ respectively. Then $\Sym(V)$ acquires an involution whose restriction to $V$ is $\sigma$. For any (graded) $k$-vector space $W$, let $\Lambda^{\rm odd}(W)$ (resp., $\Lambda^{\rm ev}(W)$) denote the direct sum of all odd (resp., even) graded symmetric powers of $W$.
\begin{cor} \la{hosym}
There is an isomorphism of graded commutative $k$-algebras
$${\rm HO}_\bullet(\Sym(V))\,\cong\, \bigwedge_k\left[\big(\Lambda^{\rm odd}(V_+[1]) \otimes \Lambda(V_-[1])\big)[-1]\right]\ .$$    
\end{cor}
\begin{proof}
Let $\mfa=V$ equipped with the trivial Lie bracket. Then, $\mfa$ is an involutive Lie algebra, and $U\mfa \cong \Sym(V)$ as involutive algebras. $\C(\mfa;k)=\Lambda_k(V[1])$, with trivial differential and skew involution whose corestriction to $V[1]$ is the involution $\sigma$ on $V$. To compute $\C(\mfa;k)_+$, we consider its graded $k$-linear dual $E:=\Lambda_k(V^{\ast}[-1])$. This graded commutative $k$-algebra is equipped with a skew involution whose restriction to $V^{\ast}[-1]$ is the operator induced by $\sigma$ on $V^{\ast}$. Abusing notation, denote this skew involution on $E$ by $\sigma$ as well. If $V^\ast_+$ and $V^\ast_-$ denote the $+1$ and $-1$ eigenspaces of $\sigma$ on $V^\ast$ respectively, $E \cong \Lambda(V^{\ast}_+[-1]) \otimes \Lambda(V^{\ast}_-[-1])$. Since $\sigma$ is a skew involution on the graded commutative algebra $E$, 
$$\sigma(x_1 \ldots x_n)\,=\,(-1)^{n-1}\sigma(x_1)\ldots \sigma(x_n)\,,\qquad \forall \,x_1,\ldots\,x_n\,\in\,E\ .$$
Hence, on the direct summand $\Lambda^i(V^{\ast}_+[-1]) \otimes \Lambda^j(V^{\ast}_-[-1])$ of $E$, $\sigma$ acts by multiplication by $(-1)^{i+j-1}(-1)^j=(-1)^{i-1}$. It follows that the $+1$-eigenspace of the skew involution on $E$ is $\Lambda^{\rm odd}(V^{\ast}_+[-1]) \otimes \Lambda(V^{\ast}_-[-1])$. Dually, $\C(\mfa;k)_+\,=\,\Lambda^{\rm odd}(V_+[1]) \otimes \Lambda(V_-[1])$. Hence, $\rH^{\rm ev}_\bullet(\mfa;k) \,\cong\, \Lambda^{\rm odd}(V_+[1]) \otimes \Lambda(V_-[1])$. The desired lemma now follows from Theorem \ref{houainv}.
\end{proof}

\begin{remark} Note that the commutative square of functors
$$
\begin{diagram}[small]
\Delta C^{\rm op} &\rTo^{i}&\Delta S\\
 \dInto & & \dInto\\
 \Delta D^{\rm op} &\rTo^{i}& \Delta H\\
\end{diagram}
$$
induces a commutative diagram of homologies 
\begin{equation} \la{hcdso}
\begin{diagram}[small]
\overline{{\rm HC}}_\bullet(A) &\rTo^{i^*} & {\rm HS}_\bullet(A)\\
\dTo & & \dTo\\
\overline{{\rm HD}}_\bullet(A) &\rTo^{i^*} & {\rm HO}_\bullet(A)\\
\end{diagram}
\end{equation}
for any $A \in \iDGA_{k/k}$. For $A=\Sym(V)$ where $V$ is equipped with an involution as above, by the Hochschild-Kostant-Rosenberg theorem we have 
\begin{equation*}
\overline{\rm HC}_\bullet(A)\,\cong\, \Omega^{\bullet}_{A/k}/d\Omega^{\bullet-1}_{A/k}\,,\qquad \overline{\rm HD}_\bullet(A)\,\cong\,(\Omega^{\bullet}_{A/k})_+/(d\Omega^{\bullet-1}_{A/k})_+\,,\end{equation*}
where $(\Omega^p_{A/k})_+$ (resp., $(d\Omega^{p-1}_{A/k})_+$) stands for the subspace of $\Omega^p_{A/k}$ (resp., $d\Omega^{p-1}_{A/k}$) on which the involution on $\Omega^{\bullet}_{A/k}$ induced by the involution on $V$ acts with eigenvalue $(-1)^p$ (see \cite{SV} for the second isomorphism in the dihedral case). With these identifications, and by \cite{BR22c} and Corollary \ref{hosym} above, the diagram \eqref{hcdso} is identified with 
\begin{equation} \la{hcdso2}
\begin{diagram}
\Omega^{\bullet}_{A/k}/d\Omega^{\bullet-1}_{A/k} & \rTo^{\Tr} & \bigwedge_k\left[\big(\Lambda^{\geq 1}(V[1])\big)[-1]\right]\\
\dTo & & \dTo\\
(\Omega^{\bullet}_{A/k})_+/(d\Omega^{\bullet-1}_{A/k})_+ & \rTo & \bigwedge_k\left[\big(\Lambda^{\rm odd}(V_+[1]) \otimes \Lambda(V_-[1])\big)[-1]\right]\\
\end{diagram}
\end{equation}

where the left vertical arrow is a projection to the direct summand, while the right vertical arrow is a quotient map: note that the space of generators $\big(\Lambda^{\rm odd}(V_+[1]) \otimes \Lambda(V_-[1])\big)[-1]$ of ${\rm HO}_\bullet(A)$ is a direct summand of the space of generators $\big(\Lambda^{\geq 1}(V[1])\big)[-1]$ of ${\rm HS}_\bullet(A)$. It follows from \eqref{hcdso2} and \cite[Theorem 4.2]{BFPRW2} that the map $i^{\ast}\,:\,\overline{\rm HD}_\bullet(\Sym(V)) \to {\rm HO}_\bullet(\Sym(V))$ can be expressed in terms of certain explicit differential operators on $\Omega^{\bullet}(\Sym(V))$.
 
\end{remark}

\end{example}
 
\begin{example}
 Recall that as graded algebras, $\H^\bullet(\gl_n;k)\,\cong\,\Lambda_k(\omega_1,\omega_3,\ldots,\omega_{2n-1})$ and $\H^{\bullet}(\mathfrak{sl}_n;k)\,\cong\,\Lambda_k(\omega_3,\omega_5,\ldots,\omega_{2n-1})$, where $\omega_{2i-1}$ is the class of the cocycle 
$$\tilde{\omega}_{2i-1}: (X_1,\ldots,X_{2i-1}) \mapsto \sum_{\tau \in S_{2i-1}} {\rm sgn}(\tau)\Tr(X_{\tau(1)}X_{\tau(2)}\ldots X_{\tau(2i-1)})\ .$$
Dually, $\H_\bullet(\gl_n(k);k) \,\cong\,\Lambda_k(\xi_1,\xi_3,\ldots,\xi_{2n-1})$ and $\H_\bullet(\mathfrak{sl}_n(k);k)\,\cong\,\Lambda_k(\xi_3,\xi_5,\ldots,\xi_{2n-1})$ (as graded cocommutative coalgebras), where the cogenerator $\xi_{2i-1}$ in degre $2i-1$ is dual to the cohomology class $\omega_{2i-1}$. Note that any involution $\sigma$ on $\gl_n(k)$ that is (the restriction of) an involution on the associative algebra $\M_n(k)$ of $n \times n$-matrices with entries in $k$ restricts to an involution on $\mathfrak{sl}_n(k)$.
\begin{cor} \la{hougl}
For any involution on $\gl_n(k)$ that is the restriction of an involution on $\M_n(k)$ as above,
\begin{align*}
   {\rm HO}_\bullet(U\gl_n(k)) &\cong\,\bigwedge_k\left[\big(\Lambda^{\rm odd}(\xi_1,\xi_5,\ldots,\xi_{4\lfloor \frac{n-1}{4}\rfloor+1}) \otimes \Lambda(\xi_3,\xi_7,\ldots,\xi_{4\lfloor \frac{n-3}{4}\rfloor+3})\big)[-1]\right]\,,\\
    {\rm HO}_\bullet(U\mathfrak{sl}_n(k)) &\cong\,\bigwedge_k\left[\big(\Lambda^{\rm odd}(\xi_5,\ldots,\xi_{4\lfloor \frac{n-1}{4}\rfloor+1}) \otimes \Lambda(\xi_3,\xi_7,\ldots,\xi_{4\lfloor \frac{n-3}{4}\rfloor+3})\big)[-1]\right]\ .
\end{align*}
\end{cor}
\begin{proof}
Abusing notation, let $\sigma$ denote the skew involution on $\H^{\bullet}(\gl_n(k);k)$ induced by the involution on $\gl_n(k)$. Let $s$ denote the permutation $j \mapsto 2i-j$ in $S_{2i-1}$. Note that ${\rm sgn}(s)=(-1)^{i-1}$. Then, 
\begin{align*}
\sigma\tilde{\omega}_{2i-1}(X_1,\ldots,X_{2i-1}) &= \tilde{\omega}(\sigma(X_1),\ldots,\sigma(X_{2i-1}))\\
 &=\sum_{\tau \in S_{2i-1}} {\rm sgn}(\tau)\Tr(\sigma(X_{\tau(1)})\sigma(X_{\tau(2)})\ldots \sigma(X_{\tau(2i-1)}))\\
 &= \sum_{\tau \in S_{2i-1}} {\rm sgn}(\tau)\Tr(\sigma(X_{\tau(2i-1)}X_{\tau(2i-2)}\ldots X_{\tau(1)}))\\
 &= \sum_{\tau \in S_{2i-1}} {\rm sgn}(\tau)\Tr(X_{\tau(2i-1)}X_{\tau(2i-2)}\ldots X_{\tau(1)})\\
 &= \sum_{\tau \in S_{2i-1}} {\rm sgn}(\tau)\Tr(X_{\tau \circ s(1)}X_{\tau \circ s(2)}\ldots X_{\tau \circ s(2i-1)})\\
 &= (-1)^{i-1}\tilde{\omega}(X_1,\ldots,X_{2i-1})\ .
\end{align*}
The fourth equality above is because any involution on $\M_n(k)$ is of the form $X \mapsto J^{-1}X^TJ$ for some $J \in \GL_n(k)$ satisfying $J^T=\pm J$ (see \cite[Thm. 1]{AA}), which implies that any involution on $\M_n(k)$ is trace preserving. Hence, $\sigma(\omega_{4i-3})=\omega_{4i-3}$ and $\sigma(\omega_{4i-1})=-\omega_{4i-1}$ for all $i \geq 1$. As in the proof of Corollary \ref{hosym}, it follows that the $+1$-eigenspace of the induced skew involution on $\H^{\bullet}(\gl_n(k);k)$ is $\Lambda^{\rm odd}(\omega_1,\omega_5,\ldots,\omega_{4\lfloor \frac{n-1}{4}\rfloor+1}) \otimes \Lambda(\omega_3,\omega_7,\ldots,\omega_{4\lfloor \frac{n-3}{4}\rfloor+3})$. Taking graded linear duals, it follows that 
$$\rH^{\rm ev}_{\bullet}(\gl_n(k);k) \,\cong\, \Lambda^{\rm odd}(\xi_1,\xi_5,\ldots,\xi_{4\lfloor \frac{n-1}{4}\rfloor+1}) \otimes \Lambda(\xi_3,\xi_7,\ldots,\xi_{4\lfloor \frac{n-3}{4}\rfloor+3})\ .$$ The desired formula for ${\rm HO}_\bullet(U\gl_n(k))$ then follows immediately from Theorem \ref{houainv}. The proof of the formula for ${\rm HO}_\bullet(U\mathfrak{sl}_n(k))$ is identical but for one minor modification: the (co)generators $\omega_1$ and $\xi_1$ are omitted.
\end{proof}
\end{example}

\section{Representation homology of groups and group algebras with involution} \la{Sgroups}

In this section, we extend the results of \cite[Sections 3 and 4]{BR22} to the involutive setting. The group algebra $k[\Gamma]$ of a simplicial group $\Gamma$ is equipped with the canonical involution $g \mapsto g^{-1}$ and the resulting dihedral and hyperoctahedral bar constructions naturally extend to homotopy simplicial groups. Cyclic and symmetric homology are replaced by dihedral and hyperoctahedral homology respectively, and the multiplicative group $\mathbb{G}_m = {\rm GL}_1$ is replaced by $\mtwo = \GL^{\!\sigma}_1$. The main results are Theorem~\ref{Tgroups} and Corollary~\ref{Cgroups} stated in the Introduction. Throughout this section, $k$ is an arbitrary commutative ring unless stated otherwise. We begin by recalling basic definitions and
some results from \cite{BR22}.

\subsection{Preliminaries}
\la{Sgr1}
Let $\fG$ denote the small category whose objects are the finitely generated free groups $\langle n \rangle := F\langle x_1, \ldots, x_n\rangle$, one for each $n \geq 0$, and whose morphisms are arbitrary group homomorphisms. Following \cite{Ba02}, we define a {\it homotopy simplicial group} to be a functor $\Gamma: \fG^{\rm op} \to \sSet$ that preserves finite products up to weak equivalence. Every (strict) simplicial group $\Gamma \in \sGr$ defines such a functor, $\langle n \rangle \mapsto \Gamma^n$, and by Badzioch's Rigidification Theorem (\cite[Theorem 6.4]{Ba02}), every homotopy simplicial group is weakly equivalent to a strict one. Hence, the homotopy simplicial groups provide natural algebraic models for pointed connected spaces.

We write $\sGrh$ for the category of homotopy simplicial groups, and 
denote by $B\Gamma$ for the classifying space of $\Gamma \in \sGrh$ (see \cite[Definition 3.4]{BR22}). Composing $\Gamma$ with the free $k$-module functor, we get a simplicial $\fG^{\rm op}$-module $k[\Gamma]: \langle n \rangle \mapsto k[\Gamma\langle n \rangle]$. On the other hand, any commutative Hopf algebra $H$ over $k$ defines a (covariant) $\fG$-module $\underline{H}: \langle n \rangle \mapsto H^{\otimes n}\,$. If $H = \mathcal{O}(G)$ is the coordinate ring of an affine group scheme $G$, then $\underline{H}(\langle n\rangle) = \mathcal{O}[{\rm Rep}_G(\langle n \rangle)] \,\cong\, \mathcal{O}(G^n)$. 
The {\it representation homology of $\Gamma$ with coefficients in $H$} is defined by (see \cite[Definition 3.6]{BR22})
\begin{equation} \la{HRgr} {\rm HR}_\bullet(\Gamma, H)\,:=\, {\rm Tor}^{\fG}_\bullet(k[\Gamma], \underline{H})\,, \end{equation}
and we write ${\rm HR}_\bullet(\Gamma, G)$ for ${\rm HR}_\bullet(\Gamma, \mathcal{O}(G))$. By \cite[Lemma 3.7]{BR22}, ${\rm HR}_\bullet(\Gamma, H)$ depends only on the homotopy type of $B\Gamma$.

We will use the following form of the Shapiro Lemma \cite[Proposition 2.6, Corollary 2.9]{BR22}. Let $F: \fG \to \Sets$ be a functor, $\fG_F := \fG\!\int\! F$ its category of elements (objects are pairs $(\langle n \rangle, a)$ with $a \in F\langle n \rangle$, morphisms are the $\varphi \in \Hom_{\fG}(\langle n\rangle, \langle m \rangle)$ with $F(\varphi)(a)=b$) and $p: \fG_F \to \fG$ the projection. Then, for any $\Gamma \in \sGrh$,
\begin{equation} \la{Shapiro} \hocolim_{\fG_F^{\rm op}}\, p^\ast\Gamma\, \simeq\, \int^{\langle n \rangle \in \fG}_{\L} \Gamma\langle n \rangle \times F\langle n\rangle\,,\qquad \H_\bullet\big(\hocolim_{\fG_F^{\rm op}} p^\ast\Gamma;\,k\big)\,\cong\,{\rm Tor}^{\fG}_\bullet(k[\Gamma], k[F])\,, \end{equation}
where $k[F]: \langle n \rangle \mapsto k[F\langle n \rangle]$ and the first homotopy colimit is taken in $\sSet$.

Finally, recall that for $\Gamma \in \sGr$, the involution $g \mapsto g^{-1}$ on $k[\Gamma]$ is already encoded in the category $\fG$: the automorphism $\iota: \langle 1 \rangle \to \langle 1 \rangle$, $x \mapsto x^{-1}$, is a morphism of $\fG$, and $\Gamma(\iota)$ is the inversion on $\Gamma\langle 1\rangle$. This is why no extra structure on $\Gamma$ is needed in the involutive theory, and why the hyperoctahedral category $\Delta H$, whose morphisms carry $C_2$-labels, maps to $\fG^{\rm op}$ (see Section~\ref{Sgr5}).

\subsection{Coefficients in $\mtwo$ and $\GLm$} \la{Sgr2}
Let $\mtwo := {\rm Spec}\,k[x]/(x^2-1)$, so that $\mathcal{O}(\mtwo) = k[\Z/2]$ as a commutative Hopf algebra, and let $\underline{\mathcal{O}(\mtwo)}$ denote the corresponding left $\fG$-module (see Section \ref{Sgr1}). Equivalently, $\underline{\mathcal{O}(\mtwo)} = k[\Ztwo]$ is the linearization of the set-valued functor
\begin{equation} \la{Z2} \Ztwo:\ \fG \to \Sets\,,\qquad \langle n \rangle \mapsto \langle n \rangle_{\rm ab} \otimes \Z/2 = (\Z/2)^n\,, \end{equation}
the mod $2$ abelianization. Thus, for $\Gamma \in \sGrh$,
\begin{equation} \la{HRmu2} {\rm HR}_\bullet(\Gamma, \mtwo)\,=\,{\rm Tor}^{\fG}_\bullet(k[\Gamma], k[\Ztwo])\,,\qquad {\rm HR}_0(\Gamma,\mtwo)\,\cong\,\mathcal{O}[\Hom(\pi_0\Gamma, \mtwo)]\,=\,k[\pi_0(\Gamma)_{\rm ab} \otimes \Z/2]\ . \end{equation}

The relation to the involutive representation homology of $k[\Gamma]$ is as follows. Let $\sigma$ be the involution on $\M_n(k)$ fixed in Section~\ref{Sec2}. Recall that every $k$-linear involution of $\M_n(k)$ has the form $\sigma_J(X) = J^{-1}X^{T}J$ with $J^{T} = \pm J$, and put
$$ \GLm\,:=\,\{g \in {\rm GL}_n\,:\,\sigma(g) = g^{-1}\}\ , $$
which is the orthogonal group ${\rm O}(J)$ if $J$ is symmetric and the symplectic group ${\rm Sp}(J)$ if $J$ is skew-symmetric. For $n=1$, $\GLm = \mtwo$. For a commutative algebra $B$, an involution preserving representation $k[\Gamma] \to \M_n(B)$ is the same as a representation $\rho: \Gamma \to {\rm GL}_n(B)$ with $\rho(g^{-1}) = \sigma(\rho(g))$, i.e. a group homomorphism $\Gamma \to \GLm(B)$. Hence $(k[\Gamma])^{\ob}_n = \mathcal{O}[{\rm Rep}_{\GLm}(\Gamma)]$, which is the involutive analogue of \cite[(1-3)]{BR22}. Deriving this yields the following analogue of \cite[Theorem 2.1]{BR22c}, whose proof we outline.
\begin{proposition}
\la{grcomp}
Let $k$ be a field of characteristic $0$ and $\Gamma \in \sGr$ a simplicial group. Then there is a natural isomorphism $\,\HRi(k[\Gamma], k^n) \cong {\rm HR}_\bullet(\Gamma, \GLm)$. In particular, $\HRi(k[\Gamma], k) \cong {\rm HR}_\bullet(\Gamma, \mtwo)$.
\end{proposition}
\begin{proof}
The proof follows that of \cite[Theorem 2.1]{BR22c}, with the category $\fM$ of finitely generated free monoids replaced by the category $\fM^{\sigma}$ of finitely generated free involutive monoids $\langle m \rangle^\sigma := J\big((C_2 \times \underline{m})_+\big)$ (the free monoid on the free $C_2$-set with $m$ orbits $\{x_i, \bar{x}_i\}$, with involution induced by the $C_2$-action), and $p: \fM^\sigma \to \fG$ the group completion functor, $\bar{x}_i \mapsto x_i^{-1}$. First, for any simplicial involutive monoid $N$, one has 
\begin{equation} \la{torhr}  {\rm Tor}^{\fM^\sigma}_\bullet(k[N], \underline{\mathcal{O}(\M_n)}) \cong \HRi(k[N], k^n)\,,\end{equation}
where $\underline{\mathcal{O}(\M_n)}$ is the left $\fM^{\sigma}$-module given by $\underline{\mathcal{O}(\M_n)}(\langle m \rangle^\sigma) := (k\langle m\rangle^\sigma)^{\ob}_n = \mathcal{O}(\M_n^{\,m})$. The isomorphism \eqref{torhr} is analogous to \cite[Proposition 2.1]{BR22c}, whose proof goes through in the involutive setting with the obvious modifications using a semi-free resolution $R \sonto N$ in the category if simplicial involutive algebras. 
Next, note that there is an isomorphism in the derived category of $\fG$-modules
\begin{equation} \la{lpsho} \L p_!\,\underline{\mathcal{O}(\M_n)} \cong \underline{\mathcal{O}(\GLm)}\ .\end{equation}
Indeed, since $k\langle m \rangle$ is a projective right $\fG$-module and by \cite[Lemma 2.1]{BR22c},
$$ \L p_!\,\underline{\mathcal{O}(\M_n)}(\langle m \rangle)\,\cong\, k[\langle m\rangle] \otimes^{\L}_{\fG} \L p_! \,\underline{\mathcal{O}(\M_n)}\,\cong\, p^\ast k[\langle m \rangle] \otimes^{\L}_{\fM^{\sigma}} \underline{\mathcal{O}(\M_n)}\ .$$
By \eqref{torhr},
$$\H_\bullet[p^\ast k[\langle m \rangle] \otimes^{\L}_{\fM^{\sigma}} \underline{\mathcal{O}(\M_n)}]\,\cong\, \HRi(k[\langle m\rangle], k^n)\ .$$
Verifying \eqref{lpsho} then amounts to the vanishing of ${\rm HR}^{\ast}_i(k[\langle m\rangle ], k^n)$ for all $i > 0$ and ${\rm HR}^{\ast}_0(k[\langle m \rangle], k^n) = \mathcal{O}\big((\GLm)^m\big)$. This is \cite[Theorem 1.7]{BR14} applied to $A = k[\langle m \rangle]$, which is formally smooth and admits a semi-free involutive resolution that is finitely finitely generated in each degree (also see Remark \ref{grrational} and Example \ref{grex} (4) below for a different proof of this vanishing for $n=1$). Finally, 
$$ {\rm HR}_\bullet(\Gamma, \GLm):={\rm Tor}^{\fG}_\bullet(k[\Gamma], \underline{\mathcal{O}(\GLm)})\cong{\rm Tor}^{\fG}_\bullet(k[\Gamma], \L p_!\,\underline{\mathcal{O}(\M_n)})\cong {\rm Tor}^{\fM^{\sigma}}_\bullet(k[\Gamma], \underline{\mathcal{O}(\M_n)})\cong \HRi(k[\Gamma],k^n)\ . $$
The second isomorphism above is by \eqref{lpsho}, the third isomorphism is by \cite[Lemma 2.1]{BR22c} and the last isomorphism is by \eqref{torhr}. This proves the desired proposition.
\end{proof}

%
\subsection{Dihedral bar construction of homotopy simplicial groups} \la{Sgr3}
Recall from \cite[(3-23), (3-24)]{BR22} the functor $\Psicyc: \Delta C \to \fG$, $[n] \mapsto \langle n+1 \rangle$, through which the cyclic bar construction of a group factors: $B^{\rm cyc}\Gamma = \Psicyc^\ast(\Gamma)$. We index the generators of $\langle n+1\rangle$ by $x_0, \ldots, x_n$ and write $D_{n+1} = \langle t_n, r_n \rangle$ for the dihedral group as in Section~\ref{bardihhoct}.
\begin{lemma} \la{grPsidih}
The functor $\Psicyc$ extends to a functor $\Psidih: \Delta D \to \fG$, $[n] \mapsto \langle n+1\rangle$, by
$$ \Psidih(r_n):\ \langle n+1 \rangle \to \langle n+1 \rangle\,,\qquad (x_0, x_1, \ldots, x_n) \mapsto (x_0^{-1}, x_n^{-1}, \ldots, x_1^{-1})\ . $$
For a discrete group $\Gamma$, $\Psidih^\ast(\underline{\Gamma}): \Delta D^{\rm op} \to \Sets$ is the dihedral bar construction of $\Gamma$, $\,r_n(g_0, \ldots, g_n) = (g_0^{-1}, g_n^{-1}, \ldots, g_1^{-1})$, so that $k[\Psidih^\ast\underline{\Gamma}] = B^{\rm dih}(k[\Gamma])$ is the dihedral bar construction of the involutive algebra $k[\Gamma]$ of Section~\ref{bardihhoct}.
\end{lemma}
\begin{proof}
The dihedral relations $r_n t_n = t_n^{-1} r_n$, $r_n^2 = 1$, $r_n d_i = d_{n-i} r_{n-1}$ and $r_n s_j = s_{n-j} r_{n+1}$ (in $\Delta D$) are verified on the free groups $\langle n+1 \rangle$ by direct substitution, exactly as the cyclic relations in \cite[Lemma 3.10]{BR22}. The second statement follows by evaluating $\underline{\Gamma}$ on the inversions in $\fG$.
\end{proof}
\begin{definition} \la{grHDdef}
For $\Gamma \in \sGrh$, we define $B^{\rm dih}\Gamma := \Psidih^\ast(\Gamma): \Delta D^{\rm op} \to \sSet$ and
$$ \HD_\bullet(k[\Gamma])\,:=\,{\rm Tor}^{\Delta D^{\rm op}}_\bullet\big(k, k[B^{\rm dih}\Gamma]\big)\,\cong\,{\rm Tor}^{\Delta D}_\bullet\big(k[B^{\rm dih}\Gamma], k\big)\ . $$
\end{definition}
As in \cite[Lemma 3.7]{BR22}, $\HD_\bullet(k[\Gamma])$ depends only on the homotopy type of $B\Gamma$. For a discrete group $\Gamma$, it is the dihedral homology of the involutive algebra $k[\Gamma]$ in the sense of Section~\ref{bardihhoct}, and $\HD_0(k[\Gamma]) = k[\Gamma]_\natural$ is the quotient of $k[\Gamma]$ by the linear span of the commutators and of $g - g^{-1}$, $g \in \Gamma$. The inclusion $\Delta C^{\rm op} \into \Delta D^{\rm op}$ induces the natural map ${\rm HC}_\bullet(k[\Gamma]) \to \HD_\bullet(k[\Gamma])$, which for $2 \in k^\times$ is the projection onto the $(+1)$-eigenspace of the involution induced by $r$.
\begin{remark}
For a discrete group $\Gamma$ and $k$ a field of characteristic $\neq 2$, Lodder's theorem \cite[Theorem 3.3.3]{Lo} identifies $\HD_\bullet(k[\Gamma])$ with the homology of the Borel construction $EO(2) \times_{O(2)} \mathcal{L}(B\Gamma)$, where $O(2)$ acts on the free loop space of $B\Gamma$ by rotation and reversal of loops. This is the dihedral analogue of Goodwillie's theorem used in \cite[(4-42)]{BR22}.
\end{remark}

\subsection{Dihedral character maps} \la{Sgr4}
The construction of \cite[Section 3.4]{BR22} extends verbatim to the functor $\Psidih: \Delta D \to \fG$. For a commutative Hopf algebra $H$, we put $B_{\rm dih}H := \Psidih^\ast(\underline{H}): \Delta D \to {\rm Mod}_k$. Then $\Psidih^\ast k[\Gamma] = k[B^{\rm dih}\Gamma]$, and $\Psidih$ induces a natural map
\begin{equation} \la{Psidihstar} \Psidih^\ast:\ {\rm Tor}^{\Delta D}_\bullet\big(k[B^{\rm dih}\Gamma], B_{\rm dih}H\big) \,\to\, {\rm Tor}^{\fG}_\bullet(k[\Gamma], \underline{H})\,=\,{\rm HR}_\bullet(\Gamma, H)\ . \end{equation}
The next result is an extension of Lemma~3.14 of \cite{BR22}.
\begin{lemma} \la{grinv}
For an affine group scheme $G$ over $k$,
$$ \Hom_{{\rm Mod}_k(\Delta D)}\big(k, B_{\rm dih}\mathcal{O}(G)\big)\,\cong\,\mathcal{O}(G)^{G,\iota}\,:=\,\{P \in \mathcal{O}(G)^G\,:\, P(g^{-1}) = P(g)\}\,, $$
the algebra of conjugation invariant functions on $G$ that are also invariant under inversion.
\end{lemma}
\begin{proof}
By \cite[Lemma 3.14]{BR22}, there is an isomorphism $\Hom_{{\rm Mod}_k(\Delta C)}(k, B_{\rm cyc}\mathcal{O}(G)) \cong \mathcal{O}(G)^G$, with an invariant $P$ corresponding to the family $(\Delta_G^{(m)}P)_{m \geq 0}$, where $\Delta_G^{(m)}P(g_0, \ldots, g_m) = P(g_0 \cdots g_m)$. The naturality with respect to the additional generators $r_m$ of $\Delta D$ reads  $\,\Delta^{(m)}_GP \circ \Psidih(r_m) = \Delta^{(m)}_GP\,$, which is
explicitly the following condition
\begin{equation} \la{invpg} P(g_0^{-1}g_m^{-1}\cdots g_1^{-1}) = P(g_0 \cdots g_m) \end{equation} 
for all $g_i \in G$. For $m=0$, this yields $P(g^{-1}) = P(g)$ for all $g \in G$. Conversely, the condition $P(g^{-1}) = P(g)$ for all $g \in G$ together with conjugation invariance implies the condition \eqref{invpg} for all $m$, since $g_0^{-1}g_m^{-1}\cdots g_1^{-1} = (g_1 \cdots g_m g_0)^{-1}$ is conjugate to $(g_0 \cdots g_m)^{-1}$.
\end{proof}
\begin{definition} \la{grchardef}
For $P \in \mathcal{O}(G)^{G,\iota}$, the {\it dihedral character map} associated with $P$ is the composite
$$ \chi^{\rm dih}_{G,P}(\Gamma)_\bullet:\ \HD_\bullet(k[\Gamma]) \xrightarrow{(\Delta_G P)_\ast} {\rm Tor}^{\Delta D}_\bullet\big(k[B^{\rm dih}\Gamma], B_{\rm dih}\mathcal{O}(G)\big) \xrightarrow{\ \Psidih^\ast\ } {\rm HR}_\bullet(\Gamma, G)\ . $$
\end{definition}
Explicitly, if $Q \sonto k[\Gamma]$ is a projective resolution in ${\rm Mod}_k(\fG^{\rm op})$, then $\chi^{\rm dih}_{G,P}(\Gamma)_\bullet$ is represented by 
$$\chi^{\rm dih}_{G,P}(\Gamma)_\bullet\,:\,\Psidih^\ast Q \otimes_{\Delta D} k \to Q \otimes_{\fG} \underline{\mathcal{O}(G)} $$ 
which, in turn, is induced by the chain map ({\it cf.} \cite[(3-32), (3-33)]{BR22}):
$$\bigoplus_{[m] \in \Delta D} Q\langle m+1\rangle \to \bigoplus_{\langle n \rangle \in \fG} Q\langle n \rangle \otimes \mathcal{O}(G)^{\otimes n}\,,\qquad \,v_{m+1} \mapsto v_{m+1} \otimes \Delta_G^{(m)}(P)\,.
$$
\begin{example} 
\la{grtraces}
We illustrate Definition~\ref{grchardef} in some special cases:
\begin{itemize}
\item[$(1)$] Let $G = {\rm GL}_n$, $P = \Tr_n$. Since $\Tr_n(g^{-1}) \neq \Tr_n(g)$, the derived character map $\Tr_n(\Gamma)_\bullet: {\rm HC}_\bullet(k[\Gamma]) \to {\rm HR}_\bullet(\Gamma, {\rm GL}_n)$ of \cite{BR22} does not factor through $\HD_\bullet(k[\Gamma])$ ({\it cf.} \cite[Remark 3.2]{BR22c}).

\item[$(2)$] Let $G = \GLm$, $P = \Tr_n|_{G}$. Now $\Tr_n(g^{-1}) = \Tr_n(\sigma(g)) = \Tr_n(g)$, since $\sigma$ is trace preserving, so that $\Tr_n$ defines a dihedral character map $\Tr^\sigma_n(\Gamma)_\bullet: \HD_\bullet(k[\Gamma]) \to {\rm HR}_\bullet(\Gamma, \GLm)^{\GLm}$. The isomorphism of Proposition~\ref{grcomp} identifies this with the involutive derived character map \eqref{DTr} of the involutive algebra $k[\Gamma]$ (cf. \cite[Remark 3.16]{BR22} in the cyclic case).

\item[$(3)$] Let $n = 1$, $G = \mtwo$, $\mathcal{O}(\mtwo) = k \cdot 1 \oplus k \cdot x$. Here $\iota = \id$ on $\mtwo$, so that $\mathcal{O}(\mtwo)^{\mtwo,\iota} = \mathcal{O}(\mtwo)$; the function $P = x$ defines the {\it one-dimensional involutive derived character map} 
\begin{equation} \la{grTr} 
\Tr^\sigma(\Gamma)_\bullet:\ \HD_\bullet(k[\Gamma]) \to {\rm HR}_\bullet(\Gamma, \mtwo)\ .
\end{equation}
%
\end{itemize}
\end{example}

\subsection{Representation homology with coefficients in diagonalizable group schemes} \la{Sgr5}
For an abelian group $\cA$, let $D_k(\cA) := {\rm Spec}\,k[\cA]$ denote the diagonalizable group scheme with character group $\cA$ (so that $\, D_k(\Z) = \mathbb{G}_m \,$ and $\, D_k(\Z/n) = \mu_n\,$ for all $n \ge 2$). Define the $\cA$-abelianization functor by
\begin{equation} \la{FA} F_{\cA}:\ \fG \to \Sets\,,\qquad \langle n\rangle \mapsto \langle n\rangle_{\rm ab}\otimes_{\Z} \cA \,=\, \cA^{n}\,. 
\end{equation}
Since $k[F_{\cA}] \cong \underline{\mathcal{O}(D_k(\cA))}$ as $\fG$-modules, there is a natural isomorphism
$$
{\rm HR}_\bullet(\Gamma, D_k(\cA)) \,\cong\, {\rm Tor}^{\fG}_\bullet(k[\Gamma], k[F_{\cA}])\,.
$$
Let $\fG_{\cA} := \fG_{F_{\cA}}$ be the category of elements of the functor $F_{\cA}$, with $p: \fG_{\cA} \to \fG$ being the projection. Writing $ p^*: \sSet^{\fG} \to \sSet^{\fG_{\cA}}$ for the restriction functor, we set for any 
$\Gamma \in \sGrh$: 
\begin{equation} \la{XA} X^{\cA}_\Gamma\,:=\,\big|\hocolim_{\fG_{\cA}^{\rm op}}\, (p^\ast\Gamma)\big|\ . \end{equation}
In the special case $\cA = \Z$, this is precisely the space $X_\Gamma$ defined in \cite[(4-1)]{BR22}. By \eqref{Shapiro}, we have
\begin{equation} \la{HXA} \H_\bullet(X^{\cA}_\Gamma;\, k)\,\cong\,{\rm HR}_\bullet(\Gamma, D_k(\cA))\,. \end{equation}

Now, on the topological side, for any pointed connected CW complex (represented by a reduced simplicial set) $X$ and any abelian group $\cA$, we define 
$$
\SPi(X;\cA) := |\widetilde{\Z}[X] \otimes \cA| \,
$$ 
where $\widetilde{\Z}[X] := \Z[X]/\Z[\ast]$ is the reduced free simplicial abelian group on $X$. 
By the Dold-Thom Theorem, the space $\SPi(X;\cA)$ is weakly equivalent to the McCord space $B(\cA, X)$ of finitely supported $\cA$-valued functions on $X$ vanishing at the basepoint (see \eqref{spxa}); in particular, we have the homotopy decomposition
$$
\SPi(X;\cA) \simeq \prod_{n \geq 1} K(\tilde{\H}_n(X;\cA), n)
$$
For $\cA = \Z$, the space $\,\SPi(X;\Z) $ is equivalent to $ \SPi(X)$, the classical infinite symmetric product. 

Next, recall the classical Quillen equivalence
$\,\mathbb{G}: \sSet_0 \leftrightarrows \sGr: \overline{W} \,$ between
the model categories of reduced simplicial sets and simplicial groups, where the left adjoint $ \mathbb{G} $ is usually called the Kan Loop Group functor and the right adjoint $\overline{W}$ is the classifying complex functor. Following \cite[(4-12), (4-13)]{BR22}, we denote by $\mathbb{A}(X) := (\mathbb{G}X)_{\rm ab}$ the abelianization of ${\mathbb G}(X)$, and observe 
that the simplicial set $ \overline{W} \mathbb{A}(X) $ is naturally a simplicial abelian group. Then, the canonical map of simplicial sets 
$\, X \to \overline{W}\mathbb{G}(X) \onto  
\overline{W} \mathbb{A}(X) \,$ extends to a homomorphism of simplicial abelian groups
\begin{equation} 
\la{hurw}
\widetilde{\Z}[X] \,\xrightarrow{\sim} \,\overline{W} \mathbb{A}(X)
\end{equation}
which is a (weak) equivalence by the Dold-Kan correspondence.
Consequently, $|\mathbb{A}(X)| \simeq \Omega\,\SPi(X)$. 
\begin{theorem} \la{Tgroups}
Let $k$ be a commutative ring and $\cA$ an abelian group. For any $\Gamma \in \sGrh$, there is a weak equivalence of spaces, natural in $\Gamma$,
\begin{equation} \la{XAeq} X^{\cA}_\Gamma\,\simeq\,\Omega\,\SPi(B\Gamma;\cA)\,, \end{equation}
and consequently a natural isomorphism
\begin{equation} \la{HRDA} {\rm HR}_\bullet(\Gamma, D_k(\cA))\,\cong\,\H_\bullet\big(\Omega\,\SPi(B\Gamma;\cA);\,k\big)\,\cong\,\H_\bullet\Big(\prod_{n \geq 1} K\big(\tilde{\H}_n(B\Gamma;\cA), n-1\big);\,k\Big)\ . \end{equation}
In particular, ${\rm HR}_0(\Gamma, D_k(\cA)) = k[\pi_0(\Gamma)_{\rm ab} \otimes \cA]$, and if $k \supseteq \Q$ and $\cA$ is a torsion group, then ${\rm HR}_\bullet(\Gamma, D(\cA)) \cong k[\H_1(B\Gamma; \cA)]$ is concentrated in degree $0$.
\end{theorem}
\begin{proof}
The isomorphism \eqref{HRDA} follows from \eqref{XAeq} and \eqref{HXA}, using the fact that a simplicial abelian group is a product of Eilenberg-MacLane spaces and that $\pi_i\,\Omega\,\SPi(X;\cA) = \tilde{\H}_{i+1}(X;\cA)$. To prove \eqref{XAeq}, we follow the proof of \cite[Theorem 4.2]{BR22}.

First, we assume that  $\cA = R$ is the additive group of a commutative ring $R$. Since the functor $\Gamma \mapsto \Gamma_{\rm ab} \otimes_\Z R$ from groups to $R$-modules is left adjoint to the forgetful functor, its degreewise extension $(\,\mbox{--}\,)_{{\rm ab},R}: \sGr \to \s{\rm Mod}_R$ is a left Quillen functor (as the right adjoint does not change the underlying simplicial sets). By \eqref{Shapiro},
\begin{equation} \la{hcl1} \hocolim_{\fG_R^{\rm op}} (p^\ast\Gamma) \simeq \int_{\L}^{\langle n \rangle \in \fG}\Gamma\langle n\rangle \times R^n\ .\end{equation} 
Next, the bifunctor $\Gamma \times F_R: \fG^{\rm op}\times \fG \to \sSet$ factors as 
$$ 
\fG^{\rm op}\times\fG \xrightarrow{\Gamma \otimes \mathbb{F}} \sGr \xrightarrow{(\mbox{--})_{{\rm ab},R}} \s{\rm Mod}_R \to \sSet
$$
where $\Gamma \otimes \mathbb{F}$ is the objectwise cofibrant bifunctor of \cite[(4-16)]{BR22}. Since derived left Quillen functors commute with homotopy coends (see \cite[(2-20)]{BR22}),
\begin{equation} \la{hcl2} \L[\L K(\Gamma)]_{{\rm ab},R} \simeq \int_L^{\langle n \rangle \in \fG} \L(\Gamma\langle n\rangle \otimes \mathbb{F}\langle n \rangle)_{{\rm ab},R} \simeq \int_L^{\langle n \rangle \in \fG}(\Gamma\langle n\rangle \otimes \mathbb{F}\langle n \rangle)_{{\rm ab},R} \simeq \int_{\L}^{\langle n \rangle \in \fG}\Gamma\langle n\rangle \times R^n  \end{equation}
where $\L K(\Gamma)$ is the simplicial group associated to $\Gamma$ by \cite[Theorem 3.2]{BR22}. The first weak equivalence above uses \cite[(3-11)]{BR22} and the second is because $\Gamma \otimes \mathbb{F}$ is objectwise cofibrant. On the other hand,
$$\L[\L K(\Gamma)]_{{\rm ab},R} \simeq [\mathbb{G}\overline{W} \L K(\Gamma)]_{{\rm ab},R} \simeq \mathbb{A}(\overline{W} \L K(\Gamma))\otimes_{\Z} R\,, $$
since the Kan loop group $\mathbb{G}\overline{W} \L K(\Gamma)$ is cofibrant and weakly-equivalent to $\L K(\Gamma)$. Therefore, 
\begin{equation} \la{hcl3} \hocolim_{\fG_R^{\rm op}} p^\ast\Gamma \simeq \mathbb{A}(X) \otimes R \end{equation}
where $X := \overline{W}\L K(\Gamma) \simeq B\Gamma$.  Note that the weak-equivalence \eqref{hurw} between levelwise free simplicial abelian groups remains a weak-equivalence after $\otimes_\Z R$. Further, by the very definition of $\overline{W}$,  $\overline{W}\mathbb{A}(X)\otimes R = \overline{W}(\mathbb{A}(X)\otimes R)$. Since, in addition,  $|V| \simeq \Omega|\overline{W}V|$ for any simplicial group $V$, we conclude
$$ |\mathbb{A}(X)\otimes R|\,\simeq\,\Omega\,|\bar{W}(\mathbb{A}(X)\otimes R)|\,\simeq\,\Omega\,|\widetilde{\Z}[X]\otimes R|\,=\,\Omega\,\SPi(X;R)\,, $$
naturally in $X$. This proves \eqref{XAeq} for $\cA = R$.

Now, if $\cA$ is finitely generated, then $\cA \cong \Z^r \oplus \bigoplus_i \Z/m_i$ is the additive group of the commutative ring $\Z^r \times \prod_i \Z/m_i$, hence the above argument applies, proving \eqref{XAeq} for finitely generated abelian groups.

Finally, for an arbitrary abelian group $\cA$, we write $\cA = \varinjlim \cA_\alpha$ as the filtered colimit of its finitely generated subgroups. Both sides of \eqref{XAeq} commute with filtered colimits in $\cA$: $F_{\cA} = \varinjlim F_{{\cA}_\alpha}$, homotopy colimits commute with filtered colimits, $\mathbb{A}(X)\otimes \cA = \varinjlim \mathbb{A}(X)\otimes \cA_\alpha$, and geometric realization commutes with filtered colimits. Naturality in $\Gamma$ is clear at every step. This proves \eqref{XAeq} in general. \eqref{HRDA} follows immediately from \eqref{XAeq}.

For the last assertion, note that $K(V, j)$ is rationally trivial for $j \geq 1$ when $V$ is a torsion group, while $\pi_0\,\Omega\,\SPi(B\Gamma;\cA) = \H_1(B\Gamma;\cA) = \pi_0(\Gamma)_{\rm ab} \otimes \cA$.
\end{proof}
\begin{cor} \la{grchar0}
For any $\Gamma \in \sGrh$ and any commutative ring $k$, 
$${\rm HR}_\bullet(\Gamma, \mtwo) \cong \H_\bullet\big(\Omega\,\SPi(B\Gamma;\Z/2);k\big)\ .$$ 
If $k \supseteq \Q$, then 
$${\rm HR}_\bullet(\Gamma,\mtwo) \cong k[\H_1(B\Gamma;\Z/2)] = k[\pi_0(\Gamma)_{\rm ab}\otimes \Z/2]\,,$$ 
concentrated in degree $0$. If $k = \mathbb{F}_2$, then ${\rm HR}_\bullet(\Gamma,\mtwo)$ is the mod $2$ homology of the generalized Eilenberg-MacLane space $\prod_{i \geq 1}K(\tilde{\H}_i(B\Gamma;\mathbb{F}_2), i-1)$.
\end{cor}
%
%

%
\begin{remark} \la{grm2}
Theorem~\ref{Tgroups} is a statement about representation homology alone, valid for every abelian group $\cA$. What is special about $\cA = \Z/2$ is the existence of the dihedral/hyperoctahedral factorization of Section~\ref{Sgr6}: for $m \geq 3$, $\mu_m = \ker(g \mapsto g^m)$ is not the fixed point scheme of any (anti-)automorphism of $\mathbb{G}_m$, there is no involution on $k[\Gamma]$ with $(k[\Gamma])^{\ob}_1 = k[\Gamma_{\rm ab}\otimes\Z/m]$, and the lifting problem of Proposition~\ref{grlift} below has no solution modulo $m$. For general $m$, one only has the composite of the map of \cite[Proposition 4.8]{BR22} with the map induced on representation homology by the inclusion $\mu_m \to \mathbb{G}_m$ of coefficients, ${\rm HS}_\bullet(k[\Gamma]) \to {\rm HR}_\bullet(\Gamma,\mathbb{G}_m) \to {\rm HR}_\bullet(\Gamma, \mu_m)$, realized by 
$$ \Omega\,\Omega^\infty\Sigma^\infty(B\Gamma) \to \Omega\,\SPi(B\Gamma) \to \Omega\,\SPi(B\Gamma;\Z/m)
$$
\end{remark}

\subsection{Hyperoctahedral bar construction of homotopy simplicial groups} \la{Sgr6}
Recall the tensor notation for morphisms of $\Delta H$ introduced in Section~\ref{bardihhoct}: a morphism $f: [n] \to [m]$ of $\Delta H$ is a tensor product $X_0 \otimes \cdots \otimes X_m$ of $m+1$ noncommutative monomials in the letters $x_i^{\alpha_i}$, $\alpha_i \in C_2 = \{1, s\}$, each of $x_0, \ldots, x_n$ occurring exactly once, and composition is substitution of monomials, with 
$$(x_{i_1}^{\alpha_{i_1}}\cdots x_{i_r}^{\alpha_{i_r}})^s = x_{i_r}^{s\alpha_{i_r}}\cdots x_{i_1}^{s\alpha_{i_1}}\ .$$ 
Also recall from \cite[(4-23), (4-24)]{BR22} the functor 
$$\Psisym: \Delta S \to \fG^{\rm op}\,, \qquad [n] \mapsto \langle n+1 \rangle\, ,$$ sending the morphism $f = X_0 \otimes \cdots \otimes X_m$ (with all $\alpha_i = 1$) to the group homomorphism $\langle m+1\rangle \to \langle n+1\rangle$ mapping the generator $x_j$ of $\langle m+1 \rangle$ to the element $X_j(x_0, \ldots, x_n) \in \langle n+1 \rangle$.
\begin{lemma} \la{grPsihyp}
The functor $\Psisym$ extends to a functor 
$$\Psihyp: \Delta H \to \fG^{\rm op}\,, \qquad [n] \mapsto \langle n+1\rangle\, $$ sending $f = X_0 \otimes \cdots \otimes X_m$ to the homomorphism $\langle m+1 \rangle \to \langle n+1 \rangle$ given by 
$$ \Psihyp(f):\ (x_j \in \langle m+1 \rangle) \mapsto (x_{i_1}^{\epsilon_{i_1}} \cdots x_{i_r}^{\epsilon_{i_r}} \in \langle n+1\rangle)\quad\text{if}\quad X_j = x_{i_1}^{\alpha_{i_1}}\cdots x_{i_r}^{\alpha_{i_r}}\,,$$
where $\epsilon_i= \pm 1$ depending on whether $\alpha_i=1$ ot $\alpha_i=s$. 
There is a commutative diagram of functors 
$$
\begin{diagram}[small]
\Delta C^{\rm op} & \rTo^{i} & \Delta S & \rTo^{\Psi_{\rm sym}} & \fG^{\rm op}\\
 \dInto & & \dInto & &  \dTo^{\rm Id}\\
\Delta D^{\rm op}  & \rTo^{i} & \Delta H & \rTo^{\Psi_{\rm hOct}} & \fG^{\rm op}\\
\end{diagram}
$$
where the first two vertical arrows are the obvious inclusions, the functor $i:\Delta D^{\rm op} \to \Delta H$ is the inclusion defined in \cite[3.3]{G21}, and the functor $i:\Delta C^{\rm op} \to \Delta S$ is the inclusion defined in \cite[6.1.11]{L}. For a discrete group $\Gamma$, $\Psihyp^\ast(\underline{\Gamma})$ is the hyperoctahedral bar construction of the group $\Gamma$ with $\bar{g} = g^{-1}$ \cite[Definition 7.12]{G21}, so that $k[\Psihyp^\ast\underline{\Gamma}] = B_{\rm hOct}(k[\Gamma])$.
\end{lemma}
\begin{proof}
Composition in $\Delta H$ is substitution of monomials, with the reversal rule for $(\,\mbox{--}\,)^s$ recalled above: under $x_i^s = x_i^{-1}$ this rule becomes the identity $(x_{i_1}^{\epsilon_{i_1}}\cdots x_{i_r}^{\epsilon_{i_r}})^{-1} = x_{i_r}^{-\epsilon_{i_r}}\cdots x_{i_1}^{-\epsilon_{i_1}}$ in the free group. Hence substitution of monomials corresponds to composition of the group homomorphisms $\Psihyp(f)$. The commutativity of the square and the last statement are immediate from the definitions (cf. \cite[Lemma 4.3]{BR22} and \cite[Definition 7.12]{G21}).
\end{proof}
\begin{definition} \la{grHOdef}
For $\Gamma \in \sGrh$, put $B_{\rm hOct}\Gamma := \Psihyp^\ast(\Gamma): \Delta H \to \sSet$ and
$$ \HO_\bullet(k[\Gamma])\,:=\,{\rm Tor}^{\Delta H}_\bullet\big(k, k[B_{\rm hOct}\Gamma]\big)\ . $$
\end{definition}
As in \cite[Remark 4.5]{BR22}, $\HO_\bullet(k[\Gamma])$ depends only on the homotopy type of $B\Gamma$. For a discrete (or simplicial) group $\Gamma$ it is the hyperoctahedral homology of the involutive algebra $k[\Gamma]$ (Definition~\ref{defho}), and it is a graded commutative algebra by \cite[Corollary 4.5]{G22}. Since for any $\Gamma \in \sGrh$, $B\Gamma \simeq B\L K\Gamma$, where $\L K\Gamma$ is the simplicial group associated to $\Gamma$ by \cite[Theorem 3.2]{BR22},  $\HO_\bullet(k[\Gamma])$ is a graded commutative algebra for all $\Gamma \in\sGrh$. As in \cite[Remark 4.7]{BR22}, $\Psihyp$ extends to $\Delta H_+ := \Delta H \sqcup \{[-1]\}$ by $[-1] \mapsto \langle 0 \rangle$; under Graves' identification of $\Delta H_+$ with the PROP of involutive monoids, $\Psihyp^{\rm op}$ is the morphism of PROPs that forgets the algebra structure of a commutative Hopf algebra and retains the coalgebra with its antipode.

The inclusion $\Delta D^{\rm op} \cong \Delta D \into \Delta H$ induces, as in Section~\ref{S3.1}, the natural map 
$$i^\ast: \HD_\bullet(k[\Gamma]) \to \HO_\bullet(k[\Gamma])\,,$$ since $k[B^{\rm dih}\Gamma] = i^\ast k[B_{\rm hOct}\Gamma]$ by Lemma~\ref{grPsihyp}.
\begin{proposition} \la{grlift}
Let $\fG_{\Z/2} := \fG_{F_{\Z/2}}$ be the category of elements of \eqref{Z2}. The functor $\Psihyp$ lifts to a functor
\begin{equation} \la{tPsihyp} \tPsihyp:\ \Delta H \to \fG_{\Z/2}^{\rm op}\,,\qquad [n] \mapsto (\langle n+1 \rangle;\, \bar{1}, \ldots, \bar{1})\,, \end{equation}
and hence, for every $\Gamma \in \sGrh$, induces a natural map $\tPsihyp^\ast: \HO_\bullet(k[\Gamma]) \to {\rm HR}_\bullet(\Gamma, \mtwo)$ such that $\Tr^\sigma(\Gamma)_\bullet = \tPsihyp^\ast \circ i^\ast$, where $\Tr^\sigma(\Gamma)_\bullet$ is the one-dimensional involutive derived character map \eqref{grTr}.
\end{proposition}
\begin{proof}
For $f \in \Hom_{\Delta H}([n], [m])$ in tensor notation, the abelianization $\Psihyp(f)_{\rm ab}: \Z^{m+1} \to \Z^{n+1}$ is given by (left multiplication by) the integral $(n+1) \times (m+1)$-matrix whose $j$-th column has entries $\epsilon_i = 1$ in the rows $i$ with $x_i$ occurring in $X_j$, $\epsilon_i=-1$ in the rows with $x_i^s$ occurring in $X_j$ and $0$ elsewhere. Since for each $i$, exactly one of the pair $\{x_i,x_i^s\}$ occurs exactly once among the $X_0,\ldots,X_m$, every row contains exactly one nonzero entry, equal to $\pm 1$. Hence $\Psihyp(f)_{\rm ab}$ maps $(1, \ldots, 1)$ to $(\pm 1, \ldots, \pm 1)$, which is $(\bar{1}, \ldots, \bar{1})$ modulo $2$. This shows that \eqref{tPsihyp} is a well-defined functor lifting $\Psihyp$ through $p^{\rm op}: \fG_{\Z/2}^{\rm op} \to \fG^{\rm op}$, and that 
$$ k[B_{\rm hOct}\Gamma] = \Psihyp^\ast k[\Gamma] = \tPsihyp^\ast(k[p^\ast\Gamma])\ .$$ 
By \cite[Lemma 3.13]{BR22}, $\tPsihyp$ induces a map 
$$\HO_\bullet(k[\Gamma]) \to {\rm Tor}^{\fG_{\Z/2}^{\rm op}}_\bullet(k, k[p^\ast\Gamma])\,\cong \, {\rm HR}_\bullet(\Gamma,\mtwo)\,,$$ where the last isomorphism is  by \eqref{Shapiro}. The factorization $\Tr^\sigma(\Gamma)_\bullet = \tPsihyp^\ast\circ i^\ast$ is verified exactly as in \cite[Proposition 4.8]{BR22}: with a projective resolution $Q \sonto k[\Gamma]$, the Shapiro isomorphism is represented by \cite[(4-37)]{BR22} with $\Z^n$ replaced by $(\Z/2)^n$, and the composite $\tPsihyp^\ast\circ i^\ast$ by \cite[(4-38)]{BR22} with $(\langle m+1\rangle; 1, \ldots, 1)$ replaced by $(\langle m+1 \rangle; \bar{1}, \ldots, \bar{1})$: the resulting map $v_{m+1} \mapsto v_{m+1}\otimes[(\bar{1},\ldots,\bar{1})]$ is \eqref{grTr}.
\end{proof}
%
%
\begin{cor} 
\la{grspaces}
The maps $\HD_\bullet(k[\Gamma]) \xrightarrow{i^\ast} \HO_\bullet(k[\Gamma]) \xrightarrow{\tPsihyp^\ast} {\rm HR}_\bullet(\Gamma,\mtwo)$ are induced on homology by the natural maps of spaces
\begin{equation} \la{grmaps} \big|\hocolim_{\Delta D^{\rm op}}B^{\rm dih}\Gamma\big| \xrightarrow{\ \ i\ \ } \big|\hocolim_{\Delta H}B_{\rm hOct}\Gamma\big| \xrightarrow{\ \tPsihyp\ } \big|\hocolim_{\fG_{\Z/2}^{\rm op}}\, p^\ast\Gamma\big|\,\simeq\,\Omega\,\SPi(B\Gamma;\Z/2)\,, \end{equation}
the last equivalence being Theorem~\ref{Tgroups}. For a discrete group $\Gamma$, the middle space is the bar construction $B(C_\infty E, C_1, \Gamma)$ of \cite[Theorem 7.20]{G21}, where $E := EC_{2+}\wedge_{C_2}(\,\mbox{--}\,)$, and the left space is $EO(2)\times_{O(2)}\mathcal{L}(B\Gamma)$ \cite{Lo}.
\end{cor}
\begin{proof}
Same as \cite[Corollary 4.10]{BR22}: the three homotopy colimits are the Bousfield-Kan models whose $k$-chains compute the three Tor's, and the two maps are induced by the functors $i$ and $\tPsihyp$. The identification of the middle space is \cite[Lemma 7.13, Remark 7.14 and Theorem 7.20]{G21} and the last identification is by Theorem \ref{Tgroups}.
\end{proof}
%
\begin{cor} \la{Cgroups}
Let $k$ be a field of characteristic $0$ and $\Gamma \in \sGrh$. Then $\tPsihyp^\ast: \HO_\bullet(k[\Gamma]) \to {\rm HR}_\bullet(\Gamma,\mtwo)$ is an isomorphism of graded commutative $k$-algebras. Consequently,
$$ \HO_\bullet(k[\Gamma])\,\cong\,k[\H_1(B\Gamma;\Z/2)]\,, $$
concentrated in degree $0$, with no assumption on $B\Gamma$.
\end{cor}
\begin{proof}
Assume first that $\Gamma$ is a simplicial group. By Theorem~\ref{T1a} applied to the simplicial involutive algebra $k[\Gamma]$, $\HOR(k[\Gamma])_\bullet: \HO_\bullet(k[\Gamma]) \to \HRi(k[\Gamma], k)$ is an isomorphism of graded commutative algebras, and by Proposition~\ref{grcomp}, $\HRi(k[\Gamma], k) \cong {\rm HR}_\bullet(\Gamma,\mtwo)$. One checks, as in the proof of \cite[Corollary 3.1]{BR22c}, that the composite of these two isomorphisms is $\tPsihyp^\ast$: both maps are induced by the augmentation $k\otimes^{\L}_{\Delta H} \to k \otimes_{\Delta H}$ followed by the Shapiro identification. Note that $\sGrh$ reduces to $\sGr$ by Badzioch's theorem \cite[Theorem 3.2]{BR22} and homotopy invariance. The last assertion is Corollary~\ref{grchar0}.
\end{proof}
\begin{remark} \la{grrational}
For a discrete group $\Gamma$, the isomorphism $\tPsihyp^\ast$ over $\Q$ can also be seen directly on the spaces of Corollary~\ref{grspaces}, without using Proposition~\ref{grcomp}. The map $\tPsihyp$ in \eqref{grmaps} is induced by the natural transformation of functors on based $C_2$-sets
$$ C_\infty\big(EC_{2+}\wedge_{C_2}X\big) \longrightarrow \SPi\big(X/C_2\big) = \N\big[(X\smallsetminus\ast)/C_2\big]\,, $$
composed of the map of monads $C_\infty \to \SPi$ (from $E_\infty$ to commutative monoids) and the collapse $EC_2 \to {\rm pt}$, applied inside the bar construction; the target is $B\big(\SPi\circ(\,\mbox{--}\,/C_2), C_1, \Gamma\big) = |B(J, J, \Gamma)_{{\rm ab},\sigma}|$, the derived involutive abelianization of the monoid $\Gamma$ (the involutive counterpart of \cite[Proposition 4.11]{BR22}). By the Cartan-Leray theorem, $\H_\bullet\big(C_\infty(EC_{2+}\wedge_{C_2}X);\Q\big) \cong \Sym\big(\Q[(X\smallsetminus\ast)/C_2]\big)$ for a discrete based $C_2$-set $X$ (cf. the proof of Proposition~\ref{pC}), so the map is a rational homology isomorphism levelwise in the bar construction, and hence on realizations.
\end{remark}
\begin{example} \la{grex}
We illustrate Corollary~\ref{Cgroups} in some interesting spacial cases:
\begin{itemize}
\item[$(1)$]
 $\Gamma = \Z$: $\HO_\bullet(k[x^{\pm1}])$, with $\bar{x} = x^{-1}$, is $k[x]/(x^2-1) \cong k^2$ concentrated in degree $0$. Directly: $\HO_0 = k[x^{\pm1}]/(x - x^{-1})$ by \cite[Theorem 6.8]{G22}, and ${\rm HR}_\bullet(\Z,\mtwo) = \mathcal{O}(\mtwo)$ since $k[\Z]$ is a representable $\fG^{\rm op}$-module. 

\item[$(2)$] $\Gamma = \Z/2$: the involution on $k[\Z/2]$ is trivial, and $\HO_\bullet(k[\Z/2]) = k[\Z/2]$ in degree $0$, matching $k[\H_1(B\Z/2;\Z/2)] = k[\Z/2]$.

\item[$(3)$] $B\Gamma$ simply connected of finite type: then $k[\Gamma] \simeq U\mfa$ for a Quillen Lie model $\mfa$ of $B\Gamma$, with the inversion corresponding to the involution $\sigma = -\id$ on $\mfa$, and Theorem~\ref{houainv} gives $\HO_\bullet(U\mfa) = k$, in agreement with $k[\H_1(B\Gamma;\Z/2)] = k$. Note that in comparison with \cite[Theorem 1.3]{BR22c}, ${\rm HS}_\bullet(U\mfa) = \Lambda_k[\rH_{\bullet+1}(\mfa;k)]$ is large, while its involutive shadow is trivial, the sign $(-1)^{p-1}\Lambda^p(-\id) = -1$ killing all of the reduced Chevalley-Eilenberg homology.

\item[$(4)$] $\Gamma = F_m$ a free group: $\HO_\bullet(k[F_m]) = k[(\Z/2)^m] = \mathcal{O}(\mtwo^m)$ in degree $0$. Combined with Theorem~\ref{T1a}, this is the vanishing ${\rm HR}^{\ast}_i(k[F_m], k) = 0$, $i > 0$, used in the proof of Proposition~\ref{grcomp}. Hence, using Remark~\ref{grrational}, the case $n=1$ of Proposition~\ref{grcomp} can be obtained without appealing to \cite[Theorem 1.7]{BR14}.
\end{itemize}
\end{example}
\begin{remark}
For $n > 1$ and $G = \GLm$ (an orthogonal or symplectic group), Example~\ref{grtraces}$(2)$ gives dihedral character maps $\Tr^\sigma_n(\Gamma)_\bullet: \HD_\bullet(k[\Gamma]) \to {\rm HR}_\bullet(\Gamma,\GLm)^{\GLm}$, which assemble into a map of graded commutative algebras $\Lambda\Tr^\sigma_n: \Lambda_k[\overline{\HD}_\bullet(k[\Gamma])] \to {\rm HR}_\bullet(\Gamma,\GLm)^{\GLm}$. In can be shown, in analogy with \cite[Theorem 1.4]{BR22} (also see \cite[Sec. 8.5]{BFPRW}), that $\Lambda\Tr^\sigma_\infty$ becomes an isomorphism in the stable range $n \to \infty$ for $B\Gamma$ simply connected of finite rational type --- a derived (Koszul dual) form of the Loday-Procesi theorem expressing the homology of $\mathfrak{o}_\infty(A)$ and $\mathfrak{sp}_\infty(A)$ in terms of dihedral homology. We do not pursue this here.
\end{remark}


\begin{thebibliography}{34}
%
%
\bibitem{AA} 
A. Albert.
\newblock Involutorial simple algebras and real Riemann matrices.
\newblock Proc. Natl. Acad. Sci. U.S.A. 20(12):676--681, 1934.
%
\bibitem{AMP}
G. Angelini-Knoll, M. Merling and M. P\'{e}roux.
\newblock Topological $\Delta G$-homology of rings with twisted $G$-action. \newblock {\em Adv. Math.}, Paper no. 110907, 75pp, 2026.
%
\bibitem{ACGMP}
G. Angelini-Knoll, D. Chan, T. Gerhardt, M. Merling and M. P\'{e}roux.
\newblock Topological symmetric and braid homologies.
\newblock  preprint {\tt arxiv:2605.22946}
%
%
\bibitem{Au1}
S.~Ault.
\newblock Symmetric homology of algebras.
\newblock {\em Algebr. Geom. Topol.}, 10(4):2343--2408, 2010.
%
\bibitem{Au2}
S.~Ault.
\newblock Homology operations in symmetric homology.
\newblock {\em Homology Homotopy Appl.}, 16(2):239--261, 2014.
%
\bibitem{AF07}
S.~Ault and Z~Fiedorowicz.
\newblock Symmetric homology of algebras.
\newblock preprint {\tt arXiv:0708.1575}.
%
\bibitem{Ba02}
B. Badzioch, \textit{Algebraic theories in homotopy theory},
Ann. Math. \textbf{155} (2002), 895--913.
%
%
\bibitem{BFPRW}
Yu. Berest, G.~Felder, S.~Patotski, A.~C. Ramadoss, and T.~Willwacher.
\newblock Representation homology, {L}ie algebra cohomology and the derived
  {H}arish-{C}handra homomorphism.
\newblock {\em J. Eur. Math. Soc. (JEMS)}, 19(9):2811--2893, 2017.
%
\bibitem{BFPRW2}
Yu. Berest, G.~Felder, S.~Patotski, A.~C. Ramadoss, and T.~Willwacher.
\newblock Chern-Simons forms and higher character maps of Lie representations.
\newblock {\em Int. Math. Res. Not. (IMRN)}, 217(1):158--212.
%
\bibitem{BFR}
Yu. Berest, G.~Felder, and A.~Ramadoss.
\newblock Derived representation schemes and noncommutative geometry.
\newblock In {\em Expository lectures on representation theory}, volume 607 of
  {\em Contemp. Math.}, pages 113--162. A. M. S., Providence, RI, 2014.
%
\bibitem{BKR}
Yu. Berest, G.~Khachatryan, and A.~Ramadoss.
\newblock Derived representation schemes and cyclic homology.
\newblock {\em Adv. Math.}, 245:625--689, 2013.
%
\bibitem{BR}
Yu. Berest and A.~Ramadoss.
\newblock Stable representation homology and {K}oszul duality.
\newblock {\em J. Reine Angew. Math.}, 715:143--187, 2016.
%
\bibitem{BR14}
Yu. Berest and A.~C. Ramadoss.
\newblock Derived character varieties,
\newblock {\em preprint}, March 2014.
%
\bibitem{BR22c}
Yu. Berest and A.~C. Ramadoss.
\newblock Symmetric homology and representation homology.
\newblock {\em Trans. Amer. Math. Soc.}, 376(9): 6475--6496, 2023.
%
\bibitem{BR22}
Yu. Berest and A.~C. Ramadoss.
\newblock Derived character maps of group representations.
\newblock {\em Algebr. Geom. Topol.}, 24(9): 4991--5044, 2024.
%
%
\bibitem{BRYI}
Yu. Berest, A.~C. Ramadoss, and W.-K. Yeung.
\newblock Representation homology of topological spaces.
\newblock {\em Int. Math. Res. Not. IMRN}, (6):4093--4180, 2022.
%
\bibitem{BRYII}
Yu. Berest, A.~C. Ramadoss, and W.-K. Yeung.
\newblock Representation homology of simply connected spaces.
\newblock {\em J. Topol.}, 15(2):692--744, 2022.
%
\bibitem{B12}
C. Braun,
\newblock Moduli Spaces of Klein surfaces and related operads.
\newblock {\em Algebr. Geom. Topol.}, 12(3):1831-1899, 2012.
%
\bibitem{B14}
C. Braun,
\newblock {Involutive $A_{\infty}$-algebras and dihedral cohomology}
\newblock {\em Journal of Homotopy and Related Structures}, 9(2):317--337, 2014.
%
\bibitem{BH}
G. W. Brumfiel and H. M. Hilden,
\newblock ${\rm SL}_2$-representations of finitely presented groups.
\newblock {\em Contemp. Math.}, 187:viii+196pp, 1995.
%
%
  North-Holland, Amsterdam, 1995.
%
%
\bibitem{Co79}
P.~M. Cohn.
\newblock The affine scheme of a general ring.
\newblock {\em Lecture Notes in Math.}, 753:197--211, 1979.
%
\bibitem{C83}
A.~Connes.
\newblock Cohomologie cyclique et foncteurs {${\rm Ext}^n$}.
\newblock {\em C. R. Acad. Sci. Paris S\'{e}r. I Math.}, 296(23):953--958,
  1983.
%
%
%
%
\bibitem{DS} W. G. Dwyer and J. Spalinski,
{\it Homotopy theories and model categories} in \textit{Handbook of Algebraic Topology},
Elsevier, 1995, pp. 73--126.
%
%
\bibitem{FT85}
B.~L. Fe\u{\i}gin and B.~L. Tsygan.
\newblock Additive $K$-theory and crystalline cohomology.
\newblock {\em Funct. Anal. Appl.}, 19(2):124--132, 1985.
%
\bibitem{FT87}
B.~L. Fe\u{\i}gin and B.~L. Tsygan.
\newblock Additive {$K$}-theory.
\newblock {\em Lecture Notes in Math.}, 1289:67--209, 1987.
%
\bibitem{FHT}
Y.~Felix, S.~Halperin and J.-C.~Thomas, {\it Differential graded algebras in topology} in
{\it Handbook of Algebraic Topology}, Elsevier, 1995, pp. 829--865.
%
\bibitem{F}
Z.~Fiedorowicz.
\newblock Symmetric bar construction, (1991).
\newblock url{http://www.math.ohio-state.edu/~fiedorow/symbar.ps.gz}.
%
\bibitem{FL91}
Z.~Fiedorowicz and J.-L. Loday.
\newblock Crossed simplicial groups and their associated homology.
\newblock {\em Trans. Amer. Math. Soc.}, 326(1):57--87, 1991.
%
%
%
\bibitem{G21} D.~Graves.
\newblock Hyperoctahedral homology of involutive algebras.
\newblock {\em Homology Homotopy Appl.} 24(1):1--26, 2022.
%
\bibitem{G22} D.~Graves.
\newblock $E_{\infty}$ Structure in Hyperoctahedral homology.
\newblock {\em Homology Homotopy Appl.} 25(1):1--19, 2023.
%
\bibitem{Go} Y.~Gouda.
\newblock The relative dihedral homology of involutive algebras.
\newblock {\em Int. J. Math. and Math. Sci.} 22(4):807--815, 1999.
%
\bibitem{Hat02}
A. Hatcher, \textit{Algebraic Topology}, Cambridge Univ. Press, 2002.
%
\bibitem{H}
V. Hinich.
\newblock Homological algebra of homotopy algebras.
\newblock {\em Comm. Algebra} 25(10):3291--3323, 1997.
%
 %
%
%
\bibitem{Ke}
B.~Keller, \textit{A-infinity algebras, modules and functor categories}, Contemp. Math. \textbf{406} (2006), 67--93.
%
\bibitem{Kr87}
R.~Krasauskas,
\newblock Skew-simplicial groups.
\newblock {\em Litovsk. Mat. Sb.}, 27(1):89--99, 1987.
%
\bibitem{Ku03}
N. Kuhn, \textit{
The McCord model for the tensor product of a space and a commutative ring spectrum}, Progr. Math. \textbf{215}, Birkh\"auser, Basel, 2004, pp.
213–236.
%
\bibitem{Le}
K.~Lef\`{e}vre-Hasagawa, \textit{Sur les $A_{\infty}$ cat\'{e}gories}, Th\`{e}se de doctorat, Universit\'{e} Denis Diderot-Paris 7, November 2003.
%
\bibitem{L}
J.-L. Loday.
\newblock {\em Cyclic homology}.
\newblock Springer-Verlag, Berlin, second edition, 1998.
%
\bibitem{LV}
J.-L. Loday and B.~Vallette.
\newblock {\em Algebraic operads}.
\newblock Springer, Heidelberg, 2012.
%
\bibitem{Lo}
G. M. Lodder, {\it Dihedral homology and the free loop space}, Proc. London Math. Soc. (3), \textbf{60} (1990), no. 1, 201-224.
%
%
\bibitem{McC} J.~McCleary, {\it A User's Guide to Spectral Sequences},
Cambridge University Press, 2001.
%
\bibitem{McC69}
M.~C. McCord.
\newblock Classifying spaces and infinite symmetric products.
\newblock {\em Trans. Amer. Math. Soc.}, 146:273--298, 1969.
%
\bibitem{PS}
J. Przytycki and A. Sikora,
\textit{On skein algebras and ${\rm SL}_2(\mathbb C)$-character varieties}, Topology \textbf{39} (2000), no. 1, 115--148.
%
%

%
%
\bibitem{Q}
D.~Quillen.
\newblock Rational homotopy theory.
\newblock {\em Ann. of Math. (2)}, 90:205--295, 1969.
%
\bibitem{Q1}
D.~G. Quillen.
\newblock {\em Homotopical algebra}.
\newblock Lecture Notes in Mathematics, No. 43. Springer-Verlag, Berlin-New
  York, 1967.
%
\bibitem{R20}
B.~Richter.
\newblock {\em From Categories to Homotopy Theory}.
\newblock Cambridge University Press, Cambridge, 2020.
%
\bibitem{SV}
A.~Solotar and M.~Vigu\'{e}-Poirrier.
\newblock Dihedral homology of commutative algebras.
\newblock {\em J. Pure Appl. Algebra}, 109(1):97--106, 1996.
%
\bibitem{T83}
B.~L. Tsygan.
\newblock Homology of matrix {L}ie algebras over rings and the {H}ochschild
  homology.
\newblock {\em Uspekhi Mat. Nauk}, 38(2(230)):217--218, 1983.
%
%
%
\end{thebibliography}
\end{document}